\documentclass[reqno,11pt]{amsart}

\usepackage{amssymb,amsmath,amsthm,amsfonts,color} 

\usepackage{etoolbox}
\usepackage[toc,page]{appendix}

\usepackage{mathrsfs,dsfont, comment,mathscinet}

\usepackage{enumerate,esint,accents}
\usepackage[left=4.0cm,right=4.0cm,top=2.0cm,bottom=2.0cm]{geometry}

\usepackage{mathtools}
\usepackage{multirow}
\usepackage{graphicx}
\usepackage{enumitem}
\usepackage{setspace}

\usepackage[colorlinks=true, pdfstartview=FitV, linkcolor=blue, 
            citecolor=blue, urlcolor=blue]{hyperref}

\allowdisplaybreaks
\numberwithin{equation}{section}

\theoremstyle{plain}
\newtheorem{theorem}{Theorem}[section]
\newtheorem{proposition}[theorem]{Proposition}
\newtheorem{lemma}[theorem]{Lemma}
\newtheorem{corollary}[theorem]{Corollary}

\theoremstyle{definition}
\newtheorem{definition}[theorem]{Definition}
\newtheorem{remark}[theorem]{Remark}

\usepackage{amssymb,amsmath}

\newcommand{\C}{\mathbb{C}}
\newcommand{\R}{\mathbb{R}}
\newcommand{\N}{\mathbb{N}}
\renewcommand{\S}{\mathbb{S}}

\newcommand\EEE{\color{black}}
\newcommand\BBB{\color{black}}

\newcommand{\black}{\color{black}}
\newcommand{\red}{\color{black}}

\newcommand{\fA}{\mathcal{A}}
\newcommand{\fm}{\mathtt{v}}

\newcommand{\cH}{\mathcal{H}}
\newcommand{\cE}{\mathcal{E}}
\newcommand{\cK}{\mathcal{K}}
\newcommand{\cS}{\mathcal{S}}
\newcommand{\cF}{\mathcal{F}}
\newcommand{\cW}{\mathcal{W}}
\newcommand{\cJ}{\mathcal{J}}

\newcommand{\p}{\partial}
\newcommand{\wk}{\rightharpoonup}
\newcommand{\cl}[1]{\overline{#1}}

\newcommand{\strictlyincluded}{\subset\subset}
\newcommand{\res}{\mathop{\hbox{\vrule height 7pt width 0.5pt depth 0pt
\vrule height 0.5pt width 3pt depth 0pt}}\nolimits}

\renewcommand{\tilde}{\widetilde}

\newcommand{\dist}{\mathrm{dist}}
\newcommand{\sdist}{\mathrm{sdist}}

\newcommand{\loc}{\mathrm{loc}}
\newcommand{\Int}[1]{\mathrm{Int}(#1)}

\renewcommand{\hat}{\widehat}
\newcommand{\co}[1]{{co}(#1)}
\newcommand{\tr}{\mathrm{tr}}

\newcommand{\openset}{P}
\newcommand{\anyset}{Q}

\newcommand{\no}{\nonumber}
\newcommand{\str}[1]{e(#1)}

\newcommand{\mtwo}{\mathbb{M}^{2\times 2}_{\rm sym}}
\newcommand{\admissible}{\mathcal{C}}
\newcommand{\substrate}{S}
\newcommand{\Ins}[1]{{\rm Int}{(#1\cup \substrate\cup \Sigma)}}

\title[A unified model for SDRI]{A unified model for stress-driven
rearrangement instabilities}

\author[Sh. Kholmatov] {Shokhrukh Yu. Kholmatov} 
\address[Shokhrukh Yu. Kholmatov]{Fakult\"at f\"ur Mathematik\\ 
Universit\"at Wien\\ Oskar-Morgenstern Platz 1\\1090 Wien 
(Austria)}
\email[Sh. Kholmatov]{shokhrukh.kholmatov@univie.ac.at}

\author[P. Piovano] {Paolo Piovano} 
\address[Paolo Piovano]{Fakult\"at f\"ur Mathematik\\ 
Universit\"at Wien\\ Oskar-Morgenstern Platz 1\\1090 Wien 
(Austria)}
\email[P. Piovano]{paolo.piovano@univie.ac.at}



\date{\today}
\begin{document} 

\begin{abstract}
A variational model to simultaneously treat Stress-Driven Rearrangement
Instabilities,  such as boundary  discontinuities, internal cracks, external filaments, edge delamination, wetting, and brittle fractures, is  introduced. The model is characterized by an energy displaying both elastic and surface terms, and allows for a unified treatment of a wide range of settings, from  epitaxially-strained thin films to crystalline cavities, and from capillarity problems to fracture models.  
 
Existence of minimizing configurations is established by adopting the direct method of the Calculus of Variations. Compactness of energy-equibounded sequences and energy lower semicontinuity are shown with respect to a proper selected topology in a class of admissible configurations that extends  the classes previously considered in the literature. In particular, graph-like constraints previously considered for the setting of thin films and crystalline cavities are substituted by the more general assumption that the free crystalline interface is the boundary, consisting of an at most fixed finite number $m$ of
connected components, of sets of finite perimeter.  
 
Finally, it is shown that, as $m\to\infty$, the energy of minimal admissible configurations  tends to the minimum energy in the general class of configurations without the bound on the number of connected components for the free interface.  

\end{abstract}
\maketitle

\tableofcontents

\newpage 

\section{Introduction}
 
Morphological destabilizations of crystalline interfaces are often referred to as Stress-Driven Rearrangement Instabilities (SDRI) from the seminal paper \cite{Gr:1993} (see also Asaro-Grinfeld-Tiller instabilities \cite{AT:1972,D:2001}). SDRI consist in various mechanisms of mass rearrangements that take place at crystalline boundaries  
because of the strong stresses originated by the mismatch between the parameters of adjacent crystalline lattices. Atoms move from their crystalline order and different modes of stress relief may  co-occur, 
such as deformations of the bulk materials with storage of \emph{elastic energy}, and boundary instabilities that contribute to the \emph{surface energy}.

In this paper we introduce a variational model displaying both  elastic and surface energy that simultaneously takes into account the various possible SDRI, such as \emph{boundary discontinuities}, internal \emph{cracks}, external \emph{filaments}, \emph{wetting} and \emph{edge delamination} with respect to a substrate, and \emph{brittle fractures}.   In particular, the model provides a unified mathematical treatment of 
epitaxially-strained thin films \cite{DP:2018_1, FFLM:2007, FG:2004,   KP:2019, S:1999}, crystal cavities \cite{FFLM:2011, SMV:2004,WL:2003}, 
capillary droplets \cite{CM:2007,GB-W:2004,DPhM:2015},  as well as
Griffith and failure models \cite{Bourdin:2008,CCI:2017,ChC:2017arma,Gr:1921,Xia:2000}, which were previously treated separately in the literature. Furthermore, the possibility of delamination and debonding, i.e.,  crack-like modes of interface failure at the interface with the substrate \cite{Deng:1995,HS:1991}, is treated in accordance with the models in \cite{Babadjian:2016,Baldelli:2014,Baldelli:2013}, that were introduced by revisiting  in the variational perspective of 
fracture mechanics the model first described in \cite{Xia:2000}. 
Notice that as a consequence the surface energy depends on 
the admissible deformations and cannot be decoupled from the elastic energy.
As a byproduct of our analysis, we extend previous results 
for the existence of minimal configurations to
anisotropic surface and elastic energies, and   
we relax constraints previously  
assumed on admissible configurations 
in the thin-film  and crystal-cavity settings.
For thin films we avoid  the reduction considered in 
\cite{DP:2018_1,DP:2018_2,FFLM:2007} to only  
film profiles parametrizable by  thickness 
functions, and for crystal cavities 
the restriction in \cite{FFLM:2011} to cavity sets consisting 
of only one connected starshaped void.  
 
The class  of   interfaces that we consider 
is given by all the boundaries, that consists of 
connected components whose number is arbitrarily 
large but not  exceeding a  fixed number $m$,  
of  sets of finite perimeter $A$. We refer to 
the class of sets of finite perimeter associated 
to the free interfaces as \emph{free crystals} 
and we notice that  free crystals  $A$ may present 
an infinite number of components. The assumption on 
the number of components for the boundaries of 
free crystals  is needed to 
apply an adaptation to our setting 
of the generalization of Golab's Theorem 
proven in \cite{Gi:2002} that allows to
establish  in dimension 2, to which we restrict, compactness 
with respect to a proper selected  topology.  
To the best of our knowledge presently no variational 
framework able to guarantee the existence of 
minimizers  in  dimension 3  in the settings of  thin films and 
crystal cavities is available  in the literature.

Furthermore, also the class of admissible 
deformations is enlarged with respect to 
\cite{DP:2018_1,DP:2018_2,FFLM:2011,FFLM:2007} 
to allow debonding and edge delamination to occur  along the 
\emph{contact surface} $\Sigma:=\partial S\cap \partial \Omega$  
between the fixed substrate $S$ and the fixed bounded region $\Omega$ 
containing the admissible free crystals  (see Figure 
\ref{free_crystal}). In the following we refer to $\Omega$ as the \emph{container} 
in analogy with capillarity problems. 
Notice that  the obtained 
results can be easily applied also for unbounded 
containers in the setting of thin films
with the graph constraint (see Subsection \ref{subsec:applications}). 
Mathematically this is modeled by considering admissible 
deformations $u$ that are Sobolev functions only in the interior 
of the free crystals $A$ and the substrate $S$, and $GSBD$, i.e., 
generalized special functions of bounded deformation
(see \cite{D:2013} for more details), on $A\cup S\cup \Sigma$. 
Thus, jumps $J_u$ that represent edge delamination can develop 
at the contact surface $\Sigma$, i.e., 
$J_u\subset\Sigma$.  

\begin{figure}[htp] 
\begin{center}
\includegraphics[scale=0.55]{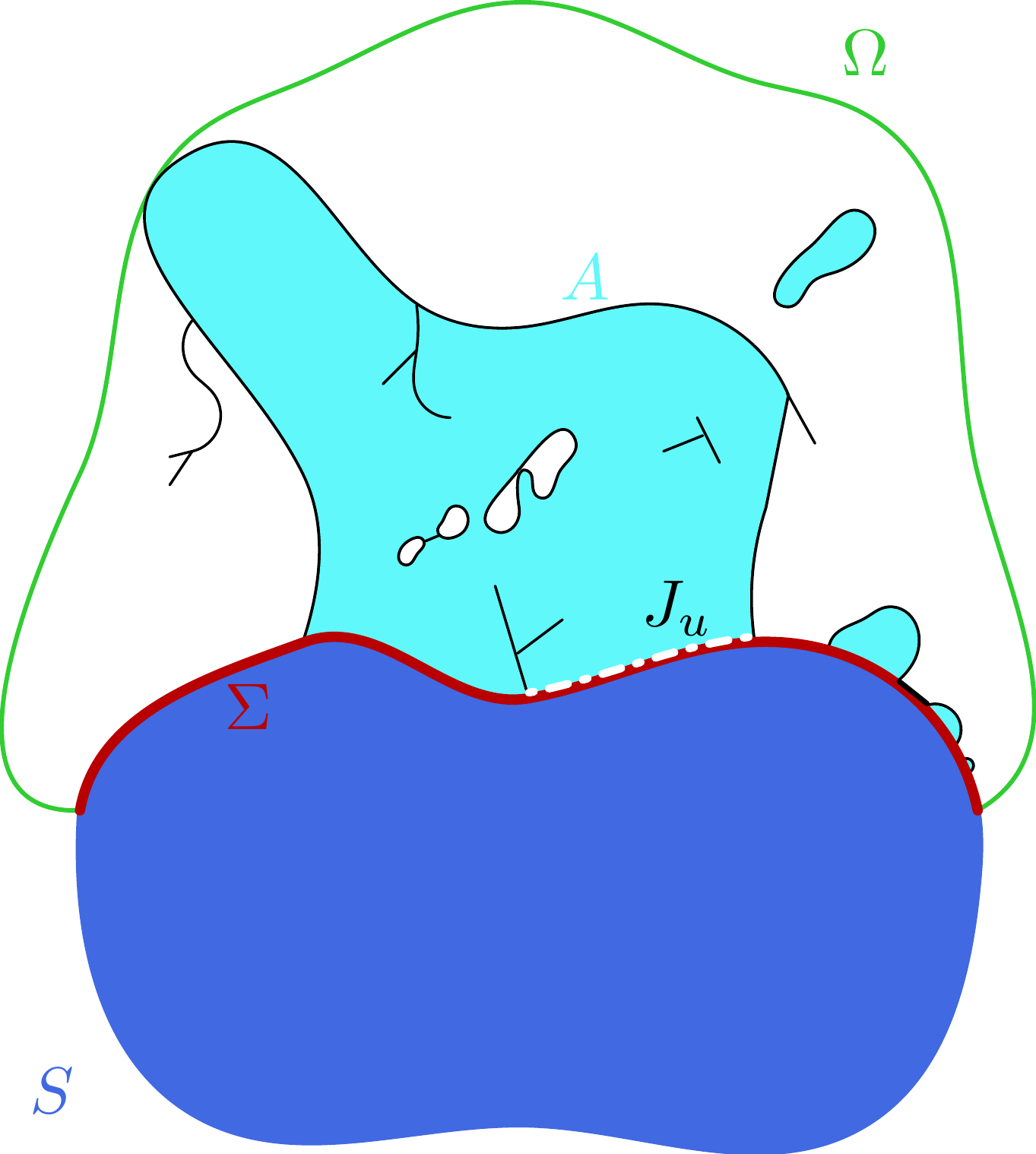} 
\caption{\BBB An admissible free (disconnected)
crystal $A$ is displayed in light blue 
in the container $\Omega$, while the substrate $S$ is represented
 in dark blue. The boundary of $A$
(with the cracks) is depicted in black, the container boundary in 
 green, the contact surface
 $\Sigma$ in red (thicker line) while the delamination 
 region $J_u$ with a white dashed line. }
\label{free_crystal}
\end{center}
\end{figure}

The energy $\cF$  that characterizes our model  is defined for 
every admissible configuration $(A,u)$ in the configurational space
$\mathcal{C}_m$ of free crystals and deformations by 
\begin{equation*}
 \cF(A,u):=\cS(A,u) +  \mathcal{W}(A,u),
\end{equation*}
where $\cS$ denotes the \emph{surface energy} and $\mathcal{W}$ the 
\emph{bulk elastic energy}.
The bulk elastic energy is given  by
$$
\mathcal{W}(A,u)=\int_{A\cup S} W(z,\BBB\str{u}-E_0)\,d z
$$
for an elastic  density $W(z,M):=\mathbb{C}(z) M:M$ 
defined with respect to a positive-definite 
elasticity tensor $\mathbb{C}$ and a 
{\it mismatch strain} $E_0$. The mismatch strain is introduced to 
represent the fact that the lattice of the  free crystal 
generally  does not match the substrate lattice.
We notice that the tensor $\mathbb{C}$ is assumed to 
be only $L^\infty(\Omega\cup S)$, therefore not 
only allowing for different elastic 
properties between the material of 
the free crystals in $\Omega$ and the one of 
the substrate, but also for non-constant 
properties in each material extending previous results.  The surface 
energy $\cS$ is defined as 
  $$
\cS(A,u)= \int_{\partial A} \psi(z,u,\nu)\,d\mathcal{H}^{d-1}
  $$
with surface tension $\psi$ defined by
 \begin{equation}\label{surface_tension}
 \psi(z,u,\nu):=\begin{cases}
 \varphi(z,\nu_A(z)) & z\in\Omega\cap\partial^*A,\\
2 \varphi(z,\nu_A(z)) & z\in \Omega\cap (A^{(1)}\cup A^{(0)})\cap\partial A,\\
\varphi(z,\nu_S(z)) + \beta(z) & \Sigma\cap A^{(0)}\cap \partial A,\\
\beta(z) & z\in\Sigma\cap \partial^*A\setminus J_u,\\
 \varphi(z,\nu_S(z)) &  J_u,
\end{cases}
\end{equation}
where $\varphi\in C(\cl{\Omega}\times\R^d;[0,+\infty))$ 
is a Finsler norm representing the
\emph{material anisotropy} with 
$c_1|\xi| \le \varphi(x,\xi) \le c_2|\xi|$ 
for some $c_1,c_2>0$, $\beta\in L^\infty(\Sigma)$ 
is the \emph{relative adhesion
coefficient} on $\Sigma$ with 
\begin{equation}\label{beta_restrictions}
|\beta(z)|\leq\varphi(z,\nu_S(z))
\end{equation}
for $z\in\Sigma$,  $\nu$ is the exterior 
normal on the reduced boundary $\partial^*A$, 
and $A^{(\delta)}$ denotes the 
set of points of $A$ with density $\delta\in[0,1]$. Notice that the 
anisotropy $\varphi$ is  counted double on the sets 
$A^{(1)}\cap\partial A\cap\Omega$ 
and $A^{(0)}\cap\partial A\cap\Omega$,  that represent 
the set of cracks and  the 
set of external filaments, respectively. 
 On the free profile
$\partial^*A$ the anisotropy is weighted the same as on 
the delamination region $J_u$, since delamination involves  
debonding between the adjacent materials by definition.  Furthermore,  the adhesion coefficient $\beta$ is considered  on the contact surface $\Sigma$, alone on the reduced boundary $\Sigma\cap \partial^*A\setminus J_u$ and together with $\varphi$ on those external filaments  $A^{(0)}\cap\partial A\cap\Sigma$, to which we refer as \emph{wetting layer}.
 
We refer the Reader to Subsection \ref{subsec:main_results} for the rigorous mathematical setting and the main results of the paper, among which we recall here the following existence result:  

\noindent 
{\textbf{Main Theorem.}}
{\it If $\fm\in (0,|\Omega |)$ or $\substrate=\emptyset,$  
then for every $m\ge1$ the volume-constrained minimum problem 
$$
\inf\limits_{(A,u)\in \admissible_m,\,\,|A| = \fm} \cF(A,u)   
$$
admits a solution and 
\begin{equation}\label{intro_zur_tenglik}
\inf\limits_{(A,u)\in \admissible,\,\,|A|=\fm} \cF(A,u) 
= \lim\limits_{m\to\infty} \inf\limits_{(A,u)\in \admissible_m,
\,\,|A|=\fm} \cF(A,u).
\end{equation} 
$ $}  
\noindent 
This existence result is accomplished in Theorem \ref{teo:existence}, 
where we also solve the related unconstraint problem with energy 
$\cF^\lambda$ given by $\cF$ plus a \emph{volume penalization} depending on the parameter $\lambda>0.$ 

The proof is based on the  \emph{direct method} of the Calculus of Variations, i.e., it consists in determining a  suitable topology $\tau_\admissible$ in $\admissible_m$ sufficiently weak  to establish the compactness of  energy-equibounded sequences in Theorem
\ref{teo:compactness_Ym} and strong enough to prove that the energy is lower semicontinuous in Theorem \ref{teo:lsemico_film}. We notice here, that Theorem \ref{teo:compactness_Ym} and Theorem \ref{teo:lsemico_film}  can also be  seen as an extension, under the condition on the maximum 
admissible number $m$ of connected components  for the boundary,  of the compactness and lower semicontinuity   results in   \cite{ChC:2018jems}  to anisotropic surface tensions   and to the other SDRI settings.

The topology $\tau_{\admissible}$ selected in $\admissible$ corresponds,  under the uniform bound on the length  of the free-crystal  boundaries,   to the convergence of both the free crystals and the free-crystal complementary sets  with respect to the Kuratowski convergence and to the pointwise  convergence of the displacements.  In \cite{DP:2018_1,DP:2018_2,FFLM:2007} 
the weaker convergence $\tau_{\admissible}'$ 
consisting of only the Kuratowski
convergence of  complementary sets of free-crystals (together 
with the $\overline{S}$) was 
considered, which in our setting without 
graph-like assumptions on the free boundary is 
not enough because not closed in $\admissible_m$. 
Working with the topology $\tau_{\admissible}$ 
also allows  to maintain track in the surface energy of the possible 
external filaments of the admissible free crystals, 
which were in previous 
results not considered. However, to establish 
compactness with respect to 
$\tau_{\admissible}$ the Blaschke Selection Theorem 
employed in \cite{DP:2018_1,DP:2018_2,FFLM:2011,FFLM:2007} 
is not enough, and a version 
for the signed distance functions from the free 
boundaries is obtained   
(see Proposition \ref{prop:compactness_wrt_sdist}).   
Furthermore, 
in order to take in consideration the situation in which connected components  of $A_k$ separates in the limit in multiple connected components of $A$, e.g., in the case of neckpinches, we need to introduce extra boundary in $A_k$ in order to divide their components accordingly (see Proposition \ref{prop:sets_changhe}). Otherwise, adding to $u_k$ different rigid displacements with respect to the components in $A$ (which are needed for compactness of $u_k$) would results in jumps for the displacements in $A_k$, which are not allowed in our setting with $H^1_{\rm loc}$-displacements. Therefore, we pass from the sequence $A_k$ to a sequence $D_k$ with such extra boundary for which we can prove compactness. Passing to $D_k$ is not a problem in the existence in view of property   \eqref{liminfga_ut_eshmat} that relates the $liminf$ of the energy with respect to $A_k$ to the one with respect to $D_k$. However, in case $\substrate\ne\emptyset,$ in order to prove \eqref{liminfga_ut_eshmat}, we need to further modify the sequence $D_k$ from the original $A_k$ by cutting out the portion  converging to delamination regions (e.g., portion containing accumulating cracks and voids at the boundary with $S$) using Proposition \ref{prop:jump_estimate}, and, in order to maintain the volume constraint, by replacing them with an extra set that does not contribute to the overall elastic energy. 
\black

The lower semicontinuity of the energy 
with respect to $\tau_{\admissible}$  is established 
for the elastic energy as in \cite{FFLM:2007} 
by convexity, and  for the surface energy
 in Proposition \ref{prop:lsc_surface_energy} in several steps by 
adopting a blow-up method (see, e.g., \cite{ADT:2017,BFM:1998}).
More precisely, 
given a sequence of configurations
$(A_k,u_k)\in\mathcal{C}_m$ converging to 
$(A,u)\in\mathcal{C}_m$ we consider a converging 
subsequence of the Radon  measures $\mu_k$ 
associated to the surface energy and $(A_k,u_k)$, and 
we estimate from below
the Radon-Nikodym derivative of their limit denoted by $\mu_0$ with respect 
to the Hausdorff measure restricted to the 5 portions of 
$\partial A$ that appear 
in the definition of the surface anisotropy 
$\psi$ in \eqref{surface_tension}. 
We overcome the fact that in general  $\mu_0$ is 
not a non-negative measure 
due to the presence  of the contact  term in the energy 
with $\beta$, by  adding 
to $\mu_k$ and  $\mu_0$ the positive measure
$$
\mu_\Sigma(B) = \int_{B\cap \Sigma} 
\varphi(x,\nu_\Sigma(x))d\cH^1
$$
defined for every Borel set $B\subset\mathbb{R}^2$ 
and using \eqref{beta_restrictions}.
The estimates for the Radon-Nikodym 
derivative related to the free boundary  
$\Omega\cap \partial^*A$ and the contact region
$(\Sigma\cap \partial^*A)\setminus J_{u_k}$ follow 
from \cite[Lemma 3.8]{ADT:2017}.  
For  the estimates related to  
exterior filaments and interior cracks we first 
separately  reduce to the case of  flat 
filaments and cracks, and then we adapt 
some  arguments from \cite{Gi:2002}. Extra care is needed to treat  the exterior 
filament lying on $\Sigma$ to which we refer as wetting layer in
analogy to the thin-film setting. The 
estimate related to the delamination region on 
$\Sigma$  follows by blow-up  under condition
\eqref{condition_J_A} that ensures that the
delamination regions between the limiting free crystal $A$ 
and the substrate $\substrate$ can be originated from 
delamination regions between $A_k$  
and $\substrate$ and  from portions of free 
boundaries $\partial^*A_k$ or interior cracks collapsing on $\Sigma$, 
as well as from accumulation of interior cracks 
starting from $(\Sigma\cap\partial^*A_k)\setminus J_{u_k}$.  

{\red }
A challenging point is to prove 
that  condition  \eqref{condition_J_A} is satisfied 
by $(A_k,u_k)$. 
In order to do that,  in Theorem \ref{teo:lsemico_film} 
we   first extend 
the displacements  $u_k$ to the set $\Omega\setminus (A_k\cup \substrate)$ using Lemma \ref{lem:extension_desparate}.  
%
The extension of the $u_k$  
is performed without creating extra jump at the interface on the exposed surface of the substrate, i.e.,  the jump set of the extensions is approximately $J_{u_k}\cup(\Omega\cap \partial A_k)$. We point out that as a consequence we obtain also in Proposition \ref{prop:lsemico_film_voids} the lower semicontinuity, with respect to the topology $\tau_\admissible'$, of a version of our energy without exterior filaments (but with wetting layer) extending the lower semicontinuity results of \cite{DP:2018_1,FFLM:2011,FFLM:2007}.

Finally, we prove  \eqref{intro_zur_tenglik}, that in particular entails the existence of a minimizing  sequence $(A_m,u_m)\in\mathcal{C}_m$ for the minimum problem of $\cF$ in  $\mathcal{C}$. This is obtained by 
considering a minimizing sequence $(A_{\varepsilon},u_{\varepsilon})\in\mathcal{C}$ for $\cF^{\lambda}$, and then by modifying 
it into a new minimizing sequence 
$(E_{\varepsilon,\lambda},v_{\varepsilon,\lambda})\in\mathcal{C}_m$  such that  
$\cF^{\lambda}(A_{\varepsilon},u_{\varepsilon})+\delta_{\varepsilon}\geq
\cF^{\lambda}(E_{\varepsilon,\lambda},v_{\varepsilon,\lambda})$ for 
some $\delta_{\varepsilon}\to0$ as $\varepsilon\to0$.  The construction of $(E_{\varepsilon,\lambda},v_{\varepsilon,\lambda})\in\mathcal{C}_m$
requires 2 steps. In  the first step we eliminate the external filaments, we remove sufficiently small connected components of $A_{\varepsilon}$, and we fill in sufficiently small holes till we reach a finite number of connected components with a finite number of holes  (see Figure \ref{step1existence}). In the second step we redefine
the deformations in the free crystal  by employing \cite[Theorem 1.1]{ChC:2017arma} in order to obtain a deformation with jump set consisting of at most finitely many components, and such that the difference in the elastic energy and the length of the jump sets with respect to $u_{\varepsilon}$ remains small.

The paper is organized as follows. In Section \ref{sec:setting_results} we introduce the model and the topology $\tau_{\mathcal{C}}$, we refer to various SDRI settings from the literature that are included in our analysis, and we state the main results.  In Section \ref{sec:compactness} we prove sequential compactness  for the free crystals with the bound $m$ on the boundary components in Proposition \ref{prop:compactness_A_m} and for $\admissible_m$ in Theorem \ref{teo:compactness_Ym}.  In Section \ref{sec:lsc_results}  
we prove the lower semicontinuity of the energy (Theorem \ref{teo:lsemico_film}) by first considering only the surface energy $\mathcal{S}$ under the condition  \eqref{condition_J_A} (see Proposition  \ref{prop:lsc_surface_energy}), and we conclude the section by showing the lower semicontinuity of the energy without the external filament and wetting-layer terms with respect to the topology 
$\tau_{\admissible}'$ (see Proposition \ref{prop:lsemico_film_voids}). In Section  \ref{sec:min_config} we prove the existence results
(Theorems \ref{teo:existence} and \ref{teo:existence_thin_void}) 
and property \eqref{intro_zur_tenglik}. The paper is concluded with an Appendix where results related to rectifiable sets and Kuratowski convergence are recalled for  Reader's convenience.

\section{Mathematical setting}\label{sec:setting_results}

We  start by introducing some notation. Since our model is two-dimensional, unless otherwise stated, all sets we consider are subsets of $\R^2.$ We choose the standard basis $\{{\bf e_1}=(1,0), {\bf e_2}=(0,1)\}$ in $\R^2$ and denote the coordinates of $x\in\R^2$  with respect to this basis by $(x_1,x_2).$ We denote by $\Int A$ the interior of $A\subset \R^2.$ Given a Lebesgue measurable set $E,$ we denote by $\chi_E$ its characteristic function and by $|E|$ its Lebesgue measure. The set 
$$
E^{(\alpha)}:=\Big\{x\in\R^2:\,\, \lim\limits_{r\to0} 
\frac{|E\cap B_r(x)|}{|B_r(x)|}=\alpha \Big\},
\qquad \alpha\in [0,1],
$$
where $B_r(x)$ denotes the ball in $\R^2$ centered at $x$ of radius $r>0,$ is called the set of points of density $\alpha$ of $E.$ Clearly, $E^{(\alpha)}\subset \p E$ for any $\alpha\in (0,1),$ where 
$$
\p E:=\{x\in \R^2:\,\,\text{$B_r(x)\cap E \ne \emptyset$ and 
$B_r(x)\setminus E \ne \emptyset$ for any $r>0$}\} 
$$
is the topological boundary. The set $E^{(1)}$ is the {\it Lebesgue 
set} of $E$ and $|E^{(1)}\Delta E|=0.$ We denote by $\p^*E$ the {\it reduced} boundary of a finite perimeter set $E$ \cite{AFP:2000,Gi:1984},
i.e.,
\begin{equation}\label{essential_boundary00}
\p^*E:=\Big\{x\in\R^2:\,\, \exists \nu_E(x):=-\lim\limits_{r\to0}
\frac{D\chi_E(B_r(x))}{|D\chi_E|(B_r(x))},\quad |\nu_E(x)|=1\Big\}.  
\end{equation}
The vector $\nu_E(x)$ is called the {\it measure-theoretic} normal to $\p E.$  

The symbol $\cH^s,$ $s\ge0,$ 
stands for the $s$-dimensional Hausdorff measure. 
An $\cH^1$-measurable set $K$ with  $0<\cH^1(K)<\infty$  is called 
$\cH^1$-{\it rectifiable} if 
$\theta^*(K,x)=\theta_*(K,x) =1$ for 
$\cH^1$-a.e.\ $x\in K,$  where 
$$
\theta^*(K,x):=\limsup\limits_{r\to0^+} \frac{\cH^1(B_r(x)\cap K)}{2r},
\qquad 
\theta_*(K,x):=\liminf\limits_{r\to0^+} \frac{\cH^1(B_r(x)\cap K)}{2r}.
$$
By \cite[Theorem 2.3]{Fa:1985} any $\cH^1$-measurable set 
$K$ with  $0<\cH^1(K)<\infty$
satisfies $\theta^*(K,x)=1$ for $\cH^1$-a.e.\ $x\in K.$

\begin{remark} \label{rem:essential_boundary}
If $E$ is a finite perimeter set, then 
\begin{itemize}
\item[(a)]  $\cl{\p ^*E} =\p E^{(1)}$  (see, e.g., \cite[Theorem 4.4]{Gi:1984} and \cite[Eq. 15.3]{Ma:2012});

\item[(b)] $\p^*E \subseteq E^{(1/2)}$ and $\cH^{1}(E^{(1/2)}\setminus \p ^*E) =0$  (see, e.g., \cite[Theorem 3.61]{AFP:2000} and \cite[Theorem 16.2]{Ma:2012}); 

\item[(c)]  $P(E,B) = \cH^{1}(B\cap \p^*E)= \cH^{1}(B\cap E^{(1/2)})$ for any  Borel set $B.$ 
\end{itemize}

\end{remark}

The notation $\dist(\cdot,E)$ stands for 
the distance function from the set $E\subset\R^2$
with the convention that $\dist(\cdot,\emptyset)\equiv+\infty.$
Given a set $A\subset\R^2,$ we consider also signed distance function
from $\p A,$ negative inside, defined as 
$$
\sdist(x,\p A):= 
\begin{cases}
\dist(x, A) &\text{if $x\in \R^2\setminus A,$}\\ 
-\dist(x,\R^2\setminus A) &\text{if $x\in  A.$}  
\end{cases}
$$

\begin{remark}\label{rem:kuratiwski_and_sdistance}
The following assertions are equivalent:
\begin{itemize}  
\item[(a)] $\sdist(x,\p E_k) \to \sdist(x,\p E)$ locally uniformly in $\R^2;$

\item[(b)] $E_k\overset{\cK}\to \cl{E}$ and 
$\R^2\setminus E_k\overset{\cK}\to \R^2\setminus \Int E,$
where $\cK$--Kuratowski convergence of sets \cite[Chapter 4]{D:1993}. 
\end{itemize}
Moreover, either assumption  implies $\p E_k \overset{\cK}{\to} \p E.$
\end{remark}

\subsection{The model} 
Given two open sets $\Omega \subset\R^2$
and $\substrate\subset\R^2\setminus\Omega,$ 
we define the family of admissible regions 
for the {\it free crystal} and 
the space of {\it admissible configurations} by
$$
\fA:=\{A\subset\overline{\Omega} :\,\,\text{$\p A$ is 
$\cH^1$-rectifiable and $\cH^1(\p A)<\infty$}\}
$$
and 
$$
\begin{aligned}
\admissible:=\big\{(A,u):\,\,& A\in \fA,\\
&u\in GSBD^2(\Ins{A};\R^2)
\cap H_\loc^1(\Int{A}\cup \substrate;\R^2)\big\},
\end{aligned}
$$
respectively, where $\Sigma:=\p \substrate\cap \p \Omega$
and 
$GSBD^2(E,\R^2)$ is the collection of all generalized 
special functions of bounded deformation
\cite{ChC:2018jems,D:2013}. Given a displacement 
field
$u\in GSBD^2(\Ins{A};\R^2)\cap H_\loc^1(\Int{A}\cup \substrate;\R^2)$
we denote by  
$\str{u(\cdot)}$ the density of  ${\bf e}(u)=(Du+(Du)^T)/2$ with respect to 
Lebesgue measure $\mathcal{L}^2$ and by $J_u$
the jump set of $u.$ Recall that 
$\str{u}\in L^2(A\cup \substrate)$ and $J_u$ is 
$\cH^1$-rectifiable. Notice  also 
that assumption $u\in H_\loc^1(\Int{A}\cup \substrate;\R^2)$ implies  
$J_u\subset\Sigma\cap \overline{\p^* A}.$  We denote the boundary trace of a function $u:A\to\R^n$ by $\tr_A$ (if exists).

\begin{remark}\label{rem:structure_of_pA} 
For any  $A\in \fA:$  
\begin{itemize}
\item[(a)]
$
\p A = N\cup (\Omega \cap \p^*A) \cup (\p\Omega \cap \p A) \cup 
(\Omega \cap A^{(0)}\cap \p A)\cup (\Omega \cap A^{(1)}\cap \p A),
$ 
where $N$ is an $\cH^1$-negligible set (see, e.g., \cite[page 184]{Ma:2012});\\

\item[(b)] $\Omega \cap \cl{\p^*A} = \Omega \cap \p A^{(1)}$ (see Remark \ref{rem:essential_boundary} (a) above);\\

\item[(c)] $\cH^1(\cl{\p^*A}\setminus \p^*A) = 0$ (since $\overline{\p^*A}\subset\p A$ is $\cH^1$-rectifiable);\\

\item[(d)] up to a $\cH^1$-negligible set, the trace  of $A\in\fA$ on $\p\Omega $ is defined as $\p\Omega \cap\p^*A$ (see, e.g., \cite[Lemma 2.10]{ADT:2017}).
\end{itemize}
\end{remark}
\black

Unless otherwise stated, in what follows $\Omega $ 
and $\substrate\subset\R^2\setminus \Omega $ are bounded Lipschitz 
open sets with finitely many connected components 
satisfying $\cH^1(\p \substrate)+\cH^1(\p\Omega)<\infty$ 
and $\Sigma\subseteq \p\Omega $ is a Lipschitz 1-manifold. 
 
We introduce in $\fA$ the following notion of convergence.

\begin{definition}[\textbf{$\tau_{\fA}$-Convergence}]
\label{def:film_convergence}
A sequence $\{A_k\}\subset\fA$ is said to $\tau_{\fA}$-converge
to $A\subset\R^2$  and  is written $A_k\overset{\tau_\fA}{\to} A$ if
\begin{itemize}
\item[--] $\sup\limits_{k\ge1} \cH^1(\p A_k) < \infty;$

\item[--] $\sdist(\cdot,\p A_k)\to \sdist(\cdot,\p A)$ 
locally uniformly in $\R^2$ as $k\to \infty.$
\end{itemize} 
\end{definition}

We endow $\admissible$ with the following notion of convergence.

\begin{definition}[\textbf{$\tau_{\admissible}$-Convergence}]
\label{def:admissible_convergence}
A sequence $\{(A_n,u_n)\}\subset \admissible$ is said to
$\tau_\admissible$-converge to $(A,u)\in \admissible$, and is written 
$(A_n,u_n)\overset{\tau_{\admissible}}{\to} (A,u)$  if

\begin{enumerate}
\item[--] $A_n\overset{\tau_{\fA}}{\to} A,$ 

\item[--]  $u_n\to u$  a.e.\ in\footnote{If $\R^2\setminus A_k \overset{K}{\to}\R\setminus \Int{A},$ then for any $x\in \Int{A}$  one has $x\in \Int{A_n}$ for all large $n.$}  $\Int{A}\cup\substrate.$ \black
\end{enumerate}

\end{definition}

The {\it energy} of admissible configurations is given by 
the functional $\cF:\admissible\to[-\infty,+\infty],$ 
\begin{equation*}
\cF:=\cS + \cW, 
\end{equation*}
where $\cS$ and $\cW$ are the surface and elastic energies of 
the configuration, respectively.
The surface energy of $(A,u)\in\admissible$ is defined as 
\begin{align}\label{func_surface_energy}
\cS(A,u):=& \int_{\Omega \cap\p^*A} \varphi(x,\nu_A(x))d\cH^1(x) \nonumber \\
&+\int_{\Omega \cap (A^{(1)}\cup A^{(0)})\cap\p A}
\big(\varphi(x,\nu_A(x)) + \varphi(x,-\nu_A(x))\big)d\cH^1(x)\nonumber\\  
& + \int_{\Sigma\cap A^{(0)}\cap \p A} \big(\varphi(x,\nu_\Sigma(x))
+ \beta(x)\big)d\cH^1(x) \nonumber \\
& + \int_{\Sigma\cap \p^*A\setminus J_u} \beta(x) d\cH^1(x) 
 + \int_{J_u} \varphi(x,-\nu_\Sigma(x))\,d\cH^1(x),
\end{align}
where
$\varphi:\overline\Omega\times\S^1\to[0,+\infty)$ and 
$\beta:\Sigma\to\R$  are  Borel 
functions  denoting the {\it anisotropy} \BBB of crystal \EEE and
the {\it relative adhesion} coefficient 
of the substrate, respectively, and  $\nu_\Sigma:=\nu_\substrate$.
In the following we refer to the first term in \eqref{func_surface_energy} as 
the {\it free-boundary energy}, to the second as the {\it energy of
internal cracks and external filaments},
to the third as the {\it wetting-layer energy}, 
to the fourth as the {\it contact energy},
and to the last as the {\it delamination energy}.
In applications instead of $\varphi(x,\cdot)$ it is more convenient to use 
its positively one-homogeneous extension $|\xi|\varphi(x,\xi/|\xi|).$
With a slight abuse of notation we denote this extension 
also by $\varphi.$

The elastic energy of $(A,u)\in\admissible$ is defined as
$$
{\mathcal W}(A,u):= \int_{A\cup  \substrate } W(x,\str{u(x)}  - E_0(x))dx,
$$
where the elastic density $W$ is determined
as the quadratic form 
\begin{equation}\label{asjkladkl}
W(x,M): =  \C(x)M:M, 
\end{equation}
by the so-called {\it stress-tensor}, a measurable function
$x\in\Omega\cup\substrate\to\C(x),$ where 
$\C(x)$ is a non-negative 
fourth-order tensor in the Hilbert space 
$\mtwo$ of all $2\times 2$-symmetric matrices with the natural 
inner product 
$$
M:N=\sum\limits_{i,j=1}^2 M_{ij}N_{ij}
$$
for  
$
M=(M_{ij})_{1\le i,j\le2},
N=(N_{ij})_{1\le i,j\le2}\in\mtwo.
$

The {\it mismatch strain}  
$x\in\Omega\cup\substrate\mapsto E_0(x)\in\mtwo$ 
is given by  
$$
E_0: = 
\begin{cases}
\str{u_0} & \text{in $\Omega ,$}\\
0 &  \text{in $\substrate,$}
\end{cases}
$$ 
for a fixed $u_0\in H^1(\Omega )$.

Given $m\ge1,$ let $\fA_m$ be a collection of all  
subsets $A$ of $\overline\Omega $ such that 
$\p A$  has  at most $m$ connected components. Recall that 
since $\p A$ is closed, it is $\cH^1$-measurable.
By Proposition \ref{prop:rectifiable_sets}, $\p A$ is 
$\cH^1$-rectifiable so that $\fA_m\subset \fA.$
We call the set 
$$
\admissible_m:=\Big\{(A,u)\in \admissible:\,\,A\in \fA_m \Big\} 
$$ 
the set of constrained admissible configurations.
We also consider a volume constraint with respect to $\fm\in (0,|\Omega |],$  i.e., 
$$
|A|=\fm
$$
for every $A\in \fA.$ 

\subsection{Applications} \label{subsec:applications}
The model introduced in this paper 
includes the settings of various free boundary problems, 
some of which are outlined below.

\begin{itemize}[leftmargin=\dimexpr+4mm]
\item[--] {\it Epitaxially-strained thin films}
\cite{BChS:2007,DP:2018_1,DP:2018_2,FFLM:2007,FM:2012}:
$\Omega :=(a,b)\times (0,+\infty),$  
$\substrate :=(a,b)\times(-\infty,0)$ 
for some $a<b,$  
free crystals in the subfamily
$$
\begin{aligned}
\quad \mathcal{A}_{\rm subgraph}:=\{A\subset\Omega:\, 
\text{$\exists h\in BV(\Sigma;[0,\infty))$ and 
l.s.c. such that $A=A_h$} \}\subset\fA_1,
\end{aligned}
$$ 
where $A_h:=\{(x^1,x^2)\,:\, 0<x^2<h(x^1)\}$,
and admissible configurations in the subspace 
$$
\mathcal{C}_{\rm subgraph}:=\{(A,u):\, A\in\fA_{\rm subgraph},\,
u\in H_\loc^1(\Ins{A};\R^2)\}\subset\admissible_1
$$
(see also \cite{BGZ:2015,GZ:2014}). 
Notice that the container $\Omega$ is not bounded, 
however, we can reduce to the situation of bounded containers
where we can apply 
Theorem \ref{teo:existence_thin_void} 
since every energy equibounded sequence 
in $\mathcal{A}_{\rm subgraph}$
is contained in an auxiliary bounded set
(see also Remark \ref{rem:thin_film_appl}).
\vspace{1mm}

\item[--] {\it Crystal cavities} \cite{FFLM:2011,FJM:2018,SMV:2004,WL:2003}: 
$\Omega \subset\R^2$ smooth set containing the origin, 
$\substrate :=\R^2\setminus\Omega,$ free crystals in the subfamily
$$
\mathcal{A}_{\rm starshaped}:=\{A\subset\Omega:\, 
\text{open, starshaped w.r.t. $(0,0)$,
and $\partial\Omega \subset\partial A$} \}\subset\fA_1,
$$ 
and the space of admissible configurations
$$
\mathcal{C}_{\rm starshaped}:=\{(A,u):\, \, A\in\fA_{\rm starshaped},\,
u\in H_\loc^1(\Ins{A};\R^2)\}\subset\admissible_1. 
$$
See Remark \ref{rem:thin_film_appl}.

\item[--] 
{\it Capillarity droplets, e.g.,} \cite{CM:2007,GB-W:2004,DPhM:2015}: 
$\Omega\subset\R^2$ is a bounded open set (or a cylinder),
$\C=0,$ $\substrate = \emptyset$, and 
admissible configurations in the collection
$$
\mathcal{C}_{\rm capillary}:=\{(A,0):\, \,A\in\fA\}\subset\admissible.
$$

\item[--]
{\it Griffith fracture model, e.g.,} \cite{CCF:2016,CCI:2017,CFI:2016,CFI:2018}:
$\substrate=\Sigma=\emptyset$  \BBB $E_0\equiv0$, and \EEE
the space of configurations
$$
\mathcal{C}_{\rm Griffith}:=\{(\Omega \setminus K,u):\,
\text{$K$ closed, $\mathcal{H}^1$-rectifiable},\,
u\in H_\loc^1(\Omega\setminus K;\R^2)\}\subset\admissible.
$$

\item[--]
{\it Mumford-Shah model \BBB(without fidelity term), \EEE e.g.,} \cite{AFP:2000,DMMS:1992,MS:2001}:
$\substrate=\Sigma=\emptyset,$ $E_0=0,$
$\C$ is such that the elastic energy $\mathcal{W}$ reduces 
to the Dirichlet energy, and
the space of configurations
$$
\admissible_{\rm Mumfard-Shah}: = 
\{
(\Omega \setminus K,u)\in \admissible_{\rm Griffith}:\,\,
u=(u_1,0)\}\subset\admissible.
$$

\item[--] 
{\it Boundary delaminations}  
\cite{Babadjian:2016,Deng:1995,HS:1991,Baldelli:2014,Baldelli:2013,Xia:2000}:  
the setting of our model finds applications to describe
debonding and edge delaminations in composites \BBB \cite{Xia:2000}. We notice that our perspective differs from 
\cite{Babadjian:2016,Baldelli:2014,Baldelli:2013} where reduced models for the  horizontal  interface between the film and the substrate are derived, since instead we focus on the 2-dimensional film and substrate vertical section. \EEE
\end{itemize}

\subsection{Main results}\label{subsec:main_results}

In this subsection we  state  the main results of the paper.
Let us  formulate  our main hypotheses:
\begin{itemize}
\item[(H1)] $\varphi\in C(\cl{\Omega }\times\R^2;[0,+\infty))$ and is a 
Finsler norm, i.e., there exist $c_2\ge c_1>0$ such that 
for every $x\in \cl{\Omega },$ $\varphi(x,\cdot)$ is a norm in $\R^2$
satisfying 
\begin{equation}\label{finsler_norm}
c_1|\xi| \le \varphi(x,\xi) \le c_2|\xi|\quad \text{ for any 
$x\in\cl{\Omega}$ and $\xi\in\R^2$};
\end{equation}

\item[(H2)] $\beta\in L^\infty(\Sigma)$ and 
satisfies
\begin{equation}\label{hyp:bound_anis}
-\varphi(x,\nu_\Sigma(x))\le \beta(x) \le \varphi(x,\nu_\Sigma(x))\qquad
\text{for $\cH^1$-a.e.\ $x\in \Sigma;$}
\end{equation}

\item[(H3)] $W$ is of the form \eqref{asjkladkl} with    
$\C\in L^{\infty}(\Omega\cup\substrate)$ such that 
\begin{equation}\label{hyp:elastic}
\C(x)M:M \ge 2c_3\,M:M\quad \text{for any $M\in\mtwo$}
\end{equation}
for some $c_3>0.$ 
\end{itemize}


\begin{theorem}[\textbf{Existence}]\label{teo:existence}
Assume {\rm (H1)-(H3)}. Let either $\fm\in (0,|\Omega |)$ or $\substrate=\emptyset.$  Then for every $m\ge1,$ $\lambda>0$ both the volume-constrained minimum problem 
\begin{equation}\label{minimum_prob_1}
\inf\limits_{(A,u)\in \admissible_m,\,\,|A| = \fm} \cF(A,u), \tag{CP}
\end{equation}
and the unconstrained minimum problem
\begin{equation}\label{minimum_prob_2}
\inf\limits_{(A,u)\in \admissible_m} \cF^\lambda(A,u), 
\tag{UP}
\end{equation}
have a solution, 
where $\cF^\lambda:\admissible_m\to\R$ is defined as 
$$
\cF^\lambda(A,u):=\cF(A,u) +\lambda\big||A| - \fm\big|.
$$
Furthermore,  there exists $\lambda_0>0$ such that 
for every $\fm\in (0,|\Omega |]$ and $\lambda>\lambda_0,$ 
\begin{equation}\label{zur_tenglik}
\inf\limits_{(A,u)\in \admissible,\,\,|A|=\fm} \cF(A,u) 
= \inf\limits_{(A,u)\in \admissible} \cF^\lambda(A,u) 
= \lim\limits_{m\to\infty} \inf\limits_{(A,u)\in \admissible_m,
\,\,|A|=\fm} \cF(A,u).  
\end{equation}

\end{theorem}

We notice that for $\lambda>\lambda_0$ solutions of \eqref{minimum_prob_1}
and \eqref{minimum_prob_2} coincide (see the proof of Theorem \ref{teo:existence})
for any $\red \fm \black \in(0,|\Omega|]$ and $m\ge1$.
Moreover, \eqref{zur_tenglik} shows that a minimizing sequence 
for $\cF$ in $\admissible$ can be chosen among the sets whose boundary 
have finitely many connected components. 

The proof of the existence part of Theorem \ref{teo:existence} is given mainly by the following two results in which we show that $\admissible_m$ is $\tau_\admissible$-compact and $\cF$ is $\tau_\admissible$-lower semicontinuous.
Recall that an (infinitesimal) rigid displacement in $\R^n$ is an affine transformation  $a(x)= Mx+b,$ where $M$ is a skew-symmetric (i.e.,  $M^T = -M$) $n\times n$-matrix and $b\in\R^n.$ Given $B\in \fA$ with $\Int{B} = \cup_{j} E_j,$ where $\{E_j\}$ are all  connected components of $\Int{B},$ we say the function 
$$
a=\sum\limits_{j\ge1} (M_jx+b_j)\chi_{E_j},
$$
a {\it piecewise rigid displacement} associated to $B,$ here $M_jx+b_j$ is a rigid displacement in $\R^2$. 

\begin{theorem}[\textbf{Compactness of $\admissible_m$}] \label{teo:compactness_Ym} 
Assume  {\rm (H1)-(H3)}. Let either $\fm\in(0,|\Omega|)$ or $\substrate=\emptyset.$ Let $\{(A_k,u_k)\}\subset\admissible_m$ be such that 
\begin{equation*}
\sup\limits_{k\ge1} \cF(A_k,u_k) < \infty 
\end{equation*}
and 
\begin{equation}\label{volume_constr}
|A_k|\le \fm 
\end{equation}
for every $k\ge1.$
Then there exist  $(A,u)\in\admissible_m$ of finite energy, a  subsequence $\{(A_{k_n},u_{k_n})\}$ and a sequence $\{(D_n,v_n)\}\subset \admissible_m$  \red with 
$$
v_n:=(u_{k_n}+a_n)\chi_{D_n\cap A_{k_n}} + u_0\chi_{D_n\setminus A_{k_n}}
$$
for some piecewise rigid displacements $a_n$ associated to $D_n,$
\black 
such that $A_{k_n} \overset{\tau_\fA}{\to} A,$
$(D_n,v_n) \overset{\tau_\admissible}{\to} (A,u)$, $|D_n|=|A_{k_n}|,$ 
 and  
\begin{equation}\label{liminfga_ut_eshmat}
\liminf \limits_{n\to \infty}  \cF(A_{k_n},u_{k_n}) \ge \liminf \limits_{n\to \infty}  \cF(D_n,v_n).  
\end{equation}
\black 
\end{theorem}

\begin{theorem}[\textbf{Lower semicontinuity of $\cF$}]\label{teo:lsemico_film}
Assume  (H1)-(H3) and let \allowbreak $\{(A_k,u_k)\}\subset\admissible_m$ and 
$(A,u)\in \admissible_m$ 
be such that  $(A_k,u_k)\overset{\tau_{\admissible}}\to (A,u).$
Then 
$$
\liminf\limits_{k\to\infty} \cF(A_k,u_k) \ge \cF(A,u).
$$
\end{theorem}  

As a byproduct of our methods we obtain the 
following existence result in a subspace of $\admissible_m$ 
with respect to a weaker topology previously used in 
\cite{DP:2018_1,FFLM:2011,FFLM:2007} for thin films 
and crystal cavities. 

\begin{theorem}[\textbf{Existence for  weaker topology}]\label{teo:existence_thin_void}
Assume {\rm(H1)-(H3)} and fix $m\ge1$ and $\fm\in (0,|\Omega |].$ The
functional $\cF':\admissible\to\R$
%
defined as 
$$
\begin{aligned}
\cF'(A,u):=\cF(A,u) &- 2\int_{\Omega \cap A^{(0)}\cap\p A}
\varphi(x,\nu_A)d\cH^1
\\
&- \int_{\Sigma \cap A^{(0)} \cap \p A} \big(\phi(x,\nu_A) 
+ \beta\big) d\cH^1 
-\int_\Sigma \beta d\cH^1, 
\end{aligned}
$$
admits a minimizer $(A,u)$ in every $\tau_\admissible'$-closed subset 
of 
$$
\admissible_m':=\big\{(A,u)\in \admissible:\,\, \text{$A$ open, $|A| = \fm$,  
and $A\cup\Sigma\in \fA_m$}\big\},
$$
where 
$\{(A_k,u_k)\}\subset\admissible$ 
converges to $(A,u)\in\admissible$ in 
$\tau_\admissible'$-sense if 
\begin{itemize}
\item[--] $\sup\limits_{k\ge1}\,\,\cH^1(\p A_k)<\infty,$
\item[--]  $\R^2\setminus A_k 
\overset{\cK}{\to} \R^2\setminus  A,$
\item[--]  $u_k\to u$ a.e.\ in $\Int{A}\cup \substrate.$ \black 
\end{itemize}

\end{theorem}
\smallskip 

\begin{remark}\label{rem:thin_film_appl}
The sets
 $\admissible_{\rm subgraph}$ and  $\admissible_{\rm starshaped}$ 
defined in Subsection \ref{subsec:applications} are $\tau_\admissible'$-closed
in $\admissible_m'$ (see e.g., \cite[Proposition 2.2]{FFLM:2007}).
In the thin-film setting, we define 
$\varphi$ and $\beta$ as $\varphi: = \gamma_f$ and 
$$
\beta:=-\max\big\{\min\{\gamma_f,\gamma_s - \gamma_{fs}\},-\gamma_f\big\},
$$ 
where $\gamma_f,$ $\gamma_s,$ and $\gamma_{fs}$ denote 
the surface tensions of 
the film-vapor, substrate-vapor, and 
film-substrate interfaces,   
respectively. 
The energy $\cF'$ coincides (apart from the presence of delamination) \black  with the thin-film energy
in \cite{DP:2018_1,DP:2018_2}  in the case 
$\gamma_f,$ $\gamma_s,$ $\gamma_{fs}$  are constants, 
$\gamma_s - \gamma_{fs}\ge0,$ $\gamma_s>0,$ and $\gamma_f>0.$
Therefore, 
Theorem \ref{teo:existence_thin_void} 
extends  the existence results in \cite{DP:2018_1,FFLM:2007}
to all values of $\gamma_s$ and
$\gamma_s - \gamma_{fs}$,  as well as to anisotropic surface tensions
and anisotropic elastic densities.
\end{remark}

\begin{remark} 
All the results contained in this subsection hold true with essentially the same proofs by replacing (H3) with the more general assumption: 

\begin{itemize}
\item[(H3')]  $W:(\Omega\cup \substrate)\times \mtwo \to [0,\infty)$ is a function such that $M\mapsto W(x,M)$ is convex for any $x\in\Omega\cup\substrate$ and 
$$
c'|M|^p \le W(x,M) \le c''|M|^p +f(x)
$$ 
for some $p\ge2,$ $c''\ge c'>0$ and $f\in L^1(\Omega\cup\substrate).$
\end{itemize}
\end{remark}

\section{Compactness}\label{sec:compactness}

In this section we prove Theorem \ref{teo:compactness_Ym}. 
Convergence of sets with respect to the signed distance functions
has the following compactness property.

\begin{proposition}[\textbf{Blaschke-type 
selection principle}]\label{prop:compactness_wrt_sdist}
For every sequence $\{A_k\}$ of subsets $\R^2$  
there exist  a subsequence
$\{A_{k_l}\}$ and $A\subset\R^2$ such that 
$\sdist(\cdot,\p A_{k_l})\to \sdist(\cdot,\p A)$ 
locally uniformly in $\R^2$ as $l\to\infty.$
\end{proposition}

\begin{proof}
Without loss of generality we suppose $A_k\notin\{\R^2,\emptyset\}.$ 
By the Blaschke selection principle 
\cite[Theorem 6.1]{AFP:2000},
there exists a not relabelled subsequence 
$\{A_k\}$ and a closed set $K\subset \R^2$
such that $\p A_k$ converges to $K$ in the Kuratowski sense 
as $k\to\infty.$ 
Notice that by Proposition \ref{prop:equivalent_kuratowski}, 
\begin{equation}\label{with_abs_values}
|\sdist(\cdot,\p A_k)|\to \dist(\cdot, K) 
\end{equation}
locally uniformly as $k\to\infty$
since $|\sdist(\cdot,\p A_k)| = \dist(\cdot,\p A_k).$ 
As $\sdist(\cdot,\p A_k)$ is 
1-Lipschitz, by the Arzela-Ascoli Theorem,
passing to  a further not relabelled subsequence 
one can find $f:\R^2\to[-\infty,+\infty]$ such that 
$$
\sdist(\cdot,\p A_k) \to f 
$$
locally uniformly in $\R^2$ as $k\to\infty.$
By \eqref{with_abs_values},
$|f(\cdot)|= \dist(\cdot,K).$ Recall 
that $K$ may have nonempty interior. Fix a countable set $Q\subset \Int{K}$
dense in $\Int{K},$
and define 
$$
A:=\{f< 0\}\cup(\Int{\overline{\{f>0\}}}\cap \p K) \cup Q.
$$
By construction, $\Int{A} = \{f<0\}$,  $\overline{A}=\{f\le 0\}\cup K$
and $\p A=\{f=0\}=K.$

Finally we show that 
$$
f(x) = \sdist(x,\p A).
$$
If $x\in A,$ by the definition of $A$ and $K,$
$f(x)\le 0$ so that
$$
f(x) = -\dist(x,K) = -\dist(x,\p A) = -\dist(x, \R^2\setminus A).
$$
Analogously, if $x\notin A,$ then $f(x)\ge0$ and hence
$$
f(x) = \dist(x,K) = \dist(x,\p A) = \dist(x, A).
$$
\end{proof}

In general, the collection $\fA$ is not closed under 
$\tau_\fA$-convergence.  Indeed, let $E:=\{x_k\}$ be a 
countable dense set in 
$B_1(0)$ and $E_k:=\{x_1,\ldots,x_k\}\in\fA.$ 
Then $\cH^1(\p E_k) = 0,$ and 
$E_k\overset{\tau_{\fA}}{\to} E$ as $k\to\infty$,
but $E\notin\fA$
since $\p E = \overline{B_1(0)}.$ 
However,  $\fA_m$ is closed with respect to
the $\tau_{\fA}$-convergence.

\begin{lemma}\label{lem:f_convergence}
Let  $A\subset\Omega $ and $\{A_k\}\subset\fA_m$ be
such that $A_k\overset{\tau_{\fA}}{\to} A$.
Then: 
\begin{itemize}
\item[\rm (a)] $A\in \fA_m$
and  
\begin{equation}\label{lsc_H1_measure}
\cH^1(\p A)\le \liminf\limits_{k\to\infty} \cH^1(\p A_k); 
\end{equation}
\item[\rm (b)] $A_k \to A$ in  $L^1(\R^2)$  as $k\to\infty.$   
\end{itemize}
\end{lemma}

\begin{proof}
(a) By Remark \ref{rem:kuratiwski_and_sdistance}, 
$\p A_k\overset{\cK}{\to}\p A$ as $k\to\infty.$
Thus, by \cite[Theorem 2.1]{Gi:2002} 
$\p A$ has at most $m$-connected components, and 
\eqref{lsc_H1_measure} holds.   

(b) As $\p A_k\overset{\cK}{\to}\p A,$ for any 
$x\in \Int{A}$ resp. $x\in \R^2\setminus\cl{A},$ there exists 
$k_x>0$ such that $x\in A_k$ resp. $x\in \R^2\setminus \cl{A_k}$ 
for all $k>k_x.$ Finally, by \eqref{lsc_H1_measure}, $|\p A| = 0,$ 
and therefore,
$$
\chi_{A_k} \to \chi_A\qquad\text{a.e.\ $x\in\R^2.$}
$$
Now (b) follows from the uniform boundedness of $\{A_k\}$
and the Dominated Convergence Theorem.
\end{proof}

Furthermore, sequences $\{A_k\}\subset\fA_m$ with equibounded
boundary lengths are compact with respect to the $\tau_\fA$-convergence.

\begin{proposition}[\textbf{Compactness of $\fA_m$}]\label{prop:compactness_A_m}
Suppose that $\{A_k\}\subset\fA_m$ is such that  
$$
\sup\limits_{k\ge1} \cH^1(\p A_k) <\infty.
$$
Then there exists a subsequence
$\{A_{k_l}\}$ and $A\in\fA_m$ such that 
$\cH^1(\p A)<\infty$ and  
$\sdist(\cdot,\p A_{k_l})\to \sdist(\cdot,\p A)$ 
locally uniformly in $\R^2$
as $l\to\infty.$
\end{proposition}

\begin{proof}
By Proposition \ref{prop:compactness_wrt_sdist}  
there exists a not relabelled subsequence 
$\{A_k\}$ and a set $A$  
such that $\p A_k\overset{\cK}{\to} \p A$ 
and $\sdist(\cdot,\p A_k)\to\sdist(\cdot,\p A)$ locally uniformly in $\R^2$
as $k\to\infty.$ 
By Lemma \ref{lem:f_convergence}, $A\in\fA_m$ and $\cH^1(\p A)<\infty.$
\end{proof}

\begin{proposition}\label{prop:good_division}
Let $\{A_k\}\subset\fA_m$ be such that $A_k\overset{\tau_\fA}{\to} A$ as $k\to\infty.$ 
Suppose that 
$$
\Int{A}=\bigcup\limits_{h\in I}E_h,\qquad F=\bigcup\limits_{i\in I_1} E_i\qquad \text{and} \qquad G=\bigcup\limits_{j\in I_2} E_j,
$$ 
where $E_h$ are disjoint connected components of $\Int{A},$ $I_1$ and $I_2$ are disjoint finite subsets of $I.$  
Then there exist  a subsequence $\{A_{k_l}\}$ and a sequence $\{\gamma_l\}$ of $\cH^1$-rectifiable sets in $\R^2$ such that
\begin{itemize}
\item[(a)] $\gamma_l\subset  \Int{A_{k_l}}$ and $\lim\limits_{l\to\infty} \cH^1(\gamma_l) =0;$

\item[(b)] $\sdist(\cdot,\p (A_{k_l}\setminus \gamma_l))\to \sdist(\cdot,\p A)$ as $l\to\infty$ locally uniformly in $\R^2;$


\item[(c)]  for any connected open sets $D'\strictlyincluded F$ and $D''\strictlyincluded G$ there exists $l'$ such that $D'$ and $D''$ belong to different connected components of $\Int{A_{k_l}\setminus \gamma_l}$ for any  $l>l'.$ 


\end{itemize}
\end{proposition}

We postpone the proof after the following lemma.
Before we need to introduce some notation. 
Let $n_0>1$ be such that $E_h\cap \{\dist(\cdot,\p A) > \frac1n\}\ne \emptyset$ for every $h\in I_1\cup I_2$ and $n> n_0.$ Given $h\in I_1\cup I_2,$ let $\{E_h^n\}_{n> n_0}$ be an increasing sequence of connected open sets satisfying $E_h\cap \{\dist(\cdot,\p A) >\frac1n\} \subseteq E_h^n\strictlyincluded E_h$ and 
\begin{equation}\label{e_i_e_in}
E_h=\bigcup\limits_{n\ge n_0} E_h^n. 
\end{equation}
By the $sdist$-convergence and the finiteness of $I_1\cup I_2,$ for any $n\ge n_0$ there exist $k_n^0>0$  such that  $E_h^n\strictlyincluded\Int{A_k}$ for all $k>k_n^0$ and $h\in I_1 \cup I_2.$ Let 
\begin{equation}\label{def_dn}
2d_n:= \min\limits_{i\in I_1,\,j\in I_2}\{\dist(E_i^n,E_j^n),\dist(E_i^n,\p A),\dist(E_j^n,\p A)\}. 
\end{equation}
Note that $0<d_n<\frac{1}{2n}.$  

The idea of the proof of Proposition \ref{prop:good_division} is to ``partition'' the connecting components of $\Int{A_k}$ which in the limit break down into connected components $\{E_h\}_{h\in I'}$ of $\Int{A}$ such that $I'\cap I_1\ne\emptyset$ and $I'\cap I_2\ne\emptyset,$  for example in the case of neckpinches.  More precisely, we cut out at most $m$-circles from $\Int{A_k}$ such that for any $n>n_0,$ for all sufficiently large $k$ (depending only on $n$), any curve $\gamma\subset\Int{A_k}$ connecting a point of $E_i^n,$  $i\in I_1,$ to a point of $E_j^n,$ $j\in I_2,$ intersects at least one of these circles. The following lemma consists in performing this argument for fixed $i\in I_1$ and $j\in I_2$.

\begin{lemma}\label{lem:bal_cutting}
Under the assumptions of Proposition \ref{prop:good_division}, let $i\in I_1,$ $j\in I_2,$ and $n>n_0$ be such that the set 
\begin{align}\label{y_ij_n}
Y=Y_{ij}^n:=\Big\{k\in\N:\,  \exists D_k\strictlyincluded &\Int{A_k}\, \text{closed, connected,} \no \\
&\text{and such that } D_k\cap E_i^{n},D_k\cap E_j^{n}\ne\emptyset\Big\}
\end{align}
%
is infinite. Then, there exists  $k_n^{ij}>k_n^0$ such that for any $k\in Y$ with $k>k_n^{ij}$ there exists a collection $\{B_{r_k^l}(z_k^l)\}_l$  of at most $m$ balls contained in $A_k$ such that $r_k^l<d_n$ and any curve $\gamma\strictlyincluded\Int{A_k},$ connecting a point of $E_i^{n}$ to a point of $E_j^{n},$ intersects at least one of $B_{r_k^l}(z_k^l).$ 
\end{lemma}

\begin{proof} 

We divide the proof into four steps.

{\it Step 1:} for any $k\in Y,$ let $C_k\strictlyincluded \Int{A_k}$  be any closed connected set intersecting both $E_i^{n_0}$ and $E_j^{n_0}.$ Then 
$$
\lim\limits_{k\in Y,\,k\to\infty}\, \dist(C_k,\p A_k) =0.
$$

By contradiction, assume that there exists $\epsilon>0$ such that 
\begin{equation}\label{adjjhsd}
\dist(C_k,\p A_k) \ge \epsilon 
\end{equation}
for infinitely many  $k\in Y.$ By the Kuratowski-compactness of closed sets there exist a closed connected set $C$  and a not relabelled subsequence $\{C_k\}_{k\in Y}$ satisfying \eqref{adjjhsd} for all $k\in Y$ such that $C_k\overset{K}{\to} C$ as $k\to\infty.$ 
Since $A_k\overset{\tau_\fA}\to A,$  in view of Remark \ref{rem:kuratiwski_and_sdistance}  $\p A_k\overset{K}\to\p A$ and $D\subset A.$ Let $x\in C$ and $y\in \p A$ be such that $|x-y| = \dist(C,\p A).$ Then by the definition of the Kuratowski convergence, there exist  sequences  $x_k\in C_k$ and $y_k\in \p A_k$ such that $x_k\to x$ and $y_k\to y.$ Since  $|x_k-y_k| \ge\dist(C_k,\p A_k)\ge\epsilon,$ it follows that 
\begin{equation}\label{dhakshjd}
\dist(C,\p A) = |x- y| =\lim\limits_{k\to\infty} |x_k-y_k|\ge\epsilon. 
\end{equation}
Thus, $C\strictlyincluded \Int{A}.$ In particular, \eqref{dhakshjd} implies that the non-empty connected open set $\{\dist(\cdot,C)<\frac{\epsilon}{4}\}$ is compactly contained in $\Int{A}$ and intersects both $E_i^{n_0}$ and $E_j^{n}$ so that $E_i^{n}\cup \{\dist(\cdot,C)<\frac{\epsilon}{4}\} \cup E_j^{n} \subset\Int{A}$ is connected. But this is a contradiction since $E_i^{n}$ and $E_j^{n}$ belong to different connected components of $\Int{A}.$

{\it Step 2:} for every $k\in Y$ there exists  a path-connected closed set $L_k\strictlyincluded\Int{A_k}$  intersecting both $E_i^{n}$ and $E_j^{n}$ such that
\begin{equation}\label{husha2}
\dist(L_k,\p A_k) =\delta_k:=  \sup\,\, \dist(D,\p A_k), 
\end{equation}
where $\sup$ is taken over all closed connected sets $D\strictlyincluded \Int{A_k},$ intersecting both $E_i^{n_0}$ and $E_j^{n_0}$ (such sets exist by definition of $Y$). Moreover, there exists $k_n^1>0$  such that $L_k$ contains  $E_i^{n}\cup E_j^{n}$ and $\delta_k<d_n$ for any $k>k_n^1.$  

Indeed, in view of the Kuratowski-compactness of closed sets and from the Kuratowski-continuity of $\dist(\cdot,\p A_k),$ \eqref{husha2} has a maximizer $L_k'.$ Applying Step 1 with $A_k$ and $C_k=L_k',$  we get $\delta_k\to0$ as $k\to\infty.$  
Let  $L_k$ be the connected component of $\{\dist(\cdot,\p A_k)\ge \delta_k\}$ containing $L_k'.$ 
Since $E_i^{n}\cup E_j^{n} \strictlyincluded \Int{A},$ the $sdist$-convergence and Remark \ref{rem:kuratiwski_and_sdistance}, $E_i^{n}\cup E_j^{n} \strictlyincluded \Int{A_k}$ for all large $k.$ More precisely, by the definition \eqref{def_dn} of $d_n,$ there exists $\bar k_n^1>0$ such that  
\begin{equation}\label{dahjkdakjad}
\min\{\dist(E_i^n,\p A_k),\dist(E_j^n,\p A_k)\}\ge d_n 
\end{equation}
for all $k>\bar k_n^1.$ By construction, $\dist(L_k,\p A_k)= \delta_k,$ and  since  
$\delta_k\to0,$  there exists $k_n^1>\bar k_n^1$ such that $\delta_k<d_n$ for any $k\ge  k_n^1.$ Note that by \eqref{dahjkdakjad} for such $k$ we have also $E_i^{n}\cup E_j^{n} \subset L_k.$ 

Let us show that $L_k$ is also path-connected. Indeed, given $x\in L_k,$ consider the ball $B_r(x)$ for small $r<\delta_k.$ Then $L_k\cap B_r(x)$ is path-connected, otherwise there would exist a curve in $B_r(x)$ with endpoints in $L_k$ containing a point $z\in B_r(x)\setminus L_k$ such that $\dist(z,\p A_k)>\delta_k$ contradicting to the definition of $L_k.$  Thus, $L_k$ is locally path-connected. Now the compactness and the connectedness of $L_k$ imply its path-connectedness.

{\it Step 3:} given $x\in E_i^{n_0}$ and $y\in E_j^{n_0},$ let $\gamma_k\subset L_k$ be a curve connecting $x$ to $y.$ Then for any $k>k_\epsilon^0$ there exists $z_k\in \gamma_k\setminus \overline{E_i^{n_0} \cup E_j^{n_0}}$ such that any curve $\gamma\strictlyincluded \Int{A_k}$ homotopic in $\Int{A_k}$ to $\gamma_k$ (with same endpoints)  intersects the ball $B_{\delta_k}(z_k).$  

Indeed, otherwise slightly perturbing the curve $\gamma_k$ around the points of the compact set $\gamma':=\{x\in\gamma_k:\, \dist(x,\p A_k)=\delta_k\}$ we would get a new curve $\tilde \gamma_k\strictlyincluded \Int{A_k}$ connecting $x$ to $y$ for which $\dist(x,\gamma_k)>\delta_k$ for all $x\in \tilde \gamma_k.$ Now the compactness of $\tilde \gamma_k$ implies $\dist(\tilde \gamma_k,\p A_k)>\delta_k,$ which contradicts to the definition \eqref{husha2} of $L_k.$

{\it Step 4:} now we prove the lemma.

Applying Steps 1-3 with $A_k,$ we find an integer $k_n^1>k_n^0,$ a curve $\gamma_k^1$ connecting a point of $E_i^{n_0}$ to a point $E_j^{n_0}$ such that 
\begin{equation}\label{step1aaa}
\dist(\gamma_k^1,\p A_k)=r_k^1=\sup \dist(D,\p A_k) < d_n 
\end{equation}
where $sup$ is taken over all connected and closed $D\strictlyincluded \Int{A_k}$ intersecting both  $E_i^{n_0}$ and $E_j^{n_0},$ and a ball $B_{r_k^1}(z_k^1)\subset A_k$ with $z_k^1\in\gamma_k^1$ such that 
any curve $\gamma\strictlyincluded \Int{A_k}$ homotopic to $\gamma_k^1$ intersects $B_{r_k^1}(z_k^1)$
for any $k\in Y$ with $k>k_n^1.$ 

For $k\in Y$ with $k>k_n^1$ set 
$$
A_k^1:=A_k\setminus(\Int{A_k}\cap \p B_{r_k^1}(z_k^1)).
$$
Now consider the set $Y_1$ of all $k\in Y$ for which there exists a closed connected set $C_k\strictlyincluded \Int{A_k}$ intersecting both $E_i^{n_0}$ and $E_j^{n_0}.$
If $Y_1$ is finite, we set $k_n^{ij}:= \max\{\max Y_1, k_n^1\}$ and we are done. 

Assume that $Y_1$ is infinite.  Note that for any $k\in Y_1,$ $\p B_{r_k^1}(z_k^1)$ touches at least two different connected components of $\p A_k$ and thus, $A_k^1\in \fA_{m-1}.$ Applying Steps 1-3 with $A_k^1$ and $Y_1,$ we find an integer $k_n^2>k_n^1,$ a curve $\gamma_k^2$ connecting a point of $E_i^{n_0}$ to a point $E_j^{n_0}$ such that 
$$
\dist(\gamma_k^2,\p A_k^1)=r_k^2=\sup \dist(D,\p A_k^1)  
$$
where $sup$ is taken over all connected and closed $D\strictlyincluded \Int{A_k^1}$ intersecting both  $E_i^{n_0}$ and $E_j^{n_0},$ and a ball $B_{r_k^2}(z_k^2)\subset A_k^1$ with $z_k^2\in\gamma_k^2$ such that  any curve $\gamma\strictlyincluded \Int{A_k}$ homotopic to $\gamma_k^2$ intersects $B_{r_k^2}(z_k^2)$ for any $k\in Y_1$ with $k>k_n^2.$  By \eqref{step1aaa}, $r_k^1\ge r_k^2.$  

For $k\in Y_1$ with $k>k_n^2$ set 
$$
A_k^2:=A_k\setminus(\Int{A_k}\cap (\p B_{r_k^1}(z_k^1) \cup \p B_{r_k^2}(z_k^2))) 
$$
and consider the set $Y_2$ of all $k\in Y_1$ for which there exists a closed connected set $C_k\strictlyincluded \Int{A_k^2}$ intersecting both $E_i^{n_0}$ and $E_j^{n_0}.$ Note that $Y_2$ is finite, setting $k_n^{ij}:= \max\{\max Y_2, k_n^2\}$ and we are done. If $Y_2$ is infinite, then $A_k^2\in \fA_{m-2},$ and we repeat the same procedure above. After at most $m$ steps  we obtain $k_n^{ij}>k_n^0$ such that  for any $k>k_n^{ij}$ there is a collection $\{B_{r_k^l}(z_k^l)\}$ of at most $m$ balls,  which satisfy the assertion of the lemma. 
\end{proof}

The assertions of  Proposition \ref{prop:good_division} follow by applying Lemma \ref{lem:bal_cutting} with all pairs $(i,j)\in I_1\times I_2.$ 

\begin{proof}[Proof of Proposition \ref{prop:good_division}]

Given $i\in I_1,$  $j\in I_2$ and $n>n_0,$ let $Y_{ij}^n$ be given by \eqref{y_ij_n}. 
If $Y_{ij}^n$ is infinite, let  $k_n^{ij}$ be given by Lemma \ref{lem:bal_cutting}, otherwise set
$k_n^{ij}:= 1+\max Y_{ij}^n.$  Let $k_n:=1+\max\limits_{i,j}\,k_n^{ij}$ and  
$$
\gamma_n^{ij}:=
\begin{cases}
\Int{A_{k_n}} \cap \bigcup\limits_l \p B_{r_{k_n}^l}^{ij}(z_{k_n}^l),&  \text{if $k_n\in Y_{ij}^n,$}\\
\emptyset, & \text{if $k_n\notin Y_{ij}^n,$}
\end{cases}
$$
where $\{B_{r_k^l}^{ij}(z_k^l)\}$ is the collection of balls given by Lemma \ref{lem:bal_cutting}.
Without loss of generality we assume that $\{k_n\}_n$ is strictly increasing and set 
$$
\gamma_n:=\bigcup\limits_{i,j} \gamma_n^{ij}.
$$
Being a union of at most $N_1N_2m$ circles, $\gamma_n$ is $\cH^1$-rectifiable, here $N_i$ is the cardinality of $I_i.$ By  Lemma \ref{lem:bal_cutting},
\begin{equation}\label{asdsjkjkasdkas}
\cH^1(\gamma_n) \le \sum\limits_{i,j,l} \cH^1(\p B_{r_{k_n}^l}^{ij}(z_{k_n}^l)) \le 2\pi N_1N_2m\,d_n.  
\end{equation}
Then $\lim\limits_{n\to\infty} \cH^1(\gamma_n)=0$ and therefore, $\gamma_n$ converges in the Kuratowski sense  to at most $N_1N_2m$ points on $\p A.$ 

We claim that the sequences $\{A_{k_n}\}$ and $\{\gamma_n\}$ satisfy assertions (a)-\red(c)\black. Indeed, by \eqref{asdsjkjkasdkas},  $\{\gamma_n\}$ satisfy (a). 
Since $\gamma_n$ converges to at most $N_1N_2m$ points on $\p A$ in the Kuratowski sense, \red(b) \black follows. To prove \red(c) \black, we take any connected open sets $D'\strictlyincluded E$ and $D''\strictlyincluded F.$ By connectedness and the definitions of $E_h$ and $E_h^n,$ there exist $i\in I_1$ and $j\in I_2$ and $\bar n>n_0$ such that $D'\strictlyincluded E_i^n$ and $D''\strictlyincluded E_j^n$ for all $n>\bar{n}.$ By the construction of $\gamma_n,$ the sets $E_i^n$ and $E_j^n$ (and hence, $D'$ and $D''$) belong to different connected components of $\Int{A_{k_n}}\setminus \gamma_n$ for all $n>\bar n.$ 
%
%
\end{proof} 

By inductively  applying Proposition \ref{prop:good_division} and by means of a diagonal argument we modify a sequence $\{A_k\}$ $\tau_\fA$-converging to a set $A$ into a sequence $\{B_k\}$ with  same $\tau_\fA$-limit and whose (open) connected components ``vanish'' or ``converge to the corresponding'' connected components of $A.$ This construction will be used in Step 1 of the proof of Theorem \ref{teo:compactness_Ym}.  We notice here that if $\substrate=\emptyset,$ then the sequence $\{D_n\}$ from Theorem \ref{teo:compactness_Ym} coincides with the sequence $\{B_n\}.$ Actually, if $\substrate=\emptyset,$ it would be enough to take $D_n=\tilde B_n,$ where $\tilde B_n$ is constructed in the Step 1 of the proof of the next proposition, since in this case we do not need properties (e) and (f) of the statement of the next proposition.

\begin{proposition}\label{prop:sets_changhe}
Let $A\in\fA_m$ and $\{A_k\}\subset\fA_m$ be such that $\sdist(\cdot,\p A_k)\to \sdist(\cdot,\p A)$ locally uniformly in $\R^2.$ 
Then  there exist  a subsequence $\{A_{k_l}\}$ and a sequence $\{B_l\}\subset\fA_m$ such that 
\begin{itemize}
\item[(a)]  $\p A_{k_l}\subset\p B_l$ and $\lim\limits_{l\to \infty} \cH^1(\p B_l\setminus \p A_{k_l})=0;$ 

\item[(b)]  $\sdist(\cdot,\p B_l)\to \sdist(\cdot,\p A)$ locally uniformly in $\R^2;$

\item[(c)] if $\{E_i\}$ is the set of all connected components of $\Int{A},$ we can choose a subfamily $\{E_i^l\}$ of connected components of $\Int{B_l}$ such  that for any $G\strictlyincluded E_i$ there exists $l_{i,G}>0$ with $G\strictlyincluded E_i^l$ for every $l>l_{i,G};$

\item[(d)] $|B_l| = |A_{k_l}|$ for every $l\ge1;$

\item[(e)]   
$$
\lim\limits_{l \to\infty}  \sup\limits_{x\in E_i^l\setminus E_i} \dist(x, E_i)=0 
$$
and 
$$
\lim\limits_{l \to\infty} \cH^1(\p\Omega\cap (\p E_i^l\setminus \p E_i)) =0.
$$

\item[(f)] the boundary of every connected component of $\Int{B_l}\setminus\bigcup_i E_i^l$ intersects the boundary of at most one connected component of $\substrate.$

\end{itemize}
\end{proposition}

\begin{proof}
Given $N,n\ge1,$ we define the index set $I_n^N$ by
$$
I_n^N:=\Big\{i>N:\,\, E_i\cap\{\dist(\cdot,\p A)>\frac{1}{n}\}\ne\emptyset\Big\}.
$$ 
We notice that $I_n^N$ is finite since $A$ is bounded.

{\it Step 1: Construction of $\{\tilde B_l\}$ and $\{A_{k_l}\}$ satisfying (a)-(d).}  This is done by using Proposition \ref{prop:good_division} iteratively in $N\in\N$ and a diagonal argument.

{\it Substep 1: Base of iteration.} By Proposition \ref{prop:good_division} applied with $\{A_k\}_{k\in Y^0}$ with $Y^0:=\N,$ $I_1=\{1\},$ and $I_2=I_n^1$ inductively with respect to $n\in\N,$ we find a decreasing sequence $Y^0\supset Y^1\supset \ldots$ of infinite subsets of $\N$ such that for the subsequence $\{A_k\}_{k\in Y^n}$ there exists a sequence $\{\gamma_k^n\}_{k\in Y^n}$ of $\cH^1$-rectifiable sets such that for any $n\ge1:$
\begin{itemize}
\item[--] $\gamma_k^n\subset \Int{A_k}$ for any $k\in Y^n$ and  $\lim\limits_{k\in Y^n,\, k\to\infty} \cH^1(\gamma_k^n) =0;$

\item[--] for any connected open sets $D\strictlyincluded E_1$ and $D'\strictlyincluded \cup_{j\in I_n^1} E_j$ there exists $k'>0$ such that $D$ and $D'$ belong to different connected components of $A_k\setminus \gamma_k^n$ for any $k\in Y^n$ with $k>k';$

\item[--]  $\sdist(\cdot,\p(A_k\setminus \gamma_k^n)) \to \sdist(\cdot,\p A)$ as $Y^n\ni k \to \infty$ locally uniformly in $\R^2.$
\end{itemize}
Then by a diagonal argument, we choose an increasing sequence $n\in\N\mapsto k_n^1\in Y^n$ such that  
$$
\tilde B_{1,n}:= A_{k_n}\setminus \gamma_{k_n^1}^n,\qquad n\in  \N,
$$
satisfies 
\begin{itemize}
\item[$a_{1n}$:\,\,] $\p A_{k_n}\subset \p \tilde B_{1,n}$  and $\cH^1(\p \tilde B_{1,n}\setminus \p A_{k_n})=\cH^1(\gamma_{k_n^1}^n)<2^{-n}$ for any $n\ge1;$

\item[$b_{1n}$:\,\,] $\sdist(\cdot,\p \tilde B_{1,n})\to \sdist(\cdot,\p A)$ as $n\to\infty$ locally uniformly in $\R^2;$

\item[$c_{1n}$:\,\,] for any connected open set  $D\strictlyincluded E_1$ there exist $n_D^1>1$ and a unique connected component denoted by $E_1^{1,n}$ of $\Int{\tilde B_{1,n}}$ such that $D\strictlyincluded E_1^{1,n}$ for all $n>n_D^1.$
\end{itemize}

{\it Substep 2: Iterative argument.} Repeating Substep 1 and applying Proposition \ref{prop:good_division} inductively in $N=1,2,\dots,$ with $A_k:=\tilde B_{N,k},$ $I_1:=\{1,\ldots,N\}$ and $I_2:=I_n^N$ for $n\in\N,$  we obtain  $\{\tilde B_{N+1,n}\}_n\subset \fA_m$ and and increasing sequence $n\in \N \mapsto k_n^{N+1}$ with $\{k_n^{N}\}_n\supset \{k_n^{N+1}\}_n$ such that for any $N\ge1:$
\begin{itemize}
\item[$a_{Nn}$:\,\,] $\p A_{k_n^N}\subset \p \tilde B_{N,n},$  $\p \tilde B_{N,n}\subset \p \tilde B_{N+1,n}$  and $\cH^1(\p \tilde B_{N+1,n}\setminus \p \tilde B_{N,n})<2^{-(N+1)n}$ for any $n\ge1;$

\item[$b_{Nn}$:\,\,] $\sdist(\cdot,\p \tilde B_{N,n})\to \sdist(\cdot,\p A)$ as $n\to\infty$ locally uniformly in $\R^2;$

\item[$c_{Nn}$:\,\,] for any connected open set  $D\strictlyincluded E_i$ for some $i\in\{1,\ldots,N\}$ there exist $n_D^i>1$ and a unique connected component denoted by $E_i^{N,n}$ of $\Int{\tilde B_{N,n}}$ such that $D\strictlyincluded E_i^{N,n}$ for all $n>n_D^i.$
\end{itemize}

 By condition $b_{Nn}$ in Substep 2  and by the uniform boundedness of $\{\tilde B_{N,n}\},$ there exists an increasing sequence $N\in \N\mapsto n_N\in\N$ such that the sequence 
$
\tilde B_N:=\tilde B_{N,n_N}
$ 
satisfies $\sdist(\cdot,\p \tilde B_N)\to \sdist(\cdot,\p A)$ as $N\to\infty$ locally uniformly in $\R^2.$
By condition $a_{Nn_N}$ of Substep 2, 
$$
\p A_{k_{n_N}^N}\subset \p \tilde B_{1,n_N}\subset\ldots \subset \p \tilde B_{N,n_N}=\p \tilde B_N  
$$
and 
$$
\cH^1(\p \tilde B_N\setminus \p A_{k_{n_N}^N}) \le \sum\limits_{i=1}^N \cH^1(\p \tilde B_{i,n_N}\setminus \p \tilde B_{i-1,n_N}) \le \sum\limits_{i=1}^N 2^{-in_N}<2^{1-n_N},
$$
where $\tilde B_{0,n_N}:=A_{k_{n_N}^N}.$ 

Furthermore, given $i\in \N,$ if $D\strictlyincluded E_i$ is any connected open set, then by condition $c_{Nn_N},$  there exists a unique connected component $\tilde E_i^N:=E_i^{N,n_N}$ of $\Int{\tilde B_N}$ such that $D\strictlyincluded E_i^N$ for all sufficiently large $N$ (depending only $D$ and $i$). Moreover, it is clear that  $|\tilde B_N|=|A_{k_{n_N}^N}|$ for any $N.$
Hence, the sequence $\{\tilde B_N\}_N$ and the subsequence $\{A_{k_{n_N}^N}\}_N$ satisfy assertions (a)-(d). 

{\it Step 2: Construction of $\{\hat B_l\}$ and $\{A_{k_l}\}$ satisfying (a)-(e).}  Notice that $\Int{\tilde B_N}\subset\Int{A_{k_{n_N}^N}}$ and by $\tilde B_N\overset{\tau_\fA}{\to} A$ and Lemma \ref{lem:f_convergence} (b),
$
\lim\limits_{N\to\infty} |\tilde B_N\Delta A|\to0.
$ 
 In particular, for any $i,$  
\begin{equation}\label{volume_get}
\lim\limits_{N \to\infty} |\tilde E_i^N\Delta E_i|=0.
\end{equation}
By the Area Formula applied with $\dist(\cdot, E_i)$ we have 
$$
|\tilde E_i^N \setminus E_i| = \int_0^\infty \cH^1\big((\tilde E_i^N\setminus E_i)\cap \{\dist(\cdot,E_i) =t\}\big)dt 
$$
for any $i.$ From this, \eqref{volume_get} and a diagonal argument, there exists a not relabelled subsequence $\{\tilde B_N\}$ for which 
$$
\lim\limits_{N\to\infty} \cH^1\big((\tilde E_i^N\setminus E_i)\cap \{\dist(\cdot,E_i) =t\}\big) =0
$$ 
for any $i$ and a.e.\ $t>0.$ Thus, we can choose $t_s\searrow0$ for which $$
\lim\limits_{N\to\infty}  \cH^1\big((\tilde E_i^N\setminus E_i)\cap \{\dist(\cdot,E_i) =t_s\}\big) =0
$$ 
for any $s\in\N$ and $i,$ and thus, by a diagonal argument we find a further subsequence $\{\tilde B_{N_s}\}_{s\in\N}$ such that 
\begin{equation}\label{hstezr}
\cH^1\big((\tilde E_i^{N_s} \setminus E_i)\cap \{\dist(\cdot,E_i) =t_s\}\big)<2^{-is} 
 \end{equation}
for any $i$ and $s$.
Let 
$
\zeta_i^s:=(\tilde E_i^{N_s}\setminus E_i)\cap \{\dist(\cdot,E_i) =t_s\},
$ 
and let 
$$
\hat B_s:=\tilde B_{N_s}\setminus \zeta_s,
$$ 
where 
$
\zeta_s:=\bigcup\limits_{i} \zeta_i^s.
$ Note that $\zeta_s$ is $\cH^1$-rectifiable and by \eqref{hstezr}, $\cH^1(\zeta_s) \le 2^{1-s}.$
Denote by $\hat E_i^s$ the connected component  of $\hat B_s$ satisfying $\hat E_i^s\subset \tilde E_i^{N_s}$ and $\hat E_i^s\cap E_i \ne\emptyset.$ 
By construction, $\sup\limits_{x\in \hat E_i^s\setminus E_i} \dist(x, E_i)\le t_s,$ thus,
\begin{equation}\label{hsgtare}
\limsup\limits_{s \to\infty}\sup\limits_{x\in \hat E_i^s\setminus E_i} \dist(x, E_i)=0,   
\end{equation}
and since  $\p A_{k_{n_{N_s}}^{N_s}}\subset \p \tilde B_{N_s}\subset\p\hat B_s,$
and
$$
\limsup\limits_{s\to\infty} \cH^1(\p \hat B_s\setminus \p A_{k_{n_{N_s}}^{N_s}}) \le \limsup\limits_{s\to\infty} \Big(\cH^1(\zeta_s) +\cH^1(\p \tilde B_{N_s}\setminus \p A_{k_{n_{N_s}}^{N_s}})\Big)=0. 
$$
Moreover, since $\p\Omega\cap (\p \hat E_i^s\setminus \p E_i)\subset \{0<\dist(\cdot,E_i)<t_s\}$ for any $s$ and $i,$ we have  
\begin{equation}\label{sgehr}
\limsup\limits_{s\to\infty} \cH^1(\p\Omega\cap (\p \hat E_i^s\setminus \p E_i)) \le \lim\limits_{s\to\infty} \cH^1(\p\Omega\cap \{0<\dist(\cdot,E_i)<t_s\} ) =0 
\end{equation}
for any $i$ since $t_s\searrow 0.$ If $D\strictlyincluded E_i,$ then $D\strictlyincluded \hat E_i^s$ provided  that $s$ is large. This, and the relations $\R^2\setminus \cl{\tilde B_{N_s}} =\R^2\setminus \cl{\hat B_s}$ and $\Int{\hat B_s}\subset \Int{\tilde B_{N_s}}$ imply the local uniform convergence of $\sdist(\cdot,\p \hat B_s)$ to $\sdist(\cdot,\p A)$ in $\R^2.$ 
Thus, $\{\hat B_s\}$ and $\{A_{k_{n_{N_s}}^{N_s}}\}$ satisfy (a)-(e).

{\it Step 3: Construction of $\{B_l\}$ and $\{A_{k_l}\}$ satisfying (a)-(f).} Consider  $C_s:=\Int{\hat B_s}\setminus \bigcup_i\hat E_i^s.$ Since $|E_i^s\Delta E_i|\to0$ and $|\Int{\hat B_s}\Delta\Int{A}|\to0$ as $s\to\infty,$ we have $|C_s|\to0.$ 
Therefore, applying the Area Formula with $\dist(\cdot,\substrate),$ we have 
$$
|C_s| = \int_0^\infty \cH^1(C_s\cap \{\dist(\cdot,\substrate)=t\})dt
$$
so that, passing to further not relabelled subsequence if necessary, we can choose $t_s'\in(0,d_0/4)$  such that
$\lim\limits_{s\to\infty} \cH^1(C_s\cap \{\dist(\cdot,\substrate)=t_s'\})=0,$
where $d_0$ is the minimal distance between connected components of $\substrate.$ Now  the sequence
$$
B_s:=\hat B_s\setminus (C_s\cap \{\dist(\cdot,\substrate)=t_s'\}) 
$$
and the subsequence $\{A_{k_{n_{N_s}}^{N_s}}\}$ satisfy all assertions of the proposition.
\end{proof}

\begin{proposition}\label{prop:yo_cheksiz_yo_finite}
Let $\anyset\subset\R^n$ be a connected open set and $\{u_k\}\subset H^1_\loc(\anyset;\R^n)$ be  such that 
\begin{equation}\label{finite_elassddff}
\sup\limits_k \int_\anyset |\str{u_k}|^2dx<+\infty. 
\end{equation}
Then either $|u_k|\to\infty$ a.e.\ in $\anyset$ or there exist $u\in H_\loc^1(\anyset ;\R^n)\cap GSBD^2(\anyset ;\R^n)$ and a subsequence $\{u_{k_l}\}$ such that $u_{k_l} \wk u$ in $H^1_\loc(\anyset;\R^n),$ and hence, $u_{k_l}\to u$ a.e.\ in $\anyset.$ 
\end{proposition}

\begin{proof}
Indeed, suppose that there exists 
a ball $B_\epsilon\strictlyincluded \anyset ,$ a measurable function $\tilde u: B_\epsilon\to\R^n$ and a not relabelled subsequence $\{u_k\}$ such that $u_k\to \tilde u$ a.e.\ in some subset $E$ of $B_\epsilon$ with positive measure. Since $u_k\in H^1(B_\epsilon;\R^n),$ by the Poincar\'e-Korn inequality, there exists a rigid displacement $a_k:\R^n\to\R^n$ such that 
$$
\|u_k + a_k\|_{H^1(B_\epsilon)}^2 \le C\int_{B_\epsilon} |\str{u_k}|^2dx
$$
for some $C>1$ independent of $k.$ In particular, by the Rellich-Kondrachov Theorem, there exists $v\in H^1(B_\epsilon;\R^n)$ such that  $u_k +a_k \wk v$  in $H^1(B_\epsilon;\R^n)$ (up to a subsequence) and a.e.\ in $B_\epsilon.$ Since $u_k\to \tilde u$ a.e.\ in $E,$ $a_k\to  v - \tilde u$ a.e.\ in $E$ as $k\to \infty.$ Thus, $v - \tilde u$ is a restriction in $E$ of some rigid displacement $a:\R^n\to\R^n.$ By linearity of rigid displacements, $a_k\to a$ pointwise in $\R^n.$   Therefore, $u_k \wk v-a$  in $H^1(B_\epsilon;\R^n),$ hence  a.e.\ in $B_\epsilon.$
In view of \eqref{finite_elassddff}, $\{u_k\}\subset GSBD^2(\anyset ;\R^n)$ with $J_{u_k}=\emptyset.$ Hence, by \cite[Theorem 1.1]{ChC:2018jems}, there exist  a  further not relabelled subsequence $\{u_k\}$  for which the set 
$$
F:= \{x\in \anyset :\,\,|u_k(x)|\to \infty\}
$$
has a finite perimeter in $\Omega$ and $u\in GSBD^2(\anyset ;\R^n)$ such that 
$
u_k\to u 
$ 
a.e.\ in $\anyset \setminus F$ and 
\begin{equation}\label{adafwe}
\cH^{n-1}(J_u\setminus \p^*F) + \cH^{n-1}(\anyset \cap\p^*F) \le \liminf\limits_{k\to \infty} \cH^1(J_{u_k}) =0. 
\end{equation}
Thus, $P(F,\anyset )=0,$ i.e., either $F=\emptyset$ or $F=\anyset .$  
Since $u_k\to u=v-a$ a.e.\ in $B_\epsilon\subset \anyset ,$ the case $F=\anyset $ is not possible. Thus, $F=\emptyset.$
By \eqref{adafwe}, $\cH^1(J_u)=0.$ 

Now we show that $u_k\wk u$  in $H^1_\loc(\anyset ;\R^n)$ and $u\in H^1_\loc(\anyset ;\R^n).$
Let $D_1\strictlyincluded D_2\strictlyincluded \ldots$ be an increasing sequence of connected Lipschitz open sets such that $D_1:=B_\epsilon$ and $\anyset =\cup_j D_j.$ Applying Poincar\'e-Korn inequality $D_j$ we find a rigid displacement $a_k^j$ such that 
$$
\|u_k + a_k^j\|_{H^1(D_j)}^2 \le c_j\int_{D_j} |\str{u_k}|^2dx,
$$
where $c_j$ is independent on $k.$ Then by the Rellich-Kondrachov Theorem, every subsequence $\{u_{k_l}\}$ admits further not relabelled subsequence such that $u_{k_l} + a_{k_l}^j \wk v$  in $H^1(D_j;\R^n)$ and a.e.\ in $D_j$ for some $v\in H^1(D_j;\R^n).$ Since $u_{k_l} \to u$ a.e.\ in $D_j,$ it follows that $a_{k_l}^j \to v-u$ a.e.\ in $D_j$ and hence, $v-u$ is also a rigid displacement. Since a.e.-convergence of linear functions implies the local strong $H^1$-convergence,  $u_{k_l}\wk u$ in $H^1(D_j;\R^n),$ and thus, $u\in H^1(D_j;\R^n).$ Since the subsequence $\{u_{k_l}\}$ is arbitrary, $u_k\wk u$  in $H^1(D_j;\R^n).$ By the choice of $D_j,$ $u_k\wk u$  in $H^1_\loc(\anyset ;\R^n)$ and $u\in H^1_\loc(\anyset ;\R^n).$
%
\end{proof}


The following corollary of Proposition \ref{prop:yo_cheksiz_yo_finite} is used in the proof of Theorem \ref{teo:compactness_Ym}.

\begin{corollary}\label{cor:add_rig_converge}
Let $\openset,\openset_k\subset\R^n$ be  connected bounded open sets such that for any $G\strictlyincluded \openset$ there exists $k_G$ such that $G\strictlyincluded \openset_k$ for all $k>k_G,$ and let $u_k\in H^1_\loc(\openset_k;\R^n)$ be  such that 
\begin{equation}\label{finite_elassddff111}
\sup\limits_k \int_{\openset_k} |\str{u_k}|^2dx<\infty. 
\end{equation}
Then there exist $u\in H^1_\loc(\openset;\R^n)\cap GSBD^2(\openset;\R^n),$ a subsequence $\{(\openset_{k_l},u_{k_l})\}$ and  a sequence $\{\red b_l^P\black \}$ of rigid displacements such that  $u_{k_l} +\red b_l^P\black \to u$ a.e.\ in $\openset.$ 
\end{corollary}

\begin{proof}
Let $B_\epsilon\strictlyincluded \openset$ be any ball. 
By assumption, $B_\epsilon\strictlyincluded \openset_k$ for all large $k.$
By the Poincar\'e-Korn inequality, for all such $k$ there exists a rigid displacement $ \red b_k^\epsilon \black $ such that 
$$
\|u_k +  \red b_k^\epsilon \black \|_{H^1(B_\epsilon)}^2 \le C_\epsilon  \int_{B_\epsilon} |\str{u_k}|^2dx.
$$ 
This, \eqref{finite_elassddff111} and the Rellich-Kondrachov Theorem imply that there exist  a not relabelled subsequence $\{u_k+ \red b_k^\epsilon \black \}$ and $v\in H^1(B_\epsilon;\R^n)$ such that $u_k+ \red b_k^\epsilon \black  \wk v$ \red as $k\to\infty$ \black in $H^1(B_\epsilon;\R^n),$ hence, a.e.\ in $B_\epsilon.$  Now applying 
Proposition \ref{prop:yo_cheksiz_yo_finite} with an increasing sequence $\{G_i\}$ of connected open sets satisfying $G_1=B_\epsilon,$ $G_i\strictlyincluded \openset$ and $\openset=\cup_i G_i$ we find $u\in H_\loc^1(\openset;\R^n)\cap GSBD^2(\openset;\R^n)$ with $u=v$ in $B_\epsilon$ and a not relabelled subsequence $\{u_k+ \red b_k^\epsilon \black \}$ such that $u_k+ \red b_k^\epsilon \black  \to u$ \red as $k\to\infty$ \black a.e.\ in $\openset.$
\end{proof}

\begin{proposition}\label{prop:jump_estimate}
Assume (H1)-(H2) and let $x_0\in\Sigma,$ $\delta\in(0,\frac12)$ and $r\in(0,1)$ be such that $\nu_0:=\nu_\Sigma(x_0)$ exists,   
\begin{equation}
|\varphi(y,\xi)- \varphi(x_0,\xi)|<\delta \label{phi_contin} 
\end{equation}
for any $y\in U_{r,\nu_0}(x_0)$ and $\xi\in \S^1,$
$U_{r,\nu_0}(x_0) \cap \Sigma$ is a graph of a Lipschitz function over tangent line $U_{r,\nu_0}(x_0) \cap T_{x_0}$ in direction $\nu_0$ 
and 
\begin{equation}
 \int_{U_{r,\nu_0}(x_0) \cap \Sigma} |\beta(y) -\beta(x_0)|d\cH^1<\delta\cH^1(U_{r,\nu_0}(x_0)\cap \Sigma). \label{beta_continu}
\end{equation}
Let $A\in\fA_m$ be such that $x_0\in \Sigma\cap \p^*A,$  $U_{r,\nu_0}(x_0)\cap \{\dist(\cdot,T_{x_0})\ge \delta r  \}\subset \Int{A}\cup\substrate,$ 
and let  $\{(A_k,u_k)\}\subset\admissible_m$ and  $u\in H_\loc^1(\Int{A};\R^2)$ be such that 
$A_k\overset{\tau_\fA}\to A$ and 
$$ 
\sup\limits_{k} \int_{U_{r,\nu_0}(x_0) \cap (A_k\cup\substrate)} |\str{u_k}|^2dx + \cH^1(U_{r,\nu_0}(x_0) \cap \p A_k) <\infty  
$$  
and  $u_k\to u$ a.e.\ in $U_{r,\nu_0}(x_0) \cap \Int{A}$ and $|u_k|\to+\infty$ a.e.\ in $\substrate\cap U_{r,\nu_0}(x_0).$ 
Then  there exists $k_\delta>1$ for which 
\begin{align} \label{before_blow_up}
& \int_{U_{r,\nu_0}(x_0)\cap \Omega \cap\p^*A_k} \varphi(x,\nu_{A_k})d\cH^1  
+2\int_{U_{r,\nu_0}(x_0) \cap \Omega  \cap (A_k^{(0)} \cup A_k^{(1)}) \cap\p A_k}
\varphi(x,\nu_{A_k})d\cH^1\nonumber \\  
&+ \int_{U_{r,\nu_0}(x_0) \cap \Sigma \cap A_k^{(0)}\cap \p A_k} \big(\varphi(x,\nu_\Sigma)
+ \beta \big)d\cH^1\nonumber \\
& +  \int_{U_{r,\nu_0}(x_0)\cap \Sigma \cap \p^*A_k\setminus J_{u_k}} \beta d\cH^1  +
\int_{U_{r,\nu_0}(x_0) \cap J_{u_k}} \varphi(x,\nu_\Sigma)\,d\cH^1 \no\\ 
&\ge \frac{1}{1+\frac{\delta}{c_2}}
\int_{U_{r,\nu_0}(x_0) \cap \Sigma\cap \p^*A} \varphi(x,\nu_\Sigma) \,d\cH^1 - \delta\,\int_{U_{r,\nu_0}(x_0) \cap \Sigma}\varphi(x,\nu_\Sigma)d\cH^1. 
\end{align} 
for any $k>k_\delta.$
\end{proposition}

We postpone the proof after  the following lemma.
\begin{lemma}\label{lem:jump_estaaa}
Let $\phi$ be a norm in $\R^2,$  $A\in \fA_m$ be such that $0\in \Sigma\cap \p^*A,$  $U_r \cap \{\dist(\cdot,\{x_2=0\})\ge \frac{r}{2} \}\subset \Int{A}\cup\substrate,$  and  $\{(A_k,u_k)\}\subset\admissible_m,$ and $u\in H^1_\loc(\Int{A};\R^2)$ be such that
\begin{equation}\label{jumapsoa}
\sup\limits_{k} \int_{U_r  \cap A_k} |\str{u_k}|^2dx + \cH^1(U_r  \cap \p A_k) <\infty 
\end{equation}
and $A_k\overset{\tau_\fA}\to A$ and  $u_k\to u$ a.e.\ in $U_r \cap \Int{A}$ and $|u_k|\to+\infty$ a.e.\ in $\substrate\cap U_r.$  Then for every $\epsilon>0$  there exists $k_\epsilon >0$ such that for any $k>k_\epsilon ,$
\begin{align}\label{alfa_1nin_ulch}
\int_{U_r \cap \Omega \cap\p^*A_k} \phi(\nu_{A_k})d\cH^1 
+2\int_{U_r \cap \Omega  \cap A_k^{(1)}  \cap\p A_k}
\phi(\nu_{A_k})d\cH^1  \no\\
\ge 2\int_{U_r \cap \Sigma\cap (\p^* A_k \setminus J_{u_k})} \phi(\nu_\Sigma)d\cH^1 -\epsilon .
\end{align}
\end{lemma}

\begin{proof}
Since $(A_k^{(1)}\cap\p A_k)\cup J_{u_k}$ is $\cH^1$-rectifiable, by \cite[pp. 80]{AFP:2000} there exists at most countably many $C^1$-curves $\{\Gamma_i^k\}_{i\ge1}$ such that 
$$
\cH^1\Big(((A_k^{(1)}\cap\p A_k)\cup J_{u_k})\setminus \bigcup\limits_{i\ge1} \Gamma_i^k\Big) =0.
$$
Selecting closed arcs inside curves if necessary, we suppose that $\Gamma_i^k\subset U_r $ and 
$$
\int_{U_r \cap A_k^{(1)}\cap\p A_k} \phi(\nu_{A_k})d\cH^1 + \int_{U_r \cap J_{u_k}} \phi(\nu_{A_k})d\cH^1 +\epsilon > \sum\limits_{i\ge1} \int_{\Gamma_i^k} \phi(\nu_{\Gamma_i^k})d\cH^1
$$
for any $k.$ Since each $\Gamma_i^k$ is $C^1,$ we can choose a Lipschitz open set $V_i^k\subset U_r $ such that $\Gamma_i^k\subset \cl{V_i^k},$  $|V_i^k|\le 2^{-i-1-k},$ 
$$
\int_{\p V_i^k} \phi(\nu_{V_i^k})d\cH^1 < 2\int_{\Gamma_i^k} \phi(\nu_{\Gamma_i^k})d\cH^1 + \frac{\epsilon }{2^{i+1}}
$$
and $\dist_\cH(\Gamma_i^k,\p V_i^k)<2^{-k},$ where $\dist_\cH$ is the Hausdorff distance (see e.g., \eqref{haus_dista} for the definition). Let $V_0^k:=U_r \setminus \cl{\Int{A_k}\cup\substrate}$ be the ``voids''. 
By the definition of $\{V_i^k\},$
\begin{align}\label{mucci_bollarni_yaxshi_kuraman}
\int_{U_r \cap \Omega \cap\p^*A_k} \phi(\nu_{A_k})d\cH^1 + \int_{U_r \cap (\Sigma \setminus \p^*A_k)} \phi(\nu_\Sigma) d\cH^1 + 2 \int_{U_r \cap J_{u_k}} \phi(\nu_{A_k})d\cH^1 \no \\
+ 2 \int_{U_r \cap \Omega  \cap A_k^{(1)}  \cap\p A_k} \phi(\nu_{A_k})d\cH^1  
\ge \sum\limits_{i\ge0} \int_{\p V_i^k} \phi(\nu_{\p V_i^k}) d\cH^1 - \frac{\epsilon }{2}. 
\end{align}
In particular, by \eqref{jumapsoa}, $\sup \limits_k \sum\limits_{i\ge0} \cH^1(\p V_i^k) <\infty,$ and hence, by \cite[Proposition 2.6]{Ma:2012}, there exists $\xi\in\R^2$  such that 
the set $\big\{x\in \bigcup_i\p V_i:\,\,\tr_{U_r \setminus \cup V_i^k} (u)(x) = \xi\big\}$ is $\cH^1$-negligible.
Define 
$$
w_k:= u_k\chi_{U_r \setminus \bigcup_i V_i^k} + \xi \chi_{\bigcup_i V_i^k}.
$$
Then $w_k\in GSBD^2(U_r ;\R^2),$ $J_{w_k}=\bigcup_i\p V_i^k$ and by \eqref{jumapsoa},
$$
\sup\limits_k \int_{U_r } |\str{w_k}|^2dx + \cH^1(J_{w_k}) <\infty.
$$
Since $\sum\limits_{i\ge1}|V_i^k| \le 2^{-k},$ by assumption on $\{u_k\}$ and $\{A_k\},$
$$
w_k\to 
\begin{cases}
u & \text{a.e.\ in $U_r\cap \Int{A},$ }\\
\xi & \text{a.e.\ in $\Omega\cap U_r\cap \setminus\cl{A}$}
\end{cases}
$$
and $|w_k|\to+\infty$ a.e.\ in $U_r\cap \substrate.$ 

We show that 
\begin{equation}\label{liminf_katta_sasa}
2\int_{U_r \cap \Sigma} \phi(\nu_\Sigma)d\cH^1 \le \liminf\limits_{k\to\infty} \int_{J_{w_1}} \phi(\nu_{J_{w_1}})d\cH^1. 
\end{equation}
By assumption, $U_r \cap \Sigma \subset (-1,1)\times(-\epsilon ,\epsilon )$ and $U_r \cap \p A \subset (-1,1)\times(-\epsilon ,\epsilon ),$ thus, by the convergence $A_k\overset{\tau_\fA}{\to} A$ and Remark \ref{rem:kuratiwski_and_sdistance}, $U_r \cap \p A_k\subset (-1,1)\times(-\epsilon ,\epsilon )$ for all large $k.$ In particular, for such $k,$ $J_{w_k}\subset (-1,1)\times(-\epsilon ,\epsilon ).$

Under the notation of \cite{ChC:2018jems}, given $\xi\in\S^1$ let $\pi_\xi$ be the orthogonal projection onto the line $\Pi_\xi:=\{\eta\in\R^2:\,\,\xi\cdot\eta=0\},$ perpendicular to $\xi;$ given a Borel set $F\subset\R^2$ and $y\in \Pi_\xi,$ let $F_y^\xi:=\{t\in\R:\,\, y+t\xi \in F\}$ be the one-dimensional slice of $F,$   and given $u\in GSBD(U_r ;\R^2)$ and $y\in \Pi_\xi,$ let $\hat u_y^\xi(t) = u(y+t\xi)\cdot \xi$ be the one-dimensional slice of $u.$
Since $w_h\to w$ a.e.\ in $U_r\setminus\substrate ,$ by \cite[Eq. 3.23]{ChC:2018jems}, for any $\epsilon>0$ and Borel set $F\subset U_r ,$
\begin{equation}\label{mucci_bolam_mening1}
\cH^0((F\setminus \substrate)_y^\xi\cap J_{\hat w_y^\xi}) + \cH^0((U_r \cap \Sigma)_y^\xi) \le\liminf\limits_{k\to\infty} \Big(\cH^0(F_y^\xi\cap J_{(\hat w_k)_y^\xi}) + \epsilon f_y^\xi(w_k)\Big)  
\end{equation}
for a.e.\ $\xi\in\S^1$ and a.e.\ $y\in \Pi_\xi,$ where the integral of $f_y^\xi(w_k)$ over $\Pi_\xi$ is uniformly bounded independent on $\xi$ and $k$ (see also \eqref{residiea} below).

Let 
$$
\hat A:=\Big\{y\in \pi_\xi(U_r \cap \Sigma):\,\, \liminf\limits_{k\to\infty} \cH^0(J_{(\widehat w_k)_y^\xi})=0\Big\}\subset \Pi_\xi. 
$$
Then $F:=U_r \cap \pi_\xi^{-1}(\hat A)$ is Borel and, thus, integrating \eqref{mucci_bolam_mening1} over $\hat A$ and using the definition of $\hat A$ and Fatou's Lemma we get 
$$
\cH^1\big(\pi_\xi(U_r \cap \Sigma) \cap \hat A\big) \le \liminf\limits_{k\to\infty}  \epsilon \int_{\hat A}f_y^\xi(v_k)dy\le M\epsilon  
$$
for some $M>0$ independent of $\epsilon.$ Thus, letting $\epsilon\to0$ we get $\cH^1(\hat A)=0.$ In particular, 
\begin{equation}\label{proj_covers_all1}
\limsup\limits_{k\to\infty} \cH^1\big(\pi_\xi(U_r \cap\Sigma)\setminus \pi_\xi(J_{w_k})\big)=0.  
\end{equation}
Note that by construction, $J_{w_k}$ is a union of open sets, thus, for a.e.\ $y\in \pi_\xi(J_{w_k}),$ the line $\pi_\xi^{-1}(y)$ passing through $y$ and parallel to $\xi$  crosses $J_{w_k}$ at least at two points.  Thus,
\begin{equation}\label{ahsssaaa1}
\cH^0\big(J_{(\widehat w_k)_y^\xi}\big)\ge 2= 2\cH^0\big((U_r \cap \Sigma)_y^\xi \big) 
\end{equation}
for $\cH^1$-a.e.\ $y\in \pi_\xi(J_{w_k})\cap \pi_\xi(U_r \cap\Sigma),$ where $o(1)\to0$ as $k\to\infty.$ Now we  choose arbitrary pairwise disjoint open sets $F_1,F_2,\ldots\strictlyincluded U_r $ and repeating the same argument of Step 1 in the proof of Proposition \ref{prop:anisotropic_chambolle} (by using \eqref{ahsssaaa1} in place of \eqref{chambolledan_lavha}
and using \eqref{proj_covers_all1}) we obtain \eqref{liminf_katta_sasa}. 

From \eqref{liminf_katta_sasa} and \eqref{mucci_bollarni_yaxshi_kuraman} it follows that there exists $k_\epsilon >0$ such that 
\begin{align}\label{alfa_1n_ch1}
\int_{U_r \cap \Omega \cap\p^*A_k} \phi(\nu_{A_k})d\cH^1 + \int_{U_r \cap (\Sigma \setminus \p^*A_k)} \phi(\nu_\Sigma) d\cH^1 + 2 \int_{U_r \cap J_{u_k}} \phi(\nu_{A_k})d\cH^1 \no \\
+ 2 \int_{U_r \cap \Omega  \cap A_k^{(1)}  \cap\p A_k} \phi(\nu_{A_k})d\cH^1  
\ge 2\int_{U_r \cap \Sigma} \phi(\nu_\Sigma)d\cH^1 - \epsilon 
\end{align}
for any $k>k_\epsilon .$ Now \eqref{alfa_1nin_ulch} follows from \eqref{alfa_1n_ch1}. 
\end{proof}

We anticipate here that in Lemma \ref{lem:creation_of_delamination} below we establish a similar result.

\begin{proof}[Proof of Proposition \ref{prop:jump_estimate}]
For simplicity, assume that $x_0=0,$ $\nu={\bf e_2}$  and $\phi(\xi) = \varphi(0,\xi).$  
Denote the left-hand side of \eqref{before_blow_up}  by $\alpha_k.$ 
By \eqref{phi_contin} and \eqref{hyp:bound_anis},
\begin{align} \label{alpha_es}
\alpha_k\ge & \hat \alpha_k - 2\delta\cH^1(U_r\cap \Omega\cap \p A_k), 
\end{align} 
where
\begin{align*}
\hat \alpha_k:=&\int_{U_r \cap \Omega \cap\p^*A_k} \phi(\nu_{A_k})d\cH^1 
+2\int_{U_r \cap \Omega  \cap (A_k^{(0)}\cup A_k^{(1)}) \cap\p A_k}
\phi(\nu_{A_k})d\cH^1\nonumber \\  
&  +  \int_{U_r \cap \Sigma \cap \p^*A_k\setminus J_{u_k}} \beta d\cH^1+
\int_{U_r \cap J_{u_k}} \varphi(x,\nu_\Sigma)\,d\cH^1. 
\end{align*}
By Lemma \ref{lem:jump_estaaa} applied with $\phi$ and $\epsilon:=\delta \int_{U_r\cap \Sigma}\varphi(x,\nu_\Sigma)d\cH^1$, there exists $k_\delta$ such that  
$$
\hat \alpha_k\ge 
2\int_{U_r \cap \Sigma\cap (\p^* A_k \setminus J_{u_k})} \phi(\nu_\Sigma)d\cH^1
+  \int_{U_r \cap \Sigma \cap \p^*A_k\setminus J_{u_k}} \beta d\cH^1+
\int_{U_r \cap J_{u_k}} \phi(\nu_\Sigma)\,d\cH^1  -\epsilon   
$$
for all $k>k_\delta.$ Then by \eqref{phi_contin}
$$
\begin{aligned}
\int_{U_r \cap \Sigma\cap (\p^* A_k \setminus J_{u_k})} \phi(\nu_\Sigma)d\cH^1 \ge 
\int_{U_r \cap \Sigma\cap (\p^* A_k \setminus J_{u_k})} \varphi(x,\nu_\Sigma)d\cH^1\\
-
\delta\cH^1(U_r \cap \Sigma\cap (\p^* A_k \setminus J_{u_k})),
\end{aligned}
$$
and therefore, 
$$
\begin{aligned}
\hat \alpha_k\ge 
\int_{U_r \cap \Sigma\cap (\p^* A_k \setminus J_{u_k})} (2\varphi(x,\nu_\Sigma) 
+  \beta) d\cH^1 + \int_{U_r \cap J_{u_k}} \phi(\nu_\Sigma)\,d\cH^1  \\
-\epsilon -  
\delta\cH^1(U_r \cap \Sigma\cap (\p^* A_k \setminus J_{u_k}))
\end{aligned}
$$
Applying \eqref{hyp:bound_anis} in the first integral we get  
$$
\begin{aligned}
\hat \alpha_k\ge 
\int_{U_r \cap \Sigma\cap \p^* A_k} \varphi(x,\nu_\Sigma) d\cH^1
-\epsilon -  \delta\cH^1(U_r \cap \Sigma\cap \p^* A_k)
\end{aligned}
$$
so that 
\begin{equation}\label{asasdada} 
\alpha_k \ge \int_{U_r \cap \Sigma\cap \p^* A_k} \varphi(x,\nu_\Sigma) d\cH^1
- \epsilon -  2\delta\cH^1(U_r \cap \p A_k). 
\end{equation}
By \eqref{finsler_norm} and \eqref{hyp:bound_anis}
$$
c_1 \cH^1(U_r \cap \p A_k) \le \alpha_k+ \int_{U_1\cap \Sigma}\varphi(x,\nu_\Sigma)d\cH^1,
$$
and hence, \eqref{asasdada}  and the definition of $\epsilon$ imply
$$
\Big(1+ \frac{\delta}{c_1}\Big)\,\alpha_k \ge \int_{U_r \cap \Sigma\cap \p^* A_k} \varphi(x,\nu_\Sigma) d\cH^1 - \delta\Big(1+ \frac{\delta}{c_1}\Big)\int_{U_r\cap \Sigma}\varphi(x,\nu_\Sigma)d\cH^1. 
$$
and \eqref{before_blow_up} follows.
\end{proof}

Finally we prove compactness of $\admissible_m.$

\begin{proof}[Proof of Theorem \ref{teo:compactness_Ym}]

Let $R:=\sup\limits_k \cF(A_k,u_k)$ and, by passing to a further not relabelled subsequence if necessary, we assume that 
$$
\liminf\limits_{k\to\infty} \cF(A_k,u_k) = \lim\limits_{k\to\infty} \cF(A_k,u_k). 
$$
By (H1)-(H3) we have 
$$
\sup\limits_k \Big(c_1\cH^1(\Omega\cap\p A_k) + 2c_3\int_{A_k\cup\substrate} |\str{u_k}|^2dx\Big) \le R + \int_\Sigma|\beta|d\cH^1 
$$
and hence,
\begin{equation}\label{ajsjakda}
\cH^1(\p A_k) \le \cH^1(\Omega\cap\p A_k) + \cH^1(\p\Omega\cap \p A_k) \le \frac{R +\int_\Sigma |\beta|d\cH^1}{c_1} +\cH^1(\p\Omega) 
\end{equation}
and 
\begin{equation}\label{dafefafa}
\int_{A_k\cup\substrate} |\str{u_k}|^2dx \le  \frac{R +\int_\Sigma |\beta|d\cH^1}{2c_3}  
\end{equation}
for any $k\ge1.$ In view of \eqref{ajsjakda} and Proposition \ref{prop:compactness_A_m}, there exists $A\in \fA_m$ with $\cH^1(\p A)<\infty$ and a not relabelled subsequence $\{A_k\}$ such that
$\sdist(\cdot,\p A_k)\to\sdist(\cdot,\p A)$ locally uniformly in $\R^2.$ 
Now we construct the sequence $\{(B_n,v_n)\}$ in three steps. In the first step we apply  Proposition \ref{prop:sets_changhe} and Corollary 
\ref{cor:add_rig_converge} to obtain a (not relabelled) subsequence and  to construct a sequence $\{B_k\}\subset \fA_m$ with associated piecewise rigid displacements $\{a_k\}$ 
such that both $B_k\overset{\tau_\fA}{\to} A$ and $u_k+a_k\to u$ a.e. in $\Int{A} \cup\substrate$ for some $u\in H^1_\loc(\Int{A}\cup \substrate,\R^2)\cap GSBD^2(\Ins{A};\R^2).$ In the second step we take care of the fact that adding different rigid motions in $B_k$ and in $\substrate$ can create extra jump at $\Sigma$ making difficult to satisfy \eqref{liminfga_ut_eshmat}. More precisely, 
by Proposition \ref{prop:jump_estimate} we modify $\{B_k\}$ and $\{u_k\}$ so that the modified sequence  $\{(B_k^\delta,u_k^\delta)\}\subset\admissible_m$ satisfies \eqref{liminfga_ut_eshmat} with some small error of order $\delta>0.$ Finally, in Step 3 we construct the sequence $\{(D_n,v_n)\}\subset\admissible_m$ by means of $\{(B_k^\delta,u_k^\delta)\}$ and a diagonal argument.

{\it Step 1: Defining a first modification $\{B_k\}$ of $\{A_k\}$.} By Proposition \ref{prop:sets_changhe} there exist a not relabelled subsequence $\{A_k\}$ and a sequence $\{B_k\}\subset \fA_m$ such that 
\begin{itemize}
\item[(a1)] $\p A_k\subset \p B_k$ and $\lim\limits_{k\to \infty} \cH^1(\p B_k\setminus \p A_k)=0;$\\

\item[(a2)] $B_k\overset{\tau_\fA}{\to} A$ as $k\to \infty;$\\

\item[(a3)] if $\{E_i\}_{i\in I}$ is all connected components of $\Int{A},$ there exists connected components of $\Int{B_k}$ enumerated as  $\{E_i^k\}_{i\in I}$  such that for any $i$ and $G\strictlyincluded E_i$ one has $G\strictlyincluded E_i^k$ for all large $k$ (depending only on $i$ and $G$);\\

\item[(a4)] $\sum\limits_i \cH^1(\p\Omega\cap (\p E_i^k\setminus \p E_i))\to0$ as $k\to\infty;$\\

\item[(a5)] $|B_k|=|A_k|$ for all $k\ge1;$\\

\item[(a6)] the boundary of every connected component of $\Int{B_k}\setminus\bigcup_i E_i^k$ intersects the boundary of at most one connected component of $\substrate.$
\end{itemize}

Notice that by  condition (a1),
$$
\lim\limits_{k\to\infty}|\cS(A_k,u_k) - \cS(B_k,u_k)| =0
$$
and 
$$
\cW(A_k,u_k)=\red \cW\black (B_k,u_k).
$$
Thus,
\begin{equation}\label{first_attempt_b}
 \lim\limits_{k\to\infty}|\cF(A_k,u_k)- \cF(B_k,u_k)|=0. 
\end{equation}

Now we define the piecewise rigid displacements $a_k$ associated to $B_k.$  Let $\{S_j\}_{j\in Y}$ be the set of connected components of $\substrate$ for some index set $Y.$ 
\red We define the index sets $I_n\subset I$ and $Y_n\subset Y$ inductively on $n$ in such a way that  
Corollary \ref{cor:add_rig_converge} holds with $\openset_k=E_i^k$ and $\openset=E_i$ and also with $\openset_k=\openset=S_j$ for every $i\in I_n$ and $j\in Y_n$ with the same rigid displacements $a_k^n$ independent of $i$ and $j.$ 
%

More precisely, let $I_0:=Y_0:=\emptyset,$ and given the sets $I_1,\ldots, I_{n-1}$ and $Y_1,\ldots,Y_{n-1}$ for $n\ge1,$ we define $I_n$ and $Y_n$ as follows.  By Corollary \ref{cor:add_rig_converge} applied with $\openset_k=\openset=S_{j_n}$ with $j_n$ the smallest element  of $Y\setminus \bigcup\limits_{l=1}^{n-1} Y_l,$ \black  we find a not relabelled subsequence $\{(B_k,u_k)\},$ a sequence $\{a_k^n\}$ of rigid displacements and $w_n\in  H^1_\loc(S_{j_n};\R^2)$ such that $u_k+a_k^n\to w_n$ a.e.\ in $S_{j_n}.$  
Let $I_n$ and $Y_n$ be \red the sets \black such that there exists a not relabelled subsequence $\{(B_k,u_k)\}$ such that  the sequence $(u_k+a_k^n)\chi_{E_i^k}$    converges a.e.\ in $E_i$ for $i\in I_n$ 
and the sequence $(u_k+a_k^n)\chi_{S_j}$  converges a.e. in $S_j$ for $j\in Y_n.$ Recall that $j_n\in Y_n.$
Let 
$$
F_n^k:=\Big(\bigcup\limits_{i\in I_n} E_i^k\Big)\cup \Big(\bigcup\limits_{j\in Y_n} S_j\Big)\qquad\text{and}\qquad 
F_n:=\Big(\bigcup\limits_{i\in I_n} E_i\Big) \cup \Big(\bigcup\limits_{j\in Y_n} S_j\Big).
$$
By the definition of $I_n$ and $Y_n,$ \red and by diagonalization \black 
the sequence $(u_k+a_k^n)\chi_{F_n^k}$ converges \red as $k\to\infty$ \black  a.e.\  in $F_n$ to some function in $H^1_\loc(F_n;\R^2),$ which we  still denote by $w_n.$ 

Note that for large $n,$ $Y_n$ is empty since $Y$ is finite by assumption.
Notice also that by definition of $I_n$ and $Y_n,$ and Proposition \ref{prop:yo_cheksiz_yo_finite} applied in connected open sets $P\strictlyincluded E_i\cup S_j,$ we have $|u_k+a_k^n|\to +\infty$ a.e in $E_i\cup S_j$ for every $i\in I\setminus  I_n$ and  $j\in Y\setminus  Y_n,$  

We now define the rigid displacements in $E_i^k$ for $i\in I\setminus \bigcup\limits_n I_n.$ By  a diagonal argument and by Corollary \ref{cor:add_rig_converge} applied with $\openset_k=E_i^k$ and $\openset=E_i$ for any  $i\in I\setminus \bigcup\limits_n I_n,$ we find a further not relabelled sequence $\{B_k,u_k\},$ sequence $\{\tilde a_k^i\}$ of rigid displacements and $w^i\in H^1_\loc(E_i;\R^2)$  such that $(u_k+ \tilde a_k^i)\chi_{E_i^k} \to w^i$ a.e.\ $E_i$ as $k\to\infty.$

Finally, we define rigid displacements in connected components $C_i^k$ of $B_k\setminus\bigcup_i E_i^k$ whose interior in the limit  becomes empty, i.e., $C_i^k$ turns into an external filament. Recall that $|C_i^k|\to0$ as $k\to\infty.$  If   $\cH^1(\p C_i^k\cap\Sigma)=0,$ we define null-rigid displacement in $C_i^k.$  If $\cH^1(\p C_i^k\cap\Sigma)>0,$ then by condition (a6), $\p C_i^k$ intersects the boundary of unique $S_{j_i},$ in which we have defined rigid displacement $a_k^{j_i}.$ In this case we define the same $a_k^{j_i}$ in $C_i^k$ so that $\bigcup_i \p C_i^k\cap J_{u_k+a_k^{j_i}} \subset J_{u_k},$ i.e., we do not create extra jump set.

Let 
$$
a_k:= \sum\limits_{n} a_k^n \sum\limits_{i\in I_n,j\in Y_n} \chi_{E_i^k\cup S_j}^{} + \sum\limits_{i\in I\setminus \cup_n I_n} \tilde  a_k^i\chi_{E_i^k}^{} + \sum_{i,\,\cH^1(\p C_i^k\cap\Sigma)>0} a_k^{j_i}\chi_{C_i^k}
$$
and 
$$
u:= \sum\limits_{n} w_n \sum\limits_{i\in I_n,j\in Y_n} \chi_{E_i\cup S_j}^{} + \sum\limits_{i\in I\setminus \cup_n I_n} w^i\chi_{E_i}^{}.
$$
By construction, $a_k$ is a piecewise rigid displacement associated to $B_k,$ $u\in H_\loc^1(\Int{A}\cup\substrate;\R^2)$ and $u_k+ a_k\to u$ a.e.\ in $\Int{A}\cup\substrate.$ 
Note that $\str{u_k+a_k} = \str{u_k}.$ Hence, by convergence $A_k\overset{\tau_\fA}\to A$ and by \eqref{dafefafa},  for any Lipschitz open set $D\strictlyincluded \Int{A}\cup \substrate,$
$$
\int_D |\str{u_k+a_k}|^2< \frac{R+\int_\Sigma|\beta|d\cH^1}{2c_3} 
$$
for all large $k$ (depending only $D$). Since $u_k+a_k\to u$ a.e.\ in $D,$ by the Poincar\'e-Korn inequality, 
$\str{u_k+a_k} \wk \str{u}$ weakly in $L^2(D;\mtwo).$ Then by the convexity of $v\mapsto \int_D |\str{v}|^2dx,$ we get 
$$
\int_D|\str{u}|^2dx \le\liminf\limits_{k\to\infty} \int_D |\str{u_k+a_k}|^2 \le \frac{R+\int_\Sigma|\beta|d\cH^1}{2c_3}. 
$$
Hence, letting $D\nearrow \Int{A}\cup \substrate$ we get $u\in GSBD^2(\Ins{A};\R^2).$ 
Consequently, $(A,u)\in \admissible_m$ and $(B_k,u_k+a_k) \overset{\tau_\admissible}\to (A,u)$ as $k\to\infty.$ 

We observe that if $\substrate=\emptyset,$ the terms of the surface energy $\cS(A_k,u_k)$ related to $\Sigma$ disappears, and hence, using $\str{u_k +a_k} = \str{u_k}$ and property (a1),
$$
\cF(A_k,u_k) =\cF(B_k,u_k+a_k) + o(1), 
$$
where $o(1)\to0$ as $n\to\infty,$ and so we can define $D_n=B_n.$

{\it Step 2: Further modification of $\{B_k\}$.} Without loss of generality we assume $\fm<|\Omega|.$ It remains to control $J_{u_k+a_k}$ at $\Sigma$ since, as mentioned above, adding different rigid displacements to $u_k$ in  connected components of the substrate and  the free crystal whose closures intersect can result in a larger jump set $J_{u_k+a_k}$ than $J_{u_k}.$ Recall that by condition (a4) and (a6), 
\begin{equation}\label{outside_jump}
\lim\limits_{k\to\infty} \cH^1(J_{u_k+a_k}\setminus \p^*A) =0. 
\end{equation}
Hence, we need only to control $J_{u_k}\cap \p^*A.$ The idea here is to remove a ``small'' subset $R_k$ of $B_k$ containing almost all points $x\in \Sigma\cap \p^*A\cap (\p^*A_k\setminus J_{u_k})$ which in the limit becomes jump for $u$. In order to keep volume constraint, 
we will insert a square $U_k$ of volume $|R_k|$ in $\Omega\setminus A_k.$ This is possible since $|R_k|\to0$ and $|A_k|\le \fm$ for all $k.$ 

More precisely, we prove that for any $\delta\in(0,1/16),$ there exist $k_\delta>0$ and $(B_k^\delta,u_k^\delta)\in \admissible_m$ with $|B_k|=|B_k^\delta|$ such that
$(B_k^\delta,u_k^\delta)$ and 
\begin{equation}\label{energy_estimare}
\cF(B_k,u_k) +  4\delta(1+c_2)\cH^1(\Sigma) + 4c_2 \delta>  \cF(B_k^\delta,u_k+a_k) 
\end{equation}
for any $k>k_\delta.$ 

We divide the proof into four steps.

{\it Substep 2.1.}
By assumptions of the theorem, conditions (a2) and (a5) and Lemma \ref{lem:f_convergence} (b), 
$|A|\le \fm.$ Hence, we can choose a square $U\strictlyincluded \Omega\setminus\cl{A}.$ 
By (a2) and the definition of $\tau_\fA$-convergence, there is no loss of generality in assuming
$U\strictlyincluded \Omega\setminus {\cl{B_k}}$ for any $k.$ 
Let 
\begin{equation}\label{volume_sq}
\epsilon_0:=\sqrt{|U|}. 
\end{equation}
Without loss of generality, we assume $\epsilon_0\in(0,\frac12).$

First we observe that for any $\delta\in(0,1):$  
\begin{itemize}
\item[(b1)] since $\Sigma$ is Lipschitz, for $\cH^1$-a.e.\ $x\in\Sigma,$ there exist a unit normal $\nu_\Sigma(x)$ to $\Sigma$ and $\overline{r_x}>0$ such that for any $r\in(0,\overline{r_x}),$ $U_{r,\nu_\Sigma(x)}(x)\cap \Sigma$ can be represented as a graph of a Lipschitz function over tangent line $U_{r,\nu_\Sigma(x)}(x)\cap T_x$ at $x$ in the direction $\nu_\Sigma(x);$

\item[(b2)] since $\varphi$ is uniformly continuous, for any $x\in \cl{\Omega}$ there exists $\overline{r_x^\delta}>0$ such that for any $y\in U_{r_x^\delta,\nu_\Sigma(x)}(x)$ and $\xi\in \S^1,$
$$
|\varphi(y,\xi) - \varphi(x,\xi)|<\delta;
$$ 

\item[(b3)] since $\cH^1$-a.e.\ $x\in\Sigma$ is the Lebesgue point of $\beta,$ there exists $\overline{r_x^\delta}>0$ such that for any $r\in(0,\overline{r_x^\delta}),$
$$
\int_{U_{r,\nu_\Sigma(x)}(x)\cap\Sigma} |\beta(y) - \beta(x)|d\cH^1(y) <\delta\cH^1(U_{r,\nu_\Sigma(x)}(x)\cap\Sigma);
$$

\item[(b4)] for $\cH^1$-a.e.\ $x\in \Sigma\cap\p^*A$ one has 
$$
\theta^*(\Sigma,x) = \theta_*(\Sigma,x)=\theta^*(\Sigma\cap \p^*A,x) = \theta_*(\Sigma\cap \p^*A,x)=1,
$$ 
thus, there exists $r_x>0$ such that for any $r\in (0,r_x),$ 
$$
\cH^1(U_{r,\nu_\Sigma(x)}(x) \cap (\Sigma\setminus \p^*A )) <2\delta r;
$$

\item[(b5)] by Proposition \ref{prop:tangent_line_conv} applied with a connected component of $\p A,$ for a.e.\ $x\in \Sigma\cap\p^*A,$ $
\cl{U_{1,\nu_\Sigma(x)}(x)}\cap\sigma_{\rho,x}(\p A)  \overset {K}{\longrightarrow} \cl{U_{1,\nu_\Sigma(x)}(0)}\cap (T_x-x)$  as $\rho\to0,$ where $\sigma_{\rho,x}(y):=\frac{y-x}{\rho}$ is the blow up map and $T_x-x$ is the straight line passing through the origin and parallel to the tangent line $T_x$ of $\p^*A$ (and of $\Sigma$) at $x$.   Thus, there exists $\overline{r_x^\delta}>0$ such that for any $r\in(0,\overline{r_x^\delta}),$ $U_{r,\nu_\Sigma(x)}(x)\cap \{\dist(\cdot, T_x)>\delta r\}\subset \Int{A}\cup\substrate.$
\end{itemize}
By  \eqref{outside_jump}, there exists  $\bar k_\delta>0$ such that 
\begin{equation}\label{tiklanish}
\cH^1(J_{u_k+a_k}\setminus \p^*A)<\delta 
\end{equation}
for any $k>\bar k_\delta.$ 

Fix $\delta\in(0,\frac{1}{16})$ and let $t_\delta>0$ be such that 
\begin{equation}\label{epsilon00}
|\{x\in\Omega:\,\,\dist(x,\Sigma)<t_\delta\}|< \delta^2\epsilon_0^2, 
\end{equation}
where $\epsilon_0$ is given in \eqref{volume_sq}.
 
We now consider connected components $E_i$ and $S_j$ such that the associated rigid displacements are different. 
Let $I'$ be the set of all $i\in I$ such that $\cH^1(\p^*E_i\cap\p \substrate)>0$ and $\liminf\limits_{k\to\infty} \cH^1(\p E_i^k\cap\p \substrate)>0,$ and there exists a connected component $S_j$ of $\substrate$ such that $u_k+b_k^i\to u$ a.e.\ in $E_i$ and $|u_k+b_k^i|\to+\infty$ a.e.\ in $S_j$ for the associated sequence $\{b_k^i\}$ of rigid displacements in $E_i.$ 

Set 
$$
L:=\Sigma\cap \bigcup\limits_{i\in I'} \p^*E_i.
$$
Note that $\cl{L}\subset \Sigma$ and by (b1)-(b5),
for a.e.\ $x\in \cl{L}$ for which $\nu_\Sigma(x)$ and $\nu_A(x)$ exist and $\nu_\Sigma(x)=\nu_A(x)$ there is 
\begin{equation}\label{def_ar_x}
r_x:=r_x^\delta\in(0,\frac{t_\delta}{8}) 
\end{equation}
such that  properties (b1)-(b5) holds with $x$ and $r=r_x.$ Note that for any such $x:$
\begin{itemize}
\item[(c1)] since $B_k\overset{\tau_\fA}{\to} A,$ by property (b5), there  exists $\overline{k_x^\delta}>\bar k_\delta$ such that 
$U_{r,\nu_\Sigma(x)}(x)\cap \{\dist(\cdot, T_x)>\delta r\}\subset \Int{B_k}\cup\substrate$ for any $k>\overline{k_x^\delta};$
 
\item[(c2)] by Proposition \ref{prop:jump_estimate}  applied with $u_k+a_k^i,$ there exists $\overline{k_x^\delta}>\bar k_\delta$ such that
\begin{align} \label{before_blow_up12}
& \int_{U_{r,\nu_0}(x)\cap \Omega \cap\p^*A_k} \varphi(y,\nu_{A_k})d\cH^1  
+2\int_{U_{r,\nu_0}(x) \cap \Omega  \cap (A_k^{(0)} \cup A_k^{(1)}) \cap\p A_k}
\varphi(y,\nu_{A_k})d\cH^1\nonumber \\  
&+ \int_{U_{r,\nu_0}(x) \cap \Sigma \cap A_k^{(0)}\cap \p A_k} \big(\varphi(y,\nu_\Sigma)
+ \beta \big)d\cH^1\nonumber \\
& +  \int_{U_{r,\nu_0}(x)\cap \Sigma \cap \p^*A_k\setminus J_{u_k}} \beta d\cH^1  +
\int_{U_{r,\nu_0}(x) \cap J_{u_k}} \varphi(y,\nu_\Sigma)\,d\cH^1 \no\\ 
&\ge \frac{1}{1+\frac{\delta}{c_2}}
\int_{U_{r,\nu_0}(x) \cap \Sigma\cap \p^*A} \varphi(y,\nu_\Sigma) \,d\cH^1 - \delta\,\int_{U_{r,\nu_0}(x) \cap \Sigma}\varphi(y,\nu_\Sigma)d\cH^1. 
\end{align} 
for any $k>\overline{k_x^\delta},$ where $r:=r_x$ and $\nu_0:=\nu_\Sigma(x).$
\end{itemize}

{\it Substep 2.2.}
Let $x\in L$ be with properties (c1)-(c2) and let $U_r:=U_{r,\nu_\Sigma(x)}(x)$ and $Q_x\subset \Omega\cap $ be the open set whose boundary consists of  $\Gamma_1:=U_r\cap \Sigma,$ two segments  $\Gamma_2,\Gamma_3\subset \p U_r$ of length at most $2\delta r,$ parallel to $\nu_\Sigma(x),$ and the segment $\Gamma_4:= \Omega\cap U_r \cap \{\dist(\cdot,T_x)=\delta\}$ of length $2r.$ Given $k>k_x^\delta$ let 
$$
\hat B_k^\delta:=B_k\setminus Q_x)\cup (\cl{\Omega}\cap \p U_r).
$$ 
Clearly, $(\hat B_k^\delta,u_k+a_k)\in\fA_m.$ We claim that 
\begin{equation}\label{shakar_bollarda_bu}
\cS(B_k,u_k) \ge \cS(\hat B_k^\delta,u_k) -4\delta(1+c_2)\cH^1(U_r\cap \Sigma). 
\end{equation}
Indeed, without loss of generality, we assume that $x=0$ and $\nu_\Sigma(x)={\bf e_2}.$  By the anisotropic minimality of segments, 
\begin{equation}\label{asedadds}
\int_{\Gamma_1} \varphi(0,\nu_\Sigma) d\cH^1 + \int_{\Gamma_2\cup \Gamma_3} \varphi(0,{\bf e_2}) d\cH^1 \ge \int_{\Gamma_4} \varphi(0,{\bf e_2})d\cH^1. 
\end{equation}
Since $\cH^1(\Gamma_1)\ge 2r=\cH^1(\Gamma_4)$ and $\cH^1(\Gamma_2),\cH^1(\Gamma_3)\le 2\delta r$ and, also by property (b2),  we have 
\begin{equation}\label{adafaev0}
\int_{\Gamma_1} \varphi(0,\nu_\Sigma) d\cH^1 \le \int_{\Gamma_1} \varphi(y,\nu_\Sigma) d\cH^1 + \delta\cH^1(\Gamma_1), 
\end{equation}
and 
\begin{equation}\label{adafaev}
\int_{\Gamma_2\cup \Gamma_3} \varphi(0,{\bf e_2}) d\cH^1 \le 2\delta \varphi(0,{\bf e_2})  \cH^1(\Gamma_1)  
\end{equation}
and 
\begin{equation}\label{adafaev2}
\int_{\Gamma_4} \varphi(0,{\bf e_2})d\cH^1 \ge \int_{\Gamma_4} \varphi(y,{\bf e_2})d\cH^1 - \delta  \cH^1(\Gamma_1), 
\end{equation}
hence, using \eqref{finsler_norm} in \eqref{adafaev}, from \eqref{asedadds} and \eqref{adafaev0}-\eqref{adafaev2} we obtain
\begin{equation}\label{sigma_vs_bound}
 \int_{\Gamma_4} \varphi(y,{\bf e_2})d\cH^1 \le  \int_{\Gamma_1} \varphi(y,\nu_\Sigma) d\cH^1 + 2\delta (1+c_2) \cH^1(\Gamma_1). 
\end{equation}
Denoting by $\alpha_k$ the left-hand side of \eqref{before_blow_up12}, by condition (a1) and the definition of $\hat B_k^\delta,$ we have 
\begin{align*}
\cS(B_k,u_k) - \cS(\hat B_k^\delta,u_k) \ge& \alpha_k -2\int_{\p B_k\setminus\p A_k} \varphi(y,\nu_{B_k})d\cH^1  \\
- & \int_{\Gamma_2\cup\Gamma_3\cup\Gamma_4} \phi(y,\nu_{B_k^\delta})d\cH^1 - 
\int_{(B_k^\delta)^{(0)}\cap (\Gamma_2\cup\Gamma_3)} \phi(y,\nu_{B_k^\delta})d\cH^1, 
\end{align*}
thus, using $\Gamma_1=U_r\cap\Sigma,$ from \eqref{before_blow_up12}, \eqref{sigma_vs_bound}, \eqref{adafaev}, \eqref{tiklanish} and \eqref{finsler_norm} 
we obtain
\begin{align*}
\cS(B_k,u_k) - \cS(\hat B_k^\delta,u_k) \ge  \frac{1}{1+\frac{\delta}{c_2}} 
\int_{U_r\cap\Sigma\cap\p^*A} \varphi(y,\nu_\Sigma)d\cH^1
- \int_{U_r\cap\Sigma} \varphi(y,\nu_\Sigma) d\cH^1\\
- 3\delta (1+c_2) \cH^1(U_r\cap\Sigma).
\end{align*}
In the last inequality using $\frac{1}{1+\frac{\delta}{c_2}}\ge 1 - \frac{\delta}{c_2}$ and  inequality \eqref{finsler_norm} once more we deduce 
\begin{equation}\label{agshetr}
\cS(B_k,u_k) \ge\cS(\hat B_k^\delta,u_k) 
- \int_{U_r\cap(\Sigma\setminus \p^*A)} \varphi(y,\nu_\Sigma)d\cH^1
- \delta (4+3c_2) \cH^1(U_r\cap\Sigma). 
\end{equation}
Now condition (b4) and \eqref{finsler_norm} and  the inequality $ \cH^1(\Gamma_1)\ge 2r $ imply 
$$
\int_{U_r\cap(\Sigma\setminus \p^*A)} \varphi(y,\nu_\Sigma)d\cH^1\le c_2\cH^1(U_r\cap(\Sigma\setminus \p^*A)) \le 2\delta c_2 r \le \delta c_2 \cH^1(U_r\cap \Sigma).
$$
Inserting this in \eqref{agshetr} we get  \eqref{shakar_bollarda_bu}. 

{\it Substep 2.3.} Now we choose finitely many points $x_1\ldots,x_N\in \cl{L}$ with corresponding $r_1,\ldots,r_N$ satisfying (b1)-(b5) and (c1)-(c2) such that the squares $\{U_{r_j,\nu_\Sigma(x_j)}(x_j)\}_{j=1}^N$ are pairwise disjoint and 
\begin{equation}\label{asgard}
\cH^1\Big(L \setminus \bigcup\limits_{j=1}^N U_{r_j,\nu_\Sigma(x_j)}(x_j)\Big) <\delta.
\end{equation}
Recalling the definition of $k_\delta^x$ in condition (c2) and the definition $\bar k_\delta$ in \eqref{tiklanish}, let $k_\delta:=\max\{\bar k_\delta,\overline{k_\delta^{x_1}},\ldots,\overline{k_\delta^{x_N}}\}$ and let $Q_{x_j}\subset \Omega\cap U_{r_j,\nu_\Sigma(x_j)}(x_j)$ be as in Substep 2.2. Set 
$$
\tilde B_k^\delta: = \Big(B_k\setminus \bigcup\limits_{j=1}^N Q_{x_j}\Big) \cup \bigcup\limits_{j=1}^N(\cl{\Omega}\cap \p U_{r_j,\nu_\Sigma(x_j)}(x_j)). 
$$
Then, as in the proof of \eqref{shakar_bollarda_bu},
$$
\cS(B_k,u_k) - \cS(\tilde B_k^\delta,u_k) \ge -4\delta(1+c_2)\sum\limits_{j=1}^N \cH^1(U_{r_j,\nu_\Sigma(x_j)}(x_j)\cap \Sigma) 
$$
so that by the pairwise disjointness of $\{ U_{r_j,\nu_\Sigma(x_j)}(x_j)\},$ 
\begin{equation}\label{surface_estimare}
\cS(B_k,u_k) - \cS(\tilde B_k^\delta,u_k) \ge -4\delta(1+c_2)\cH^1(\Sigma).  
\end{equation}
Recalling \eqref{asgard} and \eqref{tiklanish} and using \eqref{finsler_norm} we estimate
\begin{align*}
\cS(\tilde B_k^\delta,u_k+a_k) &- \cS(\tilde B_k^\delta,u_k)\\ 
\le &\int_{L\setminus \bigcup_j U_{r_j,\nu_\Sigma(x_j)}(x_j)} \varphi(y,\nu_\Sigma)d\cH^1 + \int_{J_{u_k+a_k}\setminus \p^*A} \varphi(y,\nu_\Sigma)d\cH^1<2c_2\delta,
\end{align*}
hence, from \eqref{surface_estimare} we get 
\begin{equation}\label{surface_estimare1}
\cS(B_k,u_k) - \cS(\tilde B_k^\delta,u_k+a_k) \ge -4\delta(1+c_2)\cH^1(\Sigma) - 2c_2\delta.  
\end{equation}
On the other hand, since $\Int{\tilde B_k^\delta} \subset \Int{B_k}$ and $\str{u_k} = \str{u_k+a_k},$
\begin{equation}\label{elastic_estimare1}
\cW(B_k,u_k) \ge  \cW(\tilde B_k^\delta,u_k+a_k).  
\end{equation}
From \eqref{surface_estimare1} and \eqref{elastic_estimare1} we deduce 
$$
\cF(B_k,u_k) \ge \cF(\tilde B_k^\delta,u_k+a_k) -4\delta(1+c_2)\cH^1(\Sigma) - 2c_2\delta.  
$$
However, by construction, $|B_k|\ge |\tilde B_k^\delta|$ since $\tilde B_k^\delta\subset B_k\cup \bigcup\limits_{j=1}^N(\cl{\Omega}\cap \p U_{r_j,\nu_\Sigma(x_j)}(x_j)).$ Thus, $R_k:=B_k\setminus \tilde B_k^\delta$ satisfies $|R_k|=|B_k\Delta \tilde B_k^\delta|.$
Since $\bigcup_jQ_{x_j} \subset \Omega\cap \{\dist(\cdot,\Sigma)<\frac{t_\delta}{8}\},$ thus, by \eqref{epsilon00},
$|R_k|<\delta^2\epsilon_0.$ Hence, we choose a square $U_k\subset U$ (see \eqref{volume_sq}) such that $|U_k|=|R_k|.$  For $k>k_\delta$ set 
$$
B_k^\delta:= \tilde B_k^\delta \cup U_k.
$$
In order not to increase the number of connected components of $B_k^\delta$, we translate $U_k$ in $\Omega\setminus{\cl{B_k}}$ until it touches to $\p \tilde B_k^\delta.$ 
Define 
$$
u_k^\delta:=u_k\chi_{\tilde B_k^\delta} + u_0\chi_{U_k}. 
$$
Then $\{(B_k^\delta,u_k^\delta)\}\subset\admissible_m$ and for any $k>k_\delta$ by \eqref{surface_estimare1} and \eqref{elastic_estimare1} 
$$
\cF(B_k,u_k)\ge \cF(B_k^\delta,u_k^\delta) - \cS(U_k,u_0)  -4\delta(1+c_2)\cH^1(\Sigma) - 2c_2\delta.   
$$
By the choice of $U_k,$ its sidelength is less that $\delta\epsilon_0,$ hence, using $\epsilon_0<\frac12$ and \eqref{finsler_norm}, 
$
\cS(U_k,u_0) \le 2c_2\delta 
$
so that 
$$
\cF(B_k,u_k)\ge \cF(B_k^\delta,u_k^\delta)  -4\delta(1+c_2)\cH^1(\Sigma) - 4c_2 \delta.   
$$

{\it Step 3: Construction of $(D_n,v_n).$} Notice that the sequence $\{(B_k^\delta,u_k^\delta)\}$ in general does not need to satisfy $B_k^\delta\overset{\tau_\fA}\to A,$ since we removed ``something'' from $B_k$ and added a square $U_k.$ To overcome this problem, we take $\delta=\delta_n:=\frac{1}{16n}$ and 
$(D_n,v_n):=(B_{k_n}^{\delta_n},u_{k_n}^{\delta_n}),$ where $k_n:=k_{{\delta_n}}+1,$ and there is no loss of generality in assuming $n\mapsto k_n$ is increasing. Denote $r_j^n:=r_{x_j}^{\delta_n},$ where the latter is defined in Substep 2.3 and notice that by \eqref{epsilon00} and \eqref{def_ar_x}  $r_j^n\to 0$ as $n\to\infty$ In particular, $\p D_n\overset{K}\to \p A$ as $n\to\infty.$ Thus, $D_n\overset{\tau_\fA}{\to}A.$ Since $|D_n\Delta A|\to0,$ $v_n\to u$ a.e.\ in $A\cup\substrate.$
By \eqref{energy_estimare}
\begin{equation}\label{hsgtea}
\cF(B_{k_n},u_{k_n})  + \frac{(1+c_2)\cH^1(\Sigma) + c_2}{4n} \ge \cF(D_n,u_n), 
\end{equation}
thus \eqref{liminfga_ut_eshmat} follows from \eqref{hsgtea} and \eqref{first_attempt_b}.
\end{proof}

\black

\section{Lower semicontinuity}\label{sec:lsc_results}

In this section we consider more general surface energies.
For every $A\in \fA$ and  
$J_A\in \cJ_A,$ where
$$
\cJ_A:= \big\{J\subseteq\Sigma\cap \overline{\p^* A}:\,\,\text{$J$ 
is $\cH^1$-measurable} \big\}
$$ 
is the collection of all possible delaminations on $\Sigma,$
we define 
\begin{align*}
\cS(A,J_A; \varphi, g):= \int_{\Omega \cap\p^*A} \varphi(x,\nu_A)d\cH^1 
+2\int_{\Omega \cap (A^{(1)}\cup A^{(0)})\cap\p A} \varphi(x,\nu_A)d\cH^1
\nonumber\\  
 +  \int_{\Sigma\setminus \p A} g(x,0)d\cH^1 
+ \int_{\Sigma\cap A^{(0)}\cap \p A} \big(\varphi(x,\nu_\Sigma)
+ g(x,1)\big)d\cH^1\nonumber\\
+ \int_{\Sigma\cap \p^*A\setminus J_A} g(x,1)d\cH^1 +
\int_{J_A} \big(\varphi(x,\nu_\Sigma) + g(x,0)  \big)\,d\cH^1,
\end{align*}
where $g:\Sigma\times \{0,1\}\to\R$ is a Borel function. 
We remark  that
$\cS(A,u) = \cS(A,J_u;\varphi,g)$ with 
$g(x,s)=\beta(x)s$ and $J_A = J_u.$

The main result of this section is the following.

\begin{proposition}[\textbf{Lower-semicontinuity of $\cS$}]
\label{prop:lsc_surface_energy}
Suppose that 
$g:\Sigma\times \{0,1\}\to\R$ is a Borel function 
such that $g(\cdot,s)\in L^1(\Sigma)$ for $s=0,1$ and
\begin{equation}\label{beta_condition}
|g(x,1) - g(x,0)| \le \varphi(x,\nu_\Omega (x))  
\end{equation}
for $\cH^1$-a.e.\ $x\in\Sigma.$
Let $A_k\in \fA_m,$ $J_{A_k}\in\cJ_{A_k},$
$A\in\fA_m$ and $J_A\in \cJ_A$ be such that 
\begin{itemize}
\item[\rm (a)] $A_k\overset{\tau_{\fA}}{\to} A$  as $k\to\infty;$

\item[\rm (b)] $\cH^1$-a.e.\ $x\in J_A$ there exist  
$r=r_x>0,$ $w,w_k\in GSBD^2(B_r(x);\R^2)$  
and relatively open subset $L_k$ of $\Sigma$  with $\cH^1(L_k)<1/k$
for which
\begin{equation}\label{condition_J_A}
\begin{cases}
\text{$J_{w_k}\subset B_r(x)\cap (J_{A_k}\cup (\Omega\cap \p A_k)\cup L_k)$
and $J_A\subset J_w;$}\\[1mm]
\text{$w_k\to w$ a.e.\ in $B_r(x)$ 
as $k\to\infty;$}\\[1mm]
\sup\limits_{k\ge1} \int_{B_r(x)}|\str{w_k}|^2dx<\infty.
\end{cases} 
\end{equation}
\end{itemize}
Then 
\begin{equation}\label{final_call_lsc}
\liminf\limits_{k\to\infty} \cS(A_k,J_{A_k}; \varphi, g) 
\ge \cS(A,J_A; \varphi, g). 
\end{equation}
\end{proposition}

We prove Proposition \ref{prop:lsc_surface_energy} 
using a blow-up around the points of the boundary of 
$A$.  Given $y_o\in\R^2$ and $\rho>0,$ the blow-up map 
$\sigma_{\rho,y_o}:\R^2\to\R^2$ is defined as
\begin{equation}\label{blow_up_map}
\sigma_{\rho,y_o}(y):=\frac{y - y_o}{\rho}. 
\end{equation}
When $y_o=0$ we write $\sigma_\rho$ instead of $\sigma_{\rho,0}.$
Given $\nu\in \S^1,$ 
$
U_{\rho,\nu}(x) 
$
is an open square of sidelength $2\rho>0$ centered at $x$
whose sides are either  perpendicular or parallel to $\nu;$ 
if $\nu = {\bf e_2}$ and $x=0,$
we write $U_{\rho,\nu}(0):=U_\rho=(-\rho,\rho)^2,$
$U_\rho^+=(-\rho,\rho)\times(0,\rho),$ and 
$I_\rho:=[-\rho,\rho]\times\{0\}.$ %
Observe that $\sigma_{\rho,x}(U_{\rho,\nu}(x))=U_{1,\nu}(0)$ and 
$\sigma_{\rho,x}(\cl{U_{\rho,\nu}(x)})=
\cl{U_{1,\nu}(0)}.$
We denote by $\pi$ the projection 
onto $x_1$-axis i.e., 
\begin{equation}\label{def_of_projection}
\pi(x)=(x_1,0). 
\end{equation}
The following auxiliary results will be used in the proof 
of Proposition \ref{prop:lsc_surface_energy}.
 
\begin{lemma}\label{lem:set_shrinks_to_line}
Let $U$ be any open square, $K\subset \cl{U}$ be a nonempty closed 
set and $E_k\subset U$ be  such that 
$\sdist(\cdot,\p E_k)\overset{k\to\infty} {\longrightarrow} \dist(\cdot,K)$ 
uniformly in $\cl{U}.$ 
 Then $E_k\overset{\cK}{\to} K$ 
as $k\to\infty.$
Analogously if $\sdist(\cdot,\p E_k) \overset{k\to\infty} {\longrightarrow}
-\dist(\cdot,K)$ 
uniformly in $\cl{U},$ then $\cl{U}\setminus E_k\overset{\cK}{\to} K$ 
as $k\to\infty.$ 
\end{lemma}

\begin{proof}
We prove only the first assertion, the second being the same.
If  $x_k\in E_k$ is such that $x_k\to x,$ then by assumption, 
$$
\dist(x,K) = \lim\limits_{k\to\infty} \sdist(x_k,\p E_k) \le0
$$
so that $x\in K.$  On the other hand, given $x\in K$ suppose 
that there exists $r>0$ such that $B_r(x)\cap E_k=\emptyset$
for infinitely many $k.$
Then for such $k,$ $\sdist(x,\p E_k) = \dist(x,E_k)\ge r>0,$  
which contradicts to the assumption.
\end{proof}

In the next lemma we observe that the endpoints of every curve $\Gamma$ contained in the boundary of any bounded set $A$ with connected boundary  are still arcwise connected if we remove the boundary of $\Int{\cl{A}}$ belonging to $\Gamma.$ 

\begin{lemma}\label{lem:two_fold_connection}
Let $A\subset\R^2$ be a bounded set 
such that $\p A$ is connected and has finite $\cH^1$ measure.
Suppose that $x,y\in \p A$ are such that $x\ne y$ and  
$\Gamma\subset\p A$ is a curve connecting $x$
to $y.$ Then there exists a curve   
$\Gamma'\subset \overline{\p A\setminus (\Gamma\cap \p \Int{\overline{A}})}$
connecting $x$ to $y$. 
\end{lemma}

\begin{proof} 
Without loss of generality we assume  $G:=\Int{\overline{A}}\ne \emptyset,$ otherwise we simply take $\Gamma'=\Gamma.$   Note that 
\begin{equation}\label{ftarets}
\p G = \big\{x\in \p A:\,\, B_r(x)\cap \Int{A},B_r(x)\setminus 
\overline{A} \ne \emptyset\,\,\text{for every $r>0$}\big\}. 
\end{equation}
Since connected compact sets of finite length are arcwise connected (see Proposition \ref{prop:rectifiable_sets}),
it suffices to show that $x$ and $y$ belong to the same connected 
component of $\overline{\p A\setminus (\Gamma\cap \p G)}.$  \black
Suppose that there exist  two open sets $P,Q\subset\R^2$
with disjoint closures such that 
\begin{equation}\label{eq:bhsdf}
\overline{\p A\setminus (\Gamma\cap \p G)} = 
(P\cap \overline{\p A\setminus (\Gamma\cap \p G)}) \cup 
(Q\cap \overline{\p A\setminus (\Gamma\cap \p G)}), 
\end{equation}
where $x\in P\cap \overline{\p A\setminus (\Gamma\cap \p G)}$
and $y\in Q\cap \overline{\p A\setminus (\Gamma\cap \p G)}.$
Then $\Gamma\setminus \overline{P\cup Q}\ne\emptyset$ and
\begin{equation}\label{sgerta}
\Gamma\setminus \overline{P\cup Q} = \p G \setminus \overline{P\cup Q}.
\end{equation}
Since $\overline{P}\cap \overline{Q}=\emptyset$ and $\cH^1(\Gamma)<\infty,$ the number of connected components $\{L_i\}_{i=1}^n$ of $\Gamma\setminus\cl{P\cup Q}$ connecting both $P$ and $Q$ is at most finite. Moreover, since $\Gamma$ has no self-intersections (see Subsection \ref{subsec:rectifiable} for the definition of the curve in our setting) and the endpoints of $\Gamma$ belong to $P$ and $Q,$ respectively, $n$ must be odd. However, by \eqref{sgerta}  $L_i\subset \p G,$ and hence, by \eqref{ftarets},
every neighborhood of $L_i$ contains points belonging to both $\Int{A}$ and $\R^2\setminus \cl{A}.$ We reached a contradiction since in this case $\Int{A}$ would be unbounded. 
\end{proof}

Notice that if $A\in \fA_m,$ then $\overline{\p^* A} = \p A^{(1)} =\p \Int{\overline{A}}.$

\begin{lemma}[\textbf{Creation of external filament energy}]
\label{lem:creation_filament}
Let $\phi$ be a norm in $\R^2$ satisfying
\begin{equation}\label{phi_cond1}
c_1\le \phi(\nu) \le c_2,\qquad \nu\in\S^1, 
\end{equation}
for some $c_2\ge c_1>0,$
and
let $\{E_k\}$ be a sequence of subsets of $U_1$ 
such that 
\begin{itemize}
\item[\rm (a)] $E_k \overset{\cK}\to I_1$
as $k\to\infty;$

\item[\rm (b)] there exists $m_o\in\N_0$ such that the number 
of connected components
of $\p E_k$ lying strictly inside $U_1$ does not exceed $m_o.$
\end{itemize}
Then for every $\delta\in (0,1)$ 
there exists $k_\delta>1$
such that for any $k>k_\delta,$
\begin{align}
 \label{step1ning_claimi}
\int_{U_1\cap \p^* E_k} \phi(\nu_{E_k})\,d\cH^1 
+2 \int_{U_1\cap (E_k^{(1)}\cup E_k^{(0)})\cap \p E_k}   \phi(\nu_{E_k})\,d\cH^1 
\ge 
2\int_{I_1}\phi({\bf e_2})\,d\cH^1- \delta. 
\end{align}

\end{lemma}

\begin{proof}
Let us denote the left hand side of \eqref{step1ning_claimi}
by $\alpha_k.$ We may suppose $\sup_k\alpha_k<\infty$.
By assumption (a), for every $\delta\in (0,1)$
there exists $k_{1,\delta}>1$ such that
$$
E_k\subset [-1,1]\times \Big(-\frac{\delta}{16c_2m_o},
\frac{\delta}{16c_2m_o}\Big)
$$ 
for all  $k>k_{1,\delta}.$

{\it Step 1.} Assume that for some  $k > k_{1,\delta},$  $\p E_k$ has a connected component $K^1$ intersecting both $\{x_1=1\}$ and $\{x_1=-1\}.$  In this case 
by Lemma \ref{lem:two_fold_connection},
$\overline{\p E_k\setminus (K^1\cap \p\Int{\overline{A}})}$ is also connected 
and contains a path $K^2$ connecting  $\{x_1=1\}$ to $\{x_1=-1\}.$ Note that $K^1$ and $K^2$ may coincide on $(E_k^{(1)}\cup E_k^{(0)})\cap\p E_k.$
Let $R_i^1$ and $R_i^2,$ $i=1,2,$ be the segments along 
the vertical lines $\{x_1=\pm1\}$
connecting the endpoints
of $K^1$ and $K^2$ to $(\pm 1,0),$ respectively. 
Since $K^1\cap \p^*E_k$ and $K^2\cap \p^*E_k$ are disjoint up to a $\cH^1$-negligible set 
\begin{align}\label{sFFFFFff}
\alpha_k\ge  \sum\limits_{j=1}^2\Big(\int_{K^j\cap \p^*E_k}  
\phi(\nu_{E_k})\,d\cH^1 +  \int_{{K^j\setminus \p^*E_k}}
\phi(\nu_{E_k})\,d\cH^1\Big)\nonumber \\
 =   \sum\limits_{j=1}^2 \int_{\gamma_j}  
\phi(\nu_{\gamma_j})\,d\cH^1 
- \sum\limits_{i,j=1}^2 \int_{R_i^j}  
\phi({\bf e_1})\,d\cH^1,
\end{align}
where $\gamma_j:=R_1^j\cup K^j\cup  R_2^j$   is the curve connecting 
$(-1,0)$ to $(1,0).$ By the (anisotropic) minimality of   segments \cite[Lemma 6.2]{FFLM:2011}, 
\begin{equation}\label{dafwgw}
\int_{\gamma_j} \phi(\nu_{\gamma_j})\,d\cH^1\ge 
\int_{I_1} \phi({\bf e_2})\,d\cH^1  
\end{equation}
Moreover, since $\cH^1(R_i^j) \le \frac{\delta}{16c_2m_o}$ for any 
$i,j=1,2,$ by \eqref{sFFFFFff}, \eqref{dafwgw} and
\eqref{finsler_norm} we obtain
\begin{equation}\label{one_two_three}
\alpha_k  \ge 2 \int_{I_1} \phi({\bf e_2})\,d\cH^1 
- \frac{4c_2\delta}{16c_2m_o} = 
2 \int_{I_1} \phi({\bf e_2})\,d\cH^1 
- \frac{\delta}{4 m_o}, 
\end{equation}
which implies \eqref{step1ning_claimi}.
{\it Step 2.} Assume now that every connected component of  
$U_1\cap \p E_k$  intersects at most one of
$\{x_1=1\}$ and $\{x_1=-1\}.$
In this case, let $K^{1},\ldots,K^{m_k}$ stand for the 
connected components of $\p E_k$ lying strictly
inside of $U_1$ (i.e., not intersecting $\{x_1=\pm1\}$); 
by (b), $m_k\le m_o.$ 
Since $\alpha_k<\infty,$  the connected components $\{L^i\}$
of $\overline{U_1}\cap \p E_k$ 
intersecting $\{x_1=\pm1\}$ is at most countable. 
If $\{L^i:\,L^i\cap\{x_1=1\}\ne\emptyset\}=\emptyset$, we set 
$K^{m_k+1}=\emptyset$ otherwise
let $K^{m_k+1}$ 
be such that $\pi(K^{m_k+1})$ contains all 
$\pi(L^i)$ with $L^i\cap\{x_1=1\}\ne\emptyset,$ where $\pi$
is given by \eqref{def_of_projection}.
Analogously, we define $K^{m_k+2}\in 
\{L^i:\,L^i\cap\{x_1=-1\}\ne\emptyset\}\cup\{\emptyset\}.$
By the connectedness of $K^j,$ 
for each $j=1,\ldots,m_k+2,$  
$\pi(K^j)$ (if non-empty)  is a segment
$[a_k^i,b_k^i]\times\{0\}.$ 
Then assumption  (a) and the bound $m_k\le m_o$ imply that
$$
\lim\limits_{k\to\infty} 
\cH^1\Big(I_1\setminus  
\bigcup\limits_{j=1}^{m_k+2} \pi(K^j)\Big)=0.
$$
Hence there exists $k_{2,\delta}>k_{1,\delta}$ such that 
\begin{equation}\label{small_cover}
\cH^1\Big(I_1  \setminus \bigcup\limits_{i=1}^{m_k+2} \pi(K^j)\Big) 
< \frac{\delta}{8c_2m_o} 
\end{equation}
for any $k> k_{2,\delta}.$  Then repeating the proof 
of \eqref{one_two_three} with $K^j$ in
$(a_j,b_j)\times(-1,1)$, for every $j=1,\ldots,m_k+2$   
we find
\begin{align*}
\int_{K^j\cap \p^*E_k} \phi(\nu_{E_k})d\cH^1 
+2 \int_{K^j\cap (E_k^{(1)}\cup E_k^{(0)})\cap\p E_k} \phi(\nu_{E_k})d\cH^1 
\ge 2 \int_{\pi(K^j)} \phi({\bf e_2})\,d\cH^1 
-  \frac{\delta}{4m_o}.
\end{align*}
Therefore, by  \eqref{small_cover} and \eqref{finsler_norm},
\begin{align*}
\alpha_k \ge  & \sum\limits_{j=1}^{m_k+2} 
\int_{K^j\cap \p^*E_k} \phi(\nu_{E_k})d\cH^1 
+2 \int_{K^j\cap (E_k^{(1)} \cup E_k^{(0)})\cap\p E_k} \phi(\nu_{E_k})d\cH^1 \\
\ge & \sum\limits_{j=1}^{m_k+2}  
\Big(2 \int_{\pi(K^j)} \phi({\bf e_2})\,d\cH^1 
- \frac{\delta}{4m_o}\Big)\ge 
2 \int_{\bigcup\pi(K^j)} \phi({\bf e_2})\,d\cH^1 
- \frac{(m_k+2)\delta}{4m_o} \\ 
\ge &  2\int_{I_1} \phi({\bf e_2})\,d\cH^1
- 2c_2\,\frac{\delta}{8c_2m_o}- \frac{(m_o+2)\delta}{4m_o} \\
= & 2\int_{I_1} \phi({\bf e_2})\,d\cH^1  - 
\frac{(m_o+3)\delta}{4m_o}.
\end{align*}
Since $m_o\ge1,$ this implies \eqref{step1ning_claimi}.
\end{proof}

\begin{lemma}[\textbf{Creation of internal crack energy}]
\label{lem:creation_crack}
Let $\phi$ be as in Lemma \ref{lem:creation_filament}   and
let $\{E_k\}$ be a sequence of subsets of $U_1$ 
such that 
\begin{itemize}
\item[\rm (a)] $U_1\setminus E_k \overset{\cK}\to I_1=[-1,1]\times\{0\}$
as $k\to\infty;$

\item[\rm (b)] there exists $m_o\in\N_0$ such that the number 
of connected components
of each $\p E_k$ lying strictly inside $U_1$ does not exceed $m_o.$
\end{itemize}
Then for every $\delta\in (0,1)$ 
there exists $k_\delta>1$
such that for any $k>k_\delta,$
$$
\int_{U_1\cap \p^* E_k} \phi(\nu_{E_k})\,d\cH^1 
+2 \int_{U_1\cap (E_k^{(0)}\cup E_k^{(1)})\cap \p E_k}   \phi(\nu_{E_k})\,d\cH^1
\ge 
2\int_{I_1}\phi({\bf e_2})\,d\cH^1- \delta.
$$
\end{lemma}

\begin{proof}
The assertion  follows from applying Lemma 
\ref{lem:creation_crack} to $U_1\setminus E_k.$
\end{proof}

The following result extends  the lower semicontinuity result of 
\cite[Theorem 1.1]{ChC:2018jems} to the anisotropic case.

\begin{proposition}\label{prop:anisotropic_chambolle} 
Let $D\subset\R^d$ be a bounded  open set and
let $\phi\in C(\overline{D}\times\R^d;[0,+\infty)$ be a Finsler norm in $\R^d,$ $d\ge2,$
satisfying
\begin{equation}\label{finsler_norm9}
c_1\le \phi(x,\nu) \le c_2,\qquad (x,\nu)\in\overline{D}\times \S^{d-1}, 
\end{equation}
for some $c_2\ge c_1>0.$
Consider $\{w_h\}\subset GSBD^2(D;\R^d)$ 
such that 
\begin{equation}\label{chambolenergy0000}
\sup\limits_{k\ge1} \int_D|\str{w_h}|^2dx +\cH^1(J_{w_h})\le M  
\end{equation}
for some $M>0$ and 
the set
$$
E:=\{x\in D:\,\,\liminf\limits_{h\to\infty} |w_h(x)| = +\infty\}
$$
has finite perimeter. Suppose that  $w_h\to w$ 
a.e.\ in $D\setminus E$ as  $h\to\infty$ 
(so that by \cite[Theorem 1.1]{ChC:2018jems} $w\in GSBD^2(D\setminus E;\R^d)$).
Then 
\begin{equation}\label{lsc_jumpillo}
\int_{J_w\cup \p^*E} \phi(x,\nu_{J_w\cup\p E}^{})d\cH^{d-1} \le 
\liminf\limits_{h\to\infty} \int_{J_{w_h}} \phi(x,\nu_{J_{w_h}}^{})d\cH^{d-1}.
\end{equation}

\end{proposition}

\begin{proof}
We divide the proof into two steps.

{\it Step 1.} First we prove the \eqref{lsc_jumpillo} assuming that $\phi$ is independent on $x\in D,$ i.e., $\phi(\nu)=\phi(x,\nu)$ for any $x\in \cl{D}\times\R^d.$

Let $W=\{\phi^\circ\le1\}$ be the Wulff shape of $\phi$, i.e., 
the unit ball for the dual norm 
$$
\phi^\circ(\xi) = \max\limits_{\phi(\eta)=1}\,\, |\xi\cdot\eta|.
$$
Note that $\phi^{\circ\circ} = \phi$ and by \eqref{finsler_norm9},
\begin{equation}\label{dual_norm_bound}
\frac{1}{c_2}\,|\xi|\le \phi^\circ(\xi)\le \frac{1}{c_1}\,|\xi|
\end{equation}
for any $\xi\in\R^2.$
Let  
$\{\xi_n\}\subset\p W$ be a countable dense set. 
Then since 
$$
\phi(\nu) = \sup\limits_{n\in\N} |\xi_n\cdot \nu| 
$$
from \cite[Lemma 6]{DBD:1983} it follows that 
for every bounded open set $G$ and $u\in GSBD^2(G;\R^d),$
\begin{equation}\label{de_giorgi_lemma}
\int_{G\cap J_u} \phi(\nu_{J_u}^{})d\cH^{d-1} =  
\sup\,\sum\limits_{n=1}^N \int_{F_n\cap J_u} |\xi_n\cdot\nu_{J_u}^{}|
d\cH^{d-1}, 
\end{equation}
where $\sup$ is taken over finite disjoint open sets $\{F_n\}_{n=1}^N$
whose closures are contained in $G.$

Now we prove \eqref{lsc_jumpillo}.
Under the notation of \cite{BCD:1998,ChC:2018jems},  
for any $\epsilon\in(0,1),$ open set 
$F\subset D$ with $\overline{F}\subset D$
and for $\cH^1$-a.e.\ $\xi\in \p W$  we have 

\begin{align}\label{chambolledan_lavha}
|\xi| \int_{\Pi_\xi}  
\Big[\cH^0(F_y^\xi&\cap  J_{\widehat{w}_y^\xi}\cap  (F\setminus E)_y^\xi)  
+ \cH^0(F_y^\xi\cap \p E_y^\xi)\Big]d\cH^{d-1} \nonumber \\
&\le |\xi| \liminf\limits_{h\to\infty} \int_{\Pi_\xi}
\Big[\cH^0(F_y^\xi\cap J_{(\widehat w_h)_y^\xi}) + 
c_2^{-1} \epsilon f_y^\xi(w_h) \Big]d\cH^{d-1},
\end{align}
where 
$
\Pi_\xi: = \{y\in\R^d:\,\, y\cdot\xi = 0\},
$
is the hyperplane passing through the origin and  orthogonal to $\xi,$
given $y\in\R^d,$
$
F_y^\xi:=\{t\in\R:\,\, y+t\xi\in F\}
$
is the section of the straight line passing 
through $y\in\R^d$ and parallel to $\xi,$
given $u:F\to \R^d$ and $y\in\R^d,$
$\widehat u_y^\xi:F_y^\xi\to\R$ is defined as 
$\widehat u_y^\xi(t): = u(y+t\xi)\cdot\xi,$
and 
\begin{equation}\label{residiea}
f_y^\xi(w_h) = I_y^\xi(w_h) + II_y^\xi(w_h),  
\end{equation}
with
$$
\int_{\Pi_\xi} I_y^\xi(w_h) d\cH^1\le \int_F |\str{w_h(x)}|^2dx,
\qquad h\ge1,
$$
and
$$
\int_{\Pi_\xi} II_y^\xi(w_h) d\cH^1\le 
|D_\xi(\tau(w_h\cdot\xi))|(F), \qquad h\ge1,
$$
for $\tau(t):= \tanh(t)$
(see \cite[Eq.s 3.10 and 3.11]{ChC:2018jems} applied with $F$ 
in place of $\Omega $).

By  \cite[Theorem 4.10]{ACD:1997} and \eqref{dual_norm_bound},
\eqref{chambolledan_lavha} can be rewritten as 
\begin{align}\label{kadikasss}
\int_{F\cap (J_w\cup \p^*E)} |\nu_{J_w\cup\p E}^{}\cdot \xi| d\cH^{d-1}
\le   \liminf\limits_{h\to\infty} \int_{F\cap J_{w_h}}
|\nu_{J_{w_h}}^{}\cdot \xi|d\cH^{d-1}\nonumber \\
  + \epsilon\int_F |\str{w_h(x)}|^2dx + \epsilon|D_\xi(\tau(w_h\cdot\xi))|(F). 
\end{align} 
Fix any finite family $\{F_n\}_{n=1}^N$ of pairwise 
disjoint open sets whose  closures are contained in $D.$ 
Since \eqref{kadikasss}  holds for $\cH^1$-a.e.\ $\xi\in\p W,$
we can extract a countable dense set $\{\xi_n\}\subset \p W$
satisfying \eqref{kadikasss} with $\xi=\xi_i$ and $F=F_j$
for all $i,j.$ Now 
taking $F=F_n$ and $\xi=\xi_n$ in \eqref{kadikasss} 
and summing over $n=1,\ldots,N,$ we get 
\begin{align*}
\sum\limits_{n=1}^N \int_{F_n\cap (J_w\cup \p^*E)} 
|\nu_{J_w\cup\p E}^{}\cdot \xi_n| d\cH^{d-1}
 \le  \liminf\limits_{h\to\infty} \sum\limits_{n=1}^N \int_{F_n\cap J_{w_h}}
|\nu_{J_{w_h}}^{}\cdot \xi_n|d\cH^{d-1}\\
 +\epsilon\int_{\bigcup\limits_{n=1}^NF_n} |\str{w_h(x)}|^2dx + 
\epsilon|D_{\xi_n}(\tau(w_h\cdot\xi_n))|\Big(\bigcup\limits_{n=1}^NF_n\Big).
\end{align*}
Recall that by \eqref{de_giorgi_lemma},
$$
\sum\limits_{n=1}^N \int_{F_n\cap J_{w_h}}
|\nu_{J_{w_h}}^{}\cdot \xi_n|d\cH^1 \le 
\int_{J_{w_h}} \phi(\nu_{J_{w_h}}^{})\cH^{d-1},
$$
and by \eqref{chambolenergy0000},
$$
\int_{\bigcup\limits_{n=1}^NF_n} |\str{w_h(x)}|^2dx \le 
\int_{D} |\str{w_h(x)}|^2dx \le M
$$
and 
$$
|D_{\xi_n}(\tau(w_h\cdot\xi_n))|\Big(\bigcup\limits_{n=1}^NF_n\Big)
\le |D_{\xi_n}(\tau(w_h\cdot\xi_n))|(D) \le M 
$$
for any $h\ge1.$
Therefore, 
\begin{align*}
\sum\limits_{n=1}^N \int_{F_n\cap (J_w\cup \p^*E)} 
|\nu_{J_w\cup\p E}^{}\cdot \xi_n| d\cH^{d-1}
 \le 2M\epsilon + \liminf\limits_{h\to\infty}   
 \int_{J_{w_h}} \phi(\nu_{J_{w_h}}^{})\cH^{d-1}.
\end{align*}
Now taking $\sup$ over $\{F_n\}$ and letting $\epsilon\to0$
we obtain \eqref{lsc_jumpillo}.

{\it Step 2.} Now we prove \eqref{lsc_jumpillo} in general case. 
Without loss of generality we suppose that the $liminf$ in \eqref{lsc_jumpillo} is a finite limit.
Consider the sequence $\{\mu_h\}_{h\ge0}$ of positive Radon measures in $D$ defined at Borel subsets of $B\subseteq D$   as 
$$
\mu_h(B) = \int_{B\cap J_{w_h}} \phi(x,\nu_{J_{w_h}}^{})d\cH^{d-1},\qquad h\ge1,
$$
and 
$$
\mu_0(B) = \int_{B\cap J_w} \phi(x,\nu_{J_w}^{})d\cH^{d-1}. 
$$
Since $\sup_h  \mu_h(D)<\infty,$ by compactness, there exist a positive Radon measure $\mu$ and a not relbelled subsequence $\{\mu_h\}_{h\ge1}$ such that $\mu_h\wk*\mu$ as $h\to\infty.$ 
We prove that 
\begin{equation}\label{asdgh}
\mu\ge \mu_0, 
\end{equation}
in particular from $\mu(D) \ge\mu_0(D)$ \eqref{lsc_jumpillo} follows.
Since $\mu_0$ is absolutely continuous with respect to $\cH^{d-1}\res J_w,$ to prove \eqref{asdgh} we need only to show 
\begin{equation}\label{maucci_bola}
\frac{d\mu}{d\cH^{d-1}\res J_w}(x) \ge \phi(x,\nu_{J_w})\qquad\text{for $\cH^{d-1}$-a.e.\ $x\in J_w.$} 
\end{equation}
For this aim fix  $\epsilon\in(0,c_1).$ By the uniform continuity of $\phi,$ 
there exists $r_\epsilon>0$ such that 
\begin{equation}\label{ahjadas}
|\phi(x,\nu) - \phi(y,\nu)|\le \epsilon 
\end{equation}
for any $\nu\in\S^{d-1}$ and $x,y\in D$ with $|x-y|<r_\epsilon.$ In particular, given $x\in J_w$ and for a.e.\ $r\in (0,r_\epsilon),$
\begin{align*}
\mu(B_r(x_0)) = & \lim\limits_{h\to\infty} \mu_h(B_r(x_0)) \\
\ge& \liminf\limits_{h\to\infty} \int_{B_r(x_0)\cap J_{w_h}} \phi(x_0,\nu_{J_{w_h}})d\cH^{d-1} - \epsilon \limsup\limits_{h\to\infty} \cH^{d-1}(B_r(x_0)\cap J_{w_h}),  
\end{align*}
where in the equality we use the weak* convergence of $\{\mu_h\}$ and in the inequality \eqref{ahjadas} with $y=x_0$ and $x\in B_r(x_0)\cap J_{w_h}$  
By Proposition \ref{prop:anisotropic_chambolle} applied with $\phi(x_0,\cdot),$
\begin{align*}
 \liminf\limits_{h\to\infty} \int_{B_r(x_0)\cap J_{w_h}} \phi(x_0,\nu_{J_{w_h}})d\cH^{d-1} \ge& \int_{B_r(x_0)\cap J_w} \phi(x_0,\nu_{J_w})d\cH^{d-1}\\
 \ge & \mu_0(B_r(x_0)) - \epsilon\cH^{d-1}(B_r(x_0)\cap J_w),  
\end{align*}
where in the second equality we again used \eqref{ahjadas}.
Moreover, by \eqref{finsler_norm9},
$$
\limsup\limits_{h\to\infty} \cH^{d-1}(B_r(x_0)\cap J_{w_h})  \le \frac{1}{c_1}\,\limsup\limits_{h\to\infty} \mu_h(B_r(x_0))= \frac{\mu(B_r(x_0))}{c_1}
$$
and 
$$
\cH^{d-1}(B_r(x_0)\cap J_w) \le \frac{\mu_0(B_r(x_0))}{c_1},
$$
thus, 
$$
\mu(B_r(x_0)) \ge \frac{c_1-\epsilon}{c_1+\epsilon}\,\,\mu_0(B_r(x_0)).
$$
Since $\epsilon$ and $r\in (0,r_\epsilon)$ are arbitrary, 
\eqref{maucci_bola} follows from the Besicovitch Derivation Theorem. \black 
\end{proof}

\begin{lemma}[\textbf{Creation of delamination energy}]
\label{lem:creation_of_delamination}
Let $\phi$ be as in Lemma \ref{lem:creation_filament}
and suppose that $\Omega_k\subset U_4$ is a sequence of 
Lipschitz sets,  $E_k\subset \Omega_k,$ 
$J_{E_k}\in\cJ_{E_k},$ and
$g_0,g_1\in[0,+\infty),$
$u_k\in GSBD^2(U_4;\R^2)$
and $u^\pm\in\R^2$ with $u^+\ne u^-$ are such that  
\begin{itemize}
\item[\rm (a)] $\sdist(\cdot, U_4 \cap \p\Omega_k)\to \sdist(\cdot,\p U_4^+)$ uniformly in $U_{3/2}$;

\item[\rm (b)] $\Sigma_k:=U_4\cap \p \Omega_k$ is a graph of a Lipschitz function  $l_k:{I_4}\to\R$ such that $l_k(0)=0$ and  
$|l_k'|\le \frac1k;$

\item[\rm (c)] $\sdist(\cdot, U_4 \cap \p E_k)\to \sdist(\cdot,\p U_4^+)$ uniformly in $U_{3/2}$;

\item[\rm (d)]   there exists $m_k\in\N_0$ such that the number 
of connected components of each $\p E_k$ lying strictly inside $\Omega_k$ does not exceed $m_k;$ 

\item[\rm (e)] $|g_1 - g_0|\le \phi({\bf e_2});$ 

\item[\rm (f)] $J_{u_k}\subset (\Omega_k\cap \p E_k)\cup J_{E_k}\cup L_k,$ where $L_k\subset\Sigma_k$ is a relatively open subset of $\Sigma_k$ with $\cH^1(L_k)<1/k;$ 

\item[\rm (g)] $\sup\limits_{k\ge1} \int_{U_4}|\str{u_k}|^2dx < \infty;$

\item[\rm (h)] $u_k\to u^+$ a.e.\ in $U_1^+$ and 
$u_k\to u^-$ a.e.\ in $U_1\setminus \overline{U_1^+}.$
\end{itemize}
Then for every $\delta\in (0,1)$  there exists $k_\delta>1$ for which 
\begin{align} \label{jump_estimate_blow_up}
 \int_{U_1\cap \Omega_{k} \cap\p^*E_{k}} \phi(\nu_{E_{k}})d\cH^1 
&+2\int_{U_1\cap \Omega_{k} \cap (E_{k}^{(1)}\cup E_k^{(0)}) \cap\p E_{k}}
\phi(\nu_{E_{k}})d\cH^1\nonumber \\  
& + \int_{U_1\cap \Sigma_{k}\setminus \p E_{k}} g_0 d\cH^1 
+ \int_{U_1\cap \Sigma_{k}\cap E_{k}^{(0)}\cap \p E_{k}} \big(\phi(\nu_\Sigma)
+ g_1\big)d\cH^1\nonumber \\
& +  \int_{U_1\cap \Sigma_{k}\cap \p^*E_{k}\setminus J_{E_{k}}} g_1d\cH^1 +
\int_{U_1\cap J_{E_{k}}} \big(\phi(\nu_\Sigma) + g_0  \big)\,d\cH^1\no\\ 
&\ge 
\int_{I_1}\big(\phi({\bf e_2}) + g_0\big)\,d\cH^1- \delta 
\end{align} 
for any $k>k_\delta.$
\end{lemma}

\begin{proof}
Denote the left-hand-side of \eqref{jump_estimate_blow_up} by $\alpha_k.$ We suppose that $\sup_k |\alpha_k|<\infty$ so that by  \eqref{phi_cond1} 
\begin{equation}\label{shdge}
\sup\limits_k \cH^1(\p E_k)\le M 
\end{equation}
for some  $M>0.$ Moreover, passing to a not relabelled subsequence if necessary, we assume that 
$$
\liminf\limits_{k\to\infty} \alpha_k = \lim\limits_{k\to\infty} \alpha_k.
$$
By assumption (b), $\Sigma_k$ is ``very close'' $I_2,$ hence, by the area formula \cite[Theorem 2.91]{AFP:2000} for any $K_k\subset \Sigma_k$ one has  
$$
\begin{aligned}
&\limsup\limits_{k\to\infty} 
\Big|\int_{\pi(K_k)}\phi({\bf e_2})d\cH^1 - \int_{K_k}
\phi(\nu_{\Sigma_k})d\cH^1 \Big|\\
&\le 
\limsup\limits_{k\to\infty} 
 \int_{\pi(K_k)}|\phi({\bf e_2}) -\phi(l_k',1)| d\cH^1 \le \limsup\limits_{k\to\infty} 
 \int_{\pi(K_k)}|\phi({\bf e_2} -(l_k',1))| d\cH^1 \\
 & = \limsup\limits_{k\to\infty} 
 \int_{\pi(K_k)}\phi(0,1)\,|l_k'| d\cH^1 =0,
\end{aligned}
$$
where in the last inequality and in the first equality we used that $\phi$ is a norm, and the last equality follows from $|l_k'|\le \frac1k.$
Hence,
\begin{equation}\label{projection_close_toProjected}
\lim\limits_{k\to\infty} 
\Big|\int_{\pi(K_k)}\phi({\bf e_2})d\cH^1 - \int_{K_k}
\phi(\nu_{\Sigma_k})d\cH^1 \Big| = 0.   
\end{equation}

We divide the proof into two steps.

{\it Step 1.} For shortness, let  $J_k:=J_{E_k}$ and $C_k:=\Sigma_k\setminus\overline{\p^*E_k}.$ We claim that for any $\delta\in(0,1)$ there exists $k_\delta^1>0$  such that 
for any $k>k_\delta^1,$ 
\begin{align}\label{c_k_1ning_ulchovi}
 \int_{U_1\cap \Omega_k \cap\p^*E_k} & \phi(\nu_{E_k})d\cH^1
+2\int_{U_1\cap \Omega_k \cap (E_{k}^{(1)}\cup E_k^{(0)}) \cap\p E_k}
\phi(\nu_{E_k})d\cH^1 \no \\
&\ge 
2 \int_{U_1\cap \Sigma_k\setminus (J_k\cup C_k)} \phi(\nu_{\Sigma_k})d\cH^1
+\int_{U_1\cap C_k} \phi(\nu_{\Sigma_k})d\cH^1- \delta.
\end{align}

Indeed, by adding to both sides of \eqref{c_k_1ning_ulchovi} the quantity 
$2\int_{U_1\cap J_k} \phi(\nu_{\Sigma_k})d\cH^1
+\int_{U_1\cap C_k} \phi(\nu_{\Sigma_k})d\cH^1, $ \eqref{c_k_1ning_ulchovi} is equivalent to
\begin{align}\label{c_k_1ning_ulchovi1}
\int_{U_1\cap ((\Omega_k \cap\p^*E_k) \cup C_k)}  \phi(\nu_{\p E_k\cup C_k})d\cH^1
&+2\int_{U_1\cap \Omega_k \cap (((E_{k}^{(1)}\cup E_k^{(0)}) \cap\p E_k)\cup J_k)}
\phi(\nu_{E_k})d\cH^1 \no \\
&\ge 
2 \int_{U_1\cap \Sigma_k} \phi(\nu_{\Sigma_k})d\cH^1- \delta,
\end{align}
and hence, we will prove \eqref{c_k_1ning_ulchovi1}. 

Note that since $J_k\subset\Sigma_k$ is $\cH^1$-rectifiable, given $\delta\in(0,1)$ there exists a finite union $R_k$ of intervals of $\Sigma_k$ such that 
\begin{equation}\label{recall_bu}
J_k\cup L_k\subset R_k\qquad\text{and}\qquad \cH^1(R_k\setminus (J_k\cup L_k))<\frac{\delta}{5c_2}, 
\end{equation}
where $c_2>0$ is given in \eqref{phi_cond1}.
Possibly slightly modifying $u_k$ around the (approximate) continuity points of $R_k$ and around the boundary of the voids $U_2\setminus E_k$ we assume that $J_k:=R_k,$ $L_k=\emptyset$ and  $J_{u_k}=(\Omega_k\cap \p E_k)\cup C_k\cup J_k$ (up to a $\cH^1$-negligible set).

%

Let $K_k=U_1\cap (\overline{\Omega_k\cap  \partial E_k}\cup \overline{J_k\cup C_k}).$ By relative openness of $C_k=\Sigma_k\setminus \cl{\p^*E_k}$ and $J_k$ in $\Sigma_k$ and assumption (d), $K_k$ is a union $\bigcup_{j} K_k^j$ of at most countably many  pairwise disjoint connected rectifiable sets $K_k^j$ relatively closed in $U_1.$ 

Let $\co{K_k^j}$ denote the closed convex hull of $K_k^j.$  Observe that if $K_k^j$ is not a segment, then the interior of $\co{K_k^j}$ is non-empty and 
\begin{align}\label{karna_tuni}
 \int_{K_k^j\cap (\p^*E_k\cup C_k)} \phi(\nu_{K_k})d\cH^1 
+ 2\int_{K_k^j\cap (((E_{k}^{(1)}\cup E_k^{(0)}) \cap\p E_k)\cup J_k)}
\phi(\nu_{K_k})d\cH^1 \nonumber\\ 
\ge \int_{\p \co{K_k^j}}
\phi(\nu_{\co{K_k^j}})d\cH^1. 
\end{align}

Now we define the minimal union of disjoint closed convex sets containing $K_k$ as follows.
For every $k\ge1$ let us define the sequences $\{D_i^k\}_i$ of pairwise disjoint subsets of $\N$ and $\{V_i^k\}_i$ of pairwise disjoint closed convex subsets of $\overline{U_1}$ as follows.  
Let $D_0^k:=\{1\}$ and $V_0^k:=\emptyset.$ Suppose that for some $i\ge1$ the sets $D_0^k,\ldots,D_{i-1}^k$ and $V_0^k,\ldots,V_{i-1}^k$ are defined and
let $j_o$ be the smallest element of $\N\setminus \bigcup\limits_{j=0}^{i-1}D_j^k.$ Define 
$$
D_i^k:=\Big\{h\in \N\setminus \bigcup\limits_{j=0}^{i-1}D_j^k:\, \co{K_k^{j_o}}\cap \co{K_k^h}\ne\emptyset\Big\} 
$$
and 
$$
V_i^k:=\co{\cup_{h\in D_i^k}\, K_k^h}.
$$
Note that $j_o\in D_i^k.$
As in \eqref{karna_tuni} we observe that  
\begin{align}\label{basme_basme}
\sum\limits_{h\in D_i^k} \int_{K_k^h\cap (\p^*E_k\cup C_k)} \phi(\nu_{K_k})d\cH^1 
+ 2\int_{K_k^h\cap (((E_{k}^{(1)}\cup E_k^{(0)}) \cap\p E_k)\cup J_k)}  
\phi(\nu_{K_k})d\cH^1 \nonumber \\
 \ge \int_{\p V_i^k}
\phi(\nu_{V_i^k}^{})d\cH^1.
\end{align}
Then $K_k\subset \bigcup\limits_i V_i^k$ and by \eqref{basme_basme} 
\begin{align*}
\int_{K\cap (\p^*E_k\cup C_k)} \phi(\nu_{K_k})d\cH^1 
+ &2\int_{K_k\cap (((E_{k}^{(1)}\cup E_k^{(0)}) \cap\p E_k)\cup J_k)}
\phi(\nu_{K_k})d\cH^1\\
&\ge \sum\limits_{i\notin T}\int_{\p V_i^k}
\phi(\nu_{V_i^k}^{})d\cH^1 + 2\sum\limits_{i\in T}\int_{\p V_i^k} \phi(\nu_{V_i^k}^{})d\cH^1, 
\end{align*}
where $T$ is the set of all indices $i$ for which $V_i^k$ is a line segment. For every $i\in T$ we replace the segment $V_i^k$ with a closed rectangle $Q_i$ containing $V_i^k$ and not intersecting any $V_j^k,$ $j\ne i,$ such that 
$$
2\int_{V_i^k}
\phi(\nu_{V_i^k}^{})d\cH^1\ge \int_{\p Q_i}
\phi(\nu_{Q_i})d\cH^1 -\frac{\delta}{10\cdot2^i}
$$
Therefore, redefining  $V_i^k:=Q_i$ we obtain 
\begin{align}\label{hohohasd}
\int_{K_k\cap (\p^*E_k\cup C_k)} \phi(\nu_{K_k})d\cH^1 
+ 2\int_{K_k\cap (((E_{k}^{(1)}\cup E_k^{(0)}) \cap\p E_k)\cup J_k)}
\phi(\nu_{K_k})d\cH^1 \nonumber\\ 
\ge \sum\limits_{i}\int_{\p V_i^k}
\phi(\nu_{V_i^k}^{})d\cH^1 -\frac{\delta}{5}. 
\end{align}
Note that $U_1\setminus \bigcup\limits_i V_i^k$ is a Lipschitz open set and $J_{u_k}\cap (U_1\setminus \bigcup\limits_i V_i^k) =\emptyset,$ and hence, by the Poincar\'e-Korn inequality, $u_k\in H^1(U_1\setminus \bigcup\limits_i V_i^k).$ Moreover,   \eqref{shdge}, \eqref{phi_cond1} and \eqref{hohohasd} imply   
$c_1\sum\limits_i \cH^1(\p V_i^k)<M+1,$
thus, there exists $\eta\in\R^2$ such that the set 
$$
\{x\in U_1\cap \bigcup\limits_i \p V_i^k:\,\text{trace of $u_k\big|_{U_1\setminus \cup V_i^k}$ is equal to $\eta$}\}
$$
is $\cH^1$-negligible (see \cite[Proposition 2.6]{Ma:2012}). Therefore, $v_k:=u_k\chi_{U_1\setminus \cup V_i^k}^{} +\eta \chi_{\cup V_i^k}^{}$ belongs to $GSBD^2(U_1;\R^2),$ $J_{v_k}= U_1\cap \cup_i \p V_i^k.$ By assumptions (a), (c) and (h), $v_k\to v:=u^+\chi_{U_1^+} +u^-\chi_{U_1\setminus U_1^+},$ and by assumption (g) and inequalities \eqref{phi_cond1},   and \eqref{shdge},
$$
\sup\limits_k \int_{U_1}|\str{v_k}|^2dx + \cH^1(J_{v_k}) \le 
\sup\limits_k \int_{U_1}|\str{u_k}|^2dx + \frac{M+1}{c_1}<\infty. 
$$
Repeating the same arguments of the proof of \eqref{liminf_katta_sasa} we obtain 
\begin{equation}\label{kerakli_nuqta}
2\int_{I_1}\phi({\bf e_2})d\cH^1 \le  \liminf\limits_{k\to\infty} \int_{J_{v_k}}\phi(\nu_{J_{v_k}})d\cH^1.
\end{equation}
Note that the direct application of Proposition \ref{prop:anisotropic_chambolle} would not be enough since 
we would obtain the estimate:
$$
\int_{I_1}\phi({\bf e_2})d\cH^1 \le  \liminf\limits_{k\to\infty} \int_{J_{v_k}}\phi(\nu_{J_{v_k}})d\cH^1 
$$
without coefficient $2$ on the left. 

From \eqref{hohohasd} and \eqref{kerakli_nuqta} it follows that there exists $k_\delta^1>0$ such that  
\begin{align}\label{ahteyrea}
\int_{U_1\cap ((\Omega_k\cap \p^*E_k)\cup C_k)} \phi(\nu_{\p E_k\cup C_k}^{})d\cH^1 
+ & 2\int_{U_1\cap (((E_{k}^{(1)}\cup E_k^{(0)}) \cap\p E_k)\cup J_k)}
\phi(\nu_{E_k}^{})d\cH^1\no \\
&\ge 2 \int_{I_1}\phi({\bf e_2})d\cH^1 - \frac{2\delta}{5} 
\end{align} 
for any $k\ge k_\delta^1.$
By \eqref{projection_close_toProjected} we may suppose that for such $k,$ 
$$
\int_{I_1}\phi({\bf e_2})d\cH^1 \ge \int_{U_1\cap \Sigma_k} \phi(\nu_{\Sigma_k})d\cH^1 -\frac{\delta}{5},
$$
thus, in view of \eqref{recall_bu}, from  \eqref{ahteyrea} we get \eqref{c_k_1ning_ulchovi1}.


{\it Step 2.} Finally we prove \eqref{jump_estimate_blow_up}. Let $k_{\delta/2}^1$ be given by Step 1 with $\delta/2$ in place of $\delta.$ From \eqref{c_k_1ning_ulchovi} for $k>k_{\delta/2}^1$  we have 
\begin{align}\label{raja_or_rani}  
\int_{U_1\cap \Omega_k \cap\p^*E_k} \phi(\nu_{E_k})d\cH^1 
+2\int_{U_1\cap \Omega_k \cap (E_{k}^{(1)}\cup E_k^{(0)}) \cap\p E_k}
\phi(\nu_{E_k})d\cH^1\nonumber \\  
 + \int_{U_1\cap \Sigma_k\setminus \p E_k} g_0 d\cH^1 
+ \int_{U_1\cap \Sigma_k\cap E_k^{(0)}\cap \p E_k} \big(\phi(\nu_{\Sigma_k})
+ g_1\big)d\cH^1\nonumber \\
 +  \int_{U_1\cap \Sigma_k\cap \p^*E_k\setminus J_{E_k}} g_1d\cH^1 +
\int_{U_1\cap J_{E_k}} \big(\phi(\nu_{\Sigma_k}) + g_0  \big)\,d\cH^1\no\\ 
\ge  \int_{U_1\cap \Sigma_k\setminus \p E_k} \big(\phi(\nu_{\Sigma_k}) +g_0\big) d\cH^1 
+ \int_{U_1\cap \Sigma_k\cap E_k^{(0)}\cap \p E_k} \big(2\phi(\nu_{\Sigma_k})
+ g_1\big)d\cH^1\nonumber \\
  +  \int_{U_1\cap \Sigma_k\cap \p^*E_k\setminus J_{E_k}} \big(2\phi(\nu_{\Sigma_k})+g_1\big)d\cH^1 +
\int_{U_1\cap J_{E_k}} \big(\phi(\nu_{\Sigma_k}) + g_0  \big)\,d\cH^1- \frac{\delta}{2}.
\end{align} 
Thus, \eqref{projection_close_toProjected} implies that there exists $k_\delta>k_{\delta/2}^1$ such that 
\begin{align*}
&\int_{U_1\cap \Sigma_k\setminus \p E_k} \big(\phi(\nu_{\Sigma_k}) +g_0\big) d\cH^1 
+ \int_{U_1\cap \Sigma_k\cap E_k^{(0)}\cap \p E_k} \big(2\phi(\nu_{\Sigma_k})
+ g_1\big)d\cH^1\nonumber \\
& +  \int_{U_1\cap \Sigma_k\cap \p^*E_k\setminus J_{E_k}} \big(2\phi(\nu_{\Sigma_k})+g_1\big)d\cH^1 +
\int_{U_1\cap J_{E_k}} \big(\phi(\nu_{\Sigma_k}) + g_0  \big)\,d\cH^1\\
& \ge \int_{I_1\cap \pi(\Sigma_k\setminus \p E_k)} \big(\phi({\bf e_2}) +g_0\big) d\cH^1 
+ \int_{I_1\cap \pi(\Sigma_k\cap E_k^{(0)}\cap \p E_k)} \big(2\phi({\bf e_2})
+ g_1\big)d\cH^1\nonumber \\
& +  \int_{I_1\cap \pi(\Sigma_k\cap \p^*E_k\setminus J_{E_k})} \big(2\phi({\bf e_2})+g_1\big)d\cH^1 +
\int_{I_1\cap \pi(J_{E_k})} \big(\phi({\bf e_2}) + g_0  \big)\,d\cH^1-\frac{\delta}{2}\\
&\ge \int_{I_1} \big(\phi({\bf e_2}) +g_0\big) d\cH^1 -\frac{\delta}{2}, 
\end{align*}
where we used also that $g_0$ and $g_1$ are constants. Now \eqref{jump_estimate_blow_up} follows from  \eqref{raja_or_rani} and the last inequality.
\end{proof}
\black 

\begin{proof}[Proof of Proposition \ref{prop:lsc_surface_energy}]
Without loss of generality, we suppose that the 
limit in the left-hand side of \eqref{final_call_lsc} is reached and finite. 
Define
$$
g_+(x,s) = g(x,s)  - g(x,0) + \varphi(x,\nu_\Sigma(x)).
$$
Then $g_+$ is Borel, $g_+(\cdot,s)\in L^1(\Sigma)$ for $s=0,1,$ 
and by \eqref{beta_condition}, $g_+\ge0$ and 
\begin{equation}\label{g_qushuv_condition}
|g_+(x,1) - g_+(x,0)| \le \varphi(x,\nu_\Omega (x)) 
\end{equation}
for $\cH^1$-a.e.\ $x\in\Sigma.$ 
Consider the sequence $\mu_k$ of Radon measures  in $\R^2,$
associated to $\cS(A_k,J_{A_k};\varphi,g),$ defined
at Borel sets $B\subset\R^2$ by
\begin{align*}
\mu_k(B):=& \int_{B\cap \Omega \cap\p^*A_k} \varphi(x,\nu_{A_k})d\cH^1 
+2\int_{B\cap \Omega \cap (A_k^{(1)}\cup A_k^{(0)})\cap\p A_k}
\varphi(x,\nu_{A_k})d\cH^1\nonumber \\  
 + &\int_{B\cap \Sigma\setminus \p A_k} g_+(x,0)d\cH^1 
+ \int_{B\cap \Sigma\cap A_k^{(0)}\cap \p A_k} \big(\varphi(x,\nu_\Sigma)
+ g_+(x,1)\big)d\cH^1\nonumber \\
 + &\int_{B\cap \Sigma\cap \p^*A_k\setminus J_{A_k}} g_+(x,1)d\cH^1 +
\int_{B\cap J_{A_k}} \big(\varphi(x,\nu_\Sigma) + g_+(x,0)  \big)\,d\cH^1. 
\end{align*}
Analogously, we define the positive Radon measure $\mu$ in $\R^2$
associated to $\cS(A,J_A;\varphi,g),$ writing $A$ in place of 
$A_k$ in the definition of $\mu_k.$
By \eqref{finsler_norm},  assumption $A_k\overset{\tau_\fA}{\to} A$
and the nonnegativity of $g_+,$  
$$
\sup\limits_{k\ge1} \mu_k(\R^2) \le 2c_2 \sup\limits_{k\ge1}
\cH^1(\p A_k) + 
\sum\limits_{s=0}^1 \int_\Sigma g_+(x,s)d\cH^1<\infty. 
$$
Thus, by compactness  there exists a 
(not relabelled) subsequence $\{\mu_k\}$
and a non-negative bounded Radon measure $\mu_0$ in $\R^2$ such 
that $\mu_k\overset{*}\wk \mu_0$ as $k\to\infty.$ 
We claim that
\begin{equation}\label{mu_0_katta_mu}
 \mu_0\ge \mu,
\end{equation}
which implies the assertion of the proposition.
In fact, \eqref{final_call_lsc} follows from \eqref{mu_0_katta_mu}, 
the weak*-convergence of $\mu_k,$
and the equalities
$$
\mu_k(\R^2) = \cS(A_k,J_{A_k};\varphi,g) + 
\int_{\Sigma} \big(\varphi(x,\nu_\Sigma) -
g(x,0)\big)d\cH^1
$$
and 
$$
\mu(\R^2) = \cS(A,J_{A};\varphi,g) + 
\int_{\Sigma} \big(\varphi(x,\nu_\Sigma) -
g(x,0)\big)d\cH^1.
$$ 
Since $\mu_0$ and $\mu$ are 
non-negative, and $\mu<<\cH^1\res (\p A\cup\Sigma),$ 
by Remark \ref{rem:structure_of_pA} 
to prove \eqref{mu_0_katta_mu} it suffices
to establish the following lower-bound estimates for densities
of $\mu_0$ with respect to $\cH^1$ restricted to various parts of $\p A:$
\begin{subequations}\label{eq:litdiff}
\begin{align}
& \frac{d \mu_0}{d \cH^1\res(\Omega \cap \p^*A)}\,(x) 
\ge \varphi(x,\nu_A(x))\qquad\text{for 
$\cH^1$-a.e.\ $x\in \Omega \cap \p^*A,$} 
\label{eq:at_reduced_boundary}\\ 
& \frac{d \mu_0}{d \cH^1\res(A^{(0)}\cap \p A)}\,(x)\ge 2\varphi(x,\nu_A(x)) 
\qquad \text{for $\cH^1$-a.e.\ $x\in \Omega \cap  A^{(0)}\cap \p A,$}
\label{eq:at_external_filament}\\
& \frac{d \mu_0}{d \cH^1\res(A^{(1)}\cap \p A)}\,(x)\ge 2\varphi(x,\nu_A(x)) 
\qquad \text{for $\cH^1$-a.e.\ $x\in \Omega \cap  A^{(1)}\cap \p A,$}
\label{eq:at_internal_crack}\\
& \frac{d \mu_0}{d \cH^1\res (\Sigma\setminus \p A)}\,(x) = 
g_+(x,0) \qquad  
\text{for $\cH^1$-a.e.\ $x\in \Sigma\setminus \p A,$ }
\label{eq:at_not_boundary_of_A}\\
& \frac{d \mu_0}{d \cH^1\res (\Sigma\cap A^{(0)}\cap \p A)}\,(x) \ge 
\varphi(x,\nu_{\Sigma}(x)) + g_+(x,1)\nonumber \\
&\hspace{6cm}\text{for $\cH^1$-a.e.\ $x\in \Sigma\cap A^{(0)}\cap \p A,$ }
\label{eq:filament_on_Sigma}\\
& \frac{d\mu_0}{d\cH^1\res(\Sigma\cap \p^*A)} \ge
g_+(x,1)\qquad 
\text{for $\cH^1$-a.e.\ $x\in \Sigma\cap\p^*A,$ }
\label{eq:contact_of_A}\\
& \frac{d \mu_0}{d \cH^1\res J_A}\,(x) \ge 
\varphi(x,\nu_\Omega (x)) + g_+(x,0) \qquad 
\text{for $\cH^1$-a.e.\ $x\in J_A.$}
\label{eq:at_delaminations}
\end{align}
\end{subequations}

We separately outline below the proofs of
\eqref{eq:at_reduced_boundary}-\eqref{eq:at_delaminations}. 

{\it Proof of \eqref{eq:at_reduced_boundary}.} Consider points 
$x\in \Omega\cap \p^*A$ such that 
\begin{itemize}
\item[(a1)] $\nu_A(x)$ exists;

\item[(a2)] $x$ is a Lebesgue point of 
$y\in \p^*A \mapsto \varphi(y,\nu_A(y)),$ i.e.,
$$
\lim\limits_{r\to0} \frac{1}{2r}\int_{U_r\cap\p^*A} |\varphi(y,\nu_A(y))- \varphi(x,\nu_A(x))|d\cH^1(y)=0; 
$$

\item[(a3)] $\frac{d\mu_0}{d\cH^1\res(\Omega\cap\p^*A)}(x)$ exists and is finite.
\end{itemize}
By the definition of $\p^*A,$ continuity of $\phi,$ the Borel regularity of $y\in \p^*A \mapsto \varphi(y,\nu_A(y)),$
and the Besicovitch Derivation Theorem, the set of points $x\in \Omega\cap \p^*A$ not satisfying these conditions is $\cH^1$-negligible, hence we prove \eqref{eq:at_reduced_boundary}
for $x\in\Omega\cap\p^*A$ satisfying (a1)-(a3). Without loss of generality we suppose $x=0$ and $\nu_A(x) = {\bf e_2}.$
By Lemma \ref{lem:f_convergence}, $A_k\to A$ in $L^1(\R^2),$  therefore, $D\chi_{A_k} \overset{*}\wk  D\chi_A,$ and hence, by the Besicovitch Derivation Theorem \cite[Theorem 2.22]{AFP:2000} and the definition  \eqref{essential_boundary00} of the reduced boundary, 
$$
\nu_{A_k}\cH^1\res(\Omega \cap \p^*A_k) \overset{*}\wk  
\nu_A\cH^1\res(\Omega \cap \p^*A).
$$
Then for a.e.\ $r>0$ such that $U_r\strictlyincluded \Omega$
and $\cH^1(\p U_r\cap \p A)=0,$ the Reshetnyak Lower-semicontinuity Theorem  \cite[Theorem 2.38]{AFP:2000}  implies 
\begin{align*}
\mu_0(U_r) = & \liminf\limits_{k\to\infty} \mu_k(U_r) 
\ge \liminf\limits_{k\to\infty} \int_{U_r\cap \p^*A_k} 
\varphi(y,\nu_{A_k})\,d\cH^1 
\ge  \int_{U_r\cap \p^*A} \varphi(y,\nu_A)\,d\cH^1
\end{align*}

Therefore, by \cite[Theorem 1.153]{FL:2007} and assumption (a2),
\begin{align*}
\frac{d\mu_0}{d\cH^1\res(\Omega\cap\p^*A)}\,(0) = 
\lim\limits_{r\to 0 }  \frac{\mu_0(U_r)}{2r}    
\ge  \liminf\limits_{r\to0} \frac{1}{2r} 
\int_{U_r\cap \p^*A} \varphi(y,\nu_A)\,d\cH^1
= \varphi(0,{\bf e_2}).
\end{align*}
\smallskip

{\it Proof of \eqref{eq:at_external_filament}.}
Consider points $x\in \Omega\cap A^{(0)}\cap\p A$ such that 
\begin{itemize}
\item[(b1)] $\theta^*(\p A,x) = \theta_*(\p A,x)=1;$

\item[(b2)] $\nu_A(x)$ exists;

\item[(b3)] $\cl{U_1}\cap \sigma_{\rho,x}(\p A)\overset{\cK}{\to} \cl{U_1}\cap T_x,$
where $T_x$ is the approximate tangent line to $\p A$
and $\sigma_{\rho,x}$ is given by \eqref{blow_up_map};

\item[(b4)] $\frac{d\mu_0}{d\cH^1\res(A^{(0)}\cap\p A)}$ exists and finite.
\end{itemize}
By the $\cH^1$-rectifiability of $\p A,$ Proposition \ref{prop:tangent_line_conv} (applied with the closed connected component $K$ of $\p A$ containing $x$) and the Besicovitch Derivation Theorem,  the set of points  $x\in A^{(0)}\cap \p A$ not satisfying these conditions is $\cH^1$-negligible, hence we prove \eqref{eq:at_external_filament} for $x\in A^{(0)}\cap \p A$ satisfying (b1)-(b4).  Without loss of generality we assume $x=0$, $\nu_A(x) = {\bf e_2}$ and $T_x = T_0$ is the $x_1$-axis.

Let us choose a sequence $\rho_n\searrow 0$ such that 
\begin{equation}\label{som_conditi}
\mu_0(\p U_{\rho_n})=0\qquad\text{and}\qquad  \lim\limits_{k\to\infty} 
\mu_k(\cl{U_{\rho_n}}) = \mu_0(U_{\rho_n}) 
\end{equation}
and 
\begin{equation}\label{hosila_filament}
\frac{d\mu_0}{d\cH^1\res(A^{(0)}\cap \p A)}(0) = \lim\limits_{n\to\infty}
\frac{\mu_0(U_{\rho_n})}{2\rho_n}.
\end{equation}
By Proposition \ref{prop:set_shrinks} (a), (b2) and (b3) imply that 
$\sdist(\cdot,\sigma_{\rho_n}(\p A))\to \dist(\cdot,T_0)$ uniformly in $\cl{U_1}.$ 
Since for any  $n>1,$ $\sdist(\cdot,\p A_k)\to \sdist(\cdot,\p A)$ uniformly in 
$\cl{U_{\rho_n}}$ as $k\to\infty,$  
by a  diagonal argument,   
we find a subsequence $\{A_{k_n}\}$ such that 
$$
\sdist(\cdot,\sigma_{\rho_n} (\p A_{k_n}))\to \dist(\cdot,T_0)\quad\text{uniformly in $\cl{U_1},$} 
$$
as $n\to\infty$ and 
\begin{equation}\label{shutter_island}
\mu_{k_n}(\cl{U_{\rho_n}}) \le \mu_0(U_{\rho_n}) + \rho_n^2  
\end{equation}
for any $n.$
By Lemma \ref{lem:set_shrinks_to_line},
$\cl{U_1}\cap \sigma_{\rho_n}(A_{k_n}) \overset{\cK}\to I_1:=\cl{U_1}\cap T_0.$

From \eqref{finsler_norm},  \eqref{shutter_island}, 
the definition of $\mu_k,$ \eqref{hosila_filament} and (b4) it follows that 
\begin{equation}\label{aviator}
\begin{aligned}
\limsup\limits_{n\to\infty}\frac{\cH^1(\cl{U_\rho} \cap \p A_{k_n})}{2\rho_n} 
\le & 
\limsup\limits_{n\to\infty} \frac{\mu_{k_n}(\cl{U_{\rho_n}})}{2c_1\rho_n}  
\le c_1^{-1}\,\frac{d\mu_0}{\cH^1\res(A^{(0)}\cap \p A)}(0).
\end{aligned}
\end{equation}
By the uniform continuity of $\varphi,$ for every $\epsilon>0$
there exists $n_\epsilon>0$ such that 
\begin{equation}\label{titanik}
\varphi(y,\xi) \ge \varphi(0,\xi) - \epsilon  
\end{equation}
for every $y\in U_{\rho_{n_\epsilon}}.$
Moreover, 
since $\{A_k\}\subset \fA_m,$ the number of
connected components of $\p \sigma_{\rho_n}(A_{k_n})$ 
lying strictly inside $U_1,$ does 
not exceed from $m.$ Hence, applying Lemma \ref{lem:creation_filament}
with $\phi=\varphi(0,\cdot),$ $m_o=m$ and $\delta=\epsilon,$ 
we find $n_\epsilon'>n_\epsilon$
such that for any $n>n_\epsilon',$
\begin{align*}
\int_{U_1\cap \p^*\sigma_{\rho_n}(A_{k_n})} &
\varphi(0,\nu_{\sigma_{\rho_n}(A_{k_n}) })\,d\cH^1\\
&+ 
2\int_{U_1\cap \big((\sigma_{\rho_n}(A_{k_n}))^{(0)}\cup (\sigma_{\rho_n}(A_{k_n}))^{(1)}\big)\cap \p \sigma_{\rho_n}(A_{k_n})} 
\varphi(0,\nu_{\sigma_{\rho_n}(A_{k_n}) })\,d\cH^1\\
&\ge  
2 \int_{I_1}
\varphi(0,{\bf e_2})\,d\cH^1-\epsilon = 4\varphi(0,{\bf e_2})-\epsilon.
\end{align*}
Therefore, by the definition of $\mu_k,$ for such $n$ one has
\begin{align}\label{asaloyoqpyuere}
\mu_{k_n}(\cl{U_{\rho_n}})
\ge & 
\int_{\cl{U_{\rho_n}}\cap \p^* A_{k_n}} 
\varphi(y,\nu_{A_{k_n}})\,d\cH^1 
+ 
2\int_{\cl{U_{\rho_n}}\cap \big(A_k^{(0)}\cup A_k^{(1)}\big)\cap \p A_{k_n}} 
\varphi(y,\nu_{A_{k_n}})\,d\cH^1\no  \\
\ge  &
\int_{\cl{U_{\rho_n}}\cap \p^* A_{k_n}} 
\varphi(0,\nu_{A_{k_n}})\,d\cH^1 
+ 
2\int_{\cl{U_{\rho_n}}\cap \big(A_k^{(0)}\cup A_k^{(1)}\big)\cap \p A_{k_n}} 
\varphi(0,\nu_{A_{k_n}})\,d\cH^1\no \\
&- \epsilon\,\cH^1(\cl{U_{\rho_n}}\cap \p A_k)\no \\
= & \rho_n \Big(\int_{U_1\cap \p^*\sigma_{\rho_n}(A_{k_n})} 
\varphi(0,\nu_{\sigma_{\rho_n}(A_{k_n}) })\,d\cH^1\no \\
&+ 
2\int_{U_1\cap \big((\sigma_{\rho_n}(A_{k_n}))^{(0)}\cup (\sigma_{\rho_n}(A_{k_n}))^{(1)}\big)\cap \p \sigma_{\rho_n}(A_{k_n})} 
\varphi(0,\nu_{\sigma_{\rho_n}(A_{k_n}) })\,d\cH^1\Big)\no \\
& - \epsilon\,\cH^1(\cl{U_{\rho_n}}\cap \p A_{k_n}) \no \\
\ge & 4\rho_n\varphi(0,{\bf e_2}) -\epsilon\rho_n
- \epsilon\,\cH^1(\cl{U_{\rho_n}}\cap \p A_{k_n}) , 
\end{align}
and thus, by
\eqref{hosila_filament}-\eqref{aviator}, 
\begin{align*}
\frac{d\mu_0}{d\cH^1\res(A^{(0)}\cap \p A)}(0) \ge &  
\liminf\limits_{n\to\infty}  \frac{\mu_{k_n}(\cl{U_{\rho_n}})}{2\rho_n} \\
\ge & 2\varphi(0,{\bf e_2}) -\frac{\epsilon}{2} -
\epsilon\,\limsup\limits_{n\to\infty}\frac{\cH^1(\cl{U_{\rho_n}}\cap \p A_k)}
{2\rho_n}\\
\ge &2\varphi(0,{\bf e_2}) -\frac{\epsilon}{2} -
c_1^{-1}\epsilon\, \frac{d\mu_0}{d\cH^1\res(A^{(0)}\cap \p A)}(0).
\end{align*}
Now using assumption (b4) and 
letting $\epsilon\to0^+$ we obtain \eqref{eq:at_external_filament}.
\smallskip

{\it Proof of \eqref{eq:at_internal_crack}}. 
We repeat the same arguments of the proof of 
\eqref{eq:at_external_filament} using Lemma \ref{lem:creation_crack}
in place of Lemma \ref{lem:creation_filament} and Proposition \ref{prop:set_shrinks} (a) in place of Proposition \ref{prop:set_shrinks} (b).
\smallskip 

{\it Proof of \eqref{eq:at_not_boundary_of_A}.}  
Given $x\in \Sigma\setminus \p A,$ there exists $r_x>0$ such that 
$B_{r_x}(x)\cap \p A = \emptyset.$ Since $\p A_k\overset{\cK}\to\p A,$
there exists $k_x$ such that $B_{r_x/2}(x)\cap \p A_k = \emptyset$
for all $k>k_x.$ Thus, for any $r\in (0,r_x/2),$
$$
\mu_k(B_r(x)) = \int_{\Sigma\cap B_r(x)} g_+(y,0)d\cH^1
$$
so that 
$$
\frac{d \mu_0}{d \cH^1\res (\Sigma\setminus \p A)}\,(x)= 
g_+(x,0)
$$
for $\cH^1$-a.e.\ Lebesgue points $x\in \Sigma\setminus \p A$
of $g_+.$
\smallskip

{\it Proof of \eqref{eq:filament_on_Sigma}.}  
Consider points $x\in \Sigma\cap A^{(0)}\cap \p A$ such that
\begin{itemize}
\item[(e1)] $\theta^*(\Sigma\cap \p A,x) = \theta_*(\Sigma\cap\p A,x)=1;$

\item[(e2)] $\nu_\Sigma(x)$ and $\nu_A(x)$ exist (clearly, either $\nu_\Sigma(x)=\nu_A(x)$ or $\nu_\Sigma(x) = -\nu_A(x)$);

\item[(e3)] $\cl{U_1}\cap \sigma_{\rho,x}(\p A)\overset{\cK}{\to} \cl{U_1}\cap T_x,$
where $T_x$ is the approximate tangent line to $\p A;$

\item[(e4)] $x$ is a Lebesgue point of $g_+(\cdot,1),$ i.e., 
$$
\lim\limits_{\rho\to0} \,\frac{1}{2\rho}\int_{\Sigma\cap U_{\rho,\nu_\Sigma(x)}} 
|g_+(y,1)  - g_+(x,1)|d\cH^1(y) = 0;
$$

\item[(e5)] $x$ is a Lebesgue point of $y\in\Sigma\cap \varphi(y,\nu_\Sigma(y)),$
i.e., 
$$
\lim\limits_{\rho\to0} \,\frac{1}{2\rho}\int_{\Sigma\cap U_{\rho,\nu_\Sigma(x)}} 
|\varphi(y,\nu_\Sigma(y))  - \varphi(x,\nu_\Sigma(x))|d\cH^1(y) = 0;
$$

\item[(e6)] $\frac{d \mu_0}{d \cH^1\res \Sigma}\,(x)$ exists and is finite.
\end{itemize}
By the $\cH^1$-rectifiability of $\p A,$ the Lipschitz continuity of $\Sigma$, the Borel regularity of $\nu_\Sigma(\cdot),$ Proposition \ref{prop:tangent_line_conv} (applied with closed connected component $K$ of $\p A$ containing $x$), the continuity of $\varphi$, assumptions on 
$g_+$ and the Besicovitch Derivation Theorem, the set of $x\in \Sigma\cap \p A$ not  satisfying these conditions is $\cH^1$-negligible. Hence, we prove \eqref{eq:filament_on_Sigma} 
for $x$ satisfying (e1)-(e6). Without loss of generality we assume $x=0,$ $\nu_\Sigma(x) = \nu_A(x) = {\bf e_2}$ and $T_x=T_0$ is the $x_1$-axis. Let $r_n\searrow0$  be such that 
$$
\mu_0(\p U_{r_n}) = \cH^1(\p U_{r_n}\cap \Sigma) =0 
$$
and 
\begin{equation}\label{limithada}
\frac{d \mu_0}{d \cH^1\res (\Sigma\cap A^{(0)}\cap \p A)}\,(x) = 
\lim\limits_{n\to\infty}
\frac{\mu_0(U_{r_n})}{2r_n}.  
\end{equation}
By the weak*-convergence, for any $h\ge1$ we have
$$
\lim\limits_{k\to\infty} \mu_k(\overline{U_{r_n}}) = \mu_0(U_{r_n}).  
$$
By Proposition \ref{prop:set_shrinks} (b), (e2) and (e3) imply 
$\sdist(\cdot, \sigma_{r_n}(\p A)) \to \dist(\cdot, T_0)$ uniformly in $\cl{U_1}.$
Since for any $n,$
$\sdist(\cdot,\sigma_{r_n}(\p A_k)) \to \sdist(\cdot,\sigma_{r_n}(\p A))$ uniformly in $\cl{U_1}$ as $k\to\infty,$ 
by a diagonal argument, we can find a subsequence 
$\{k_n\}$ and not relabelled subsequence $\{r_n\}$ such that 
\begin{equation}\label{danflkjegeqw}
\mu_{k_n}(\overline{U_{r_n}}) \le \mu_0(U_{r_n}) + 
r_n^2  
\end{equation}
for any $n\ge1$ and $\sdist(\cdot,\sigma_{r_n}(A_k)) \to \sigma(\cdot,T_0)$ uniformly in $\cl{U_1}$ as $k\to\infty,$ 
thus, by Lemma \ref{lem:set_shrinks_to_line}, 
\begin{equation}\label{kur_konb_dafg}
U_1 \cap \sigma_{r_n}(A_{k_n})  \overset{\cK}\to I_1:=U_1\cap T_0
\end{equation}
as $n\to\infty.$ 
Notice also that by (e2) and Proposition \ref{prop:tangent_line_conv} (applied with the closed connected component $K$ of $\Sigma$),
$U_1\cap \sigma_{r_n}(\Sigma)\overset{\cK}{\to} I_1$ as $n\to\infty.$

\black 
By \eqref{g_qushuv_condition},
$$
\varphi(y,\nu_\Sigma(x))+g_+(y,0)\ge g_+(y,1)
$$
for $\cH^1$-a.e.\ on $\Sigma,$ in particular on $J_{A_k},$
hence, by  Remark \ref{rem:structure_of_pA}  and the definition of 
$\mu_k,$
\begin{align*}
\mu_{k_n}(\overline{U_{r_n}}) \ge & 
\int_{\overline{U_{r_n}} \cap \Omega \cap\p^*A_{k_n}} 
\varphi(y,\nu_{A_{k_n}})d\cH^1 
+2\int_{\overline{U_{r_n}} \cap \Omega \cap \big(A_{k_n}^{(0)} \cup A_{k_n}^{(1)}\big)\cap\p A_{k_n}} 
\varphi(y,\nu_{A_{k_n}})d\cH^1\\  
+ & \int_{U_{r_n}\cap \Sigma} g_+(y,1)d\cH^1
+ \int_{U_{r_n} \cap \Sigma\cap A_{k_n}^{(0)}\cap \p A_{k_n}} 
\varphi(y,\nu_\Sigma)d\cH^1\\
 +&\int_{U_{r_n} \cap \Sigma\setminus \p A_{k_n}} 
\big(g_+(y,0) - g_+(y,1)\big)d\cH^1
\end{align*}
Adding and subtracting 
$\int_{\overline{U_{r_n}} \cap \Sigma\cap \p^* A_{k_n}}
\phi(y,\nu_{A_{k_n}})d\cH^1$ to the right 
and using \eqref{g_qushuv_condition} once more in the 
integral over $U_{r_n} \cap \Sigma\setminus \p A_{k_n}$
we get 
\begin{align}\label{mu_nB_n}
\mu_{k_n}(\overline{U_{r_n}}) \ge &  \int_{\overline{U_{r_n}} \cap \p^*A_{k_n}}
\varphi(y,\nu_{A_{k_n}})d\cH^1 
+2\int_{\overline{U_{r_n}} \cap\big( A_{k_n}^{(0)}\cup A_{k_n}^{(1)}\big) \cap\p A_{k_n}} 
\varphi(y,\nu_{A_{k_n}})d\cH^1\nonumber \\  
& + \int_{U_{r_n}\cap \Sigma} g_{\red + \black}(y,1)d\cH^1 
- \int_{U_{r_n} \cap \Sigma} 
\varphi(y,\nu_\Sigma(y))d\cH^1 
\end{align}
By the uniform continuity of
$\varphi,$ given $\epsilon\in(0,1)$ there exists $n_\epsilon>0$ such that 
$$
|\varphi(y,\nu) -  \varphi(0,\nu)| < \epsilon
$$
for all $y\in U_{r_n},$ $\nu\in\S^1$ and $n>n_\epsilon.$
We suppose also that Lemma \ref{lem:creation_filament}
holds with $n_\epsilon$ when $\delta=\epsilon.$ 
Since the number of connected components of $\p A_{k_n}$ 
lying strictly inside $U_{r_n}$ is not greater than $m,$
in view of \eqref{kur_konb_dafg} and the non-negativity of 
$g_+,$ as in \eqref{asaloyoqpyuere} for all $n>n_\epsilon$ we obtain
\begin{align}\label{kuch_bahoas} 
\mu_{k_n}(\overline{U_{r_n}}) \ge &
4r_n \varphi(0,{\bf e_2}) - \epsilon r_n - 
\epsilon\cH^1(\overline{U_{r_n}\cap \p A_{k_n}})\no \\
& + \int_{U_{r_n}\cap \Sigma} g_{\red + \black}(y,1)d\cH^1 
- \int_{U_{r_n} \cap \Sigma} 
\varphi(y,\nu_\Sigma(y))d\cH^1. 
\end{align}
By the non-negativity of $g_+,$ \eqref{finsler_norm}  and \eqref{danflkjegeqw},
$$
\begin{aligned}
\cH^1(\overline{U_{r_n}\cap \p A_{k_n}}) \le &  \cH^1(U_{r_n}\cap\Omega\cap \cap \p A_{k_n}) + \cH^1(U_{r_n}\cap\Sigma\cap \cap \p A_{k_n})\\
\le &\frac{\mu_{k_n}(U_{r_n})}{c_1} +   \cH^1(U_{r_n}\cap\Sigma)\le \mu_0(U_{r_n}) +r_n^2 +\cH^1(U_{r_n}\cap\Sigma) 
\end{aligned}
$$
thus again using \eqref{danflkjegeqw}, also  \eqref{limithada}, 
\eqref{mu_nB_n} and \eqref{kuch_bahoas}, as well as (e1) and (e3)-(e5)  we establish
\begin{align*}
&\frac{d\mu_0}{d\cH^1\res(\Sigma\cap A^{(0)}\cap\p A)}(0)
= \lim\limits_{h\to\infty} \frac{\mu_0(U_{r_n})}{2r_n} 
\ge  \limsup\limits_{h\to\infty} 
\frac{\mu_{k_n}(U_{r_n})}{2r_n} \\
& \ge 2\varphi(0,{\bf e_2}) - \frac{3\epsilon}{2} - 
\epsilon\frac{d\mu_0}{d\cH^1\res(\Sigma\cap A^{(0)}\cap\p A)}(0)
 + g_+(0,1)d\cH^1 
- \varphi(0,\nu_\Sigma(0)).
\end{align*}
Now  letting  $\epsilon\to0$ and using $\nu_\Sigma(0)={\bf e_2}$
we obtain \eqref{eq:filament_on_Sigma}.
\smallskip

{\it Proof of \eqref{eq:contact_of_A}.}
Since  $g_+$ is non-negative and  $\chi_{A_k}\to \chi_A$ in $L^1(\R^2),$ the inequality follows from \cite[Lemma 3.8]{ADT:2017} \red applied to $u_k=\chi_{A_k},$ $u=\chi_A$ and
$$
\cF_U(v,B) := \int_{U\cap B\cap J_v\setminus \Sigma} \varphi(x,\nu_{J_v})d\cH^1 + \int_{\Sigma\cap B} g_+(x,1)v^+(x)\,d\cH^1, 
$$
where $U:=\Omega$. More precisely, by \eqref{beta_condition},
\begin{equation}\label{passage_to_dalmaso}
\mu_k(B) \ge  \cF_U(\chi_{A_k},B), 
\end{equation}
so that the total variation of $\cF_U(\chi_{A_k},\cdot)$ is uniformly bounded. Therefore,  passing to a further not relabelled subsequence if necessary, we have $\cF_U(\chi_{A_k},\cdot) \wk^* \bar\mu$ as $k\to\infty$ for some bounded positive Radon measure $\bar\mu$ in $\R^2.$ By \cite[Lemma 3.8, Eq. 3.15]{ADT:2017},
$$
\frac{d\bar\mu}{d\cH^1\res\Sigma}(x) \ge g_+(x,1)\chi_{\p^*A}\qquad\text{for $\cH^1$-a.e. $x\in\Sigma$}
$$
By \eqref{passage_to_dalmaso}, $\mu_0\ge \bar\mu,$ and thus, \eqref{eq:contact_of_A} follows.
\black 

\smallskip 

{\it Proof of \eqref{eq:at_delaminations}.} 
Consider points $x\in J_A$ for which  
\begin{itemize}
\item[(g1)] $\theta^*(J_A,x) = \theta_*(J_A,x) =
\theta^*(\p A,x) = \theta_*(\p A,x) =\theta^*(\Sigma,x) = \theta_*(\Sigma,x) =1;$
\smallskip

\item[(g2)] assumption (b) of Proposition \ref{prop:lsc_surface_energy} 
holds with some $r=r_x>0$;
\smallskip

\item[(g3)] $\Sigma$ is differentiable at $x$ and 
$\nu_\Sigma(x)$ exists;
\smallskip 

\item[(g4)] one-sided  traces 
$w^+(x)\ne w^-(x)$ 
of 
$w,$ given by assumption (b) of 
Proposition \ref{prop:lsc_surface_energy}, exists;
\smallskip

\item[(g5)] $x$ is a Lebesgue point of $g_+(\cdot,s)$ and 
$|g_+(x,0) - g_+(x,1)|\le \phi(\nu_\Sigma(x));$
\smallskip

\item[(g6)] $\frac{d\mu_0}{d\cH^1\res J_A}(x)$ exists and finite.
\end{itemize}
 By the  $\cH^1$-rectifiability of $J_A,$ $\p A$ and $\Sigma,$
assumption (b) of Proposition \ref{prop:lsc_surface_energy} (recall that $J_A\subset J_w$), the definition of the jump set of $GSBD$-functions,  \eqref{g_qushuv_condition},  
and the Besicovitch Derivation Theorem, 
the set of points  $x\in J_A$ not satisfying these conditions is
$\cH^1$-negligible. Hence we prove \eqref{eq:at_delaminations}
for $x\in J_A$ satisfying (g1)-(g6).  Without loss of generality, we assume $x=0$ and $\nu_\Sigma(x) = {\bf e_2}.$ Let $r_0=r_x$ and $w_k\in GSBD^2(B_{r_0}(0);\R^2)$ be given by 
assumption (b) of Proposition \ref{prop:lsc_surface_energy}. Note that by the weak*-convergence of $\mu_k,$
\begin{equation}\label{eq_weak_nmeasdas}
\lim\limits_{k\to\infty} \mu_k(\overline{U_r}) = \mu_0(U_r). 
\end{equation}
for a.e.\ $r\in(0,r_0),$ and by (g1), (g3), and Proposition \ref{prop:tangent_line_conv} (applied with connected components of $\Sigma$ and $\p A$ intersecting at $x$) and also by the definition of blow-up,   
\begin{equation}\label{sigma_blow}
U_4 \cap \sigma_{r}(\Sigma)\overset{\cK}{\to} I_4 
\qquad\text{and}\qquad
\cH^1\res(U_4 \cap \sigma_{r}(\Sigma))\overset{*}{\wk} \cH^1\res I_4, 
\end{equation}
and
\begin{equation}\label{A_blow_up}
U_4 \cap \sigma_{r}(\p A)\overset{\cK}{\to} I_4 
\qquad \text{and }\qquad \cH^1\res (U_4 \cap \sigma_{r_h}(\p A))\overset{*}{\wk} \cH^1\res I_4
\end{equation}
as $r\to0.$ Since $J_A\subset\Sigma,$ in view of \eqref{sigma_blow},
\begin{equation}\label{jump_blow_up1}
U_4 \cap \sigma_{r}(J_A)\overset{\cK}{\to} I_4.
\end{equation}
Moreover, since  $J_A$ has a generalized normal at $x=0,$
\begin{equation}\label{jump_blow_up2}
\cH^1\res(U_4 \cap \sigma_{r}(J_A))\overset{*}{\wk} \cH^1\res I_4  
\end{equation}
as $r\to0.$ In particular, from \eqref{sigma_blow} and \eqref{A_blow_up},
\begin{equation}\label{omega_r_to_half}
\sdist(\cdot, U_4 \cap \sigma_r(\p \Omega))\to \sdist(\cdot,\p U_4^+) 
\end{equation}
and
\begin{equation}\label{A_r_to_half}
\sdist(\cdot, U_4 \cap \sigma_r(\p A))\to \sdist(\cdot,\p U_4^+)
\end{equation}
locally uniformly in $U_{3/2}$ as $r\to0.$

Letting $\phi=\varphi(0,\cdot),$ \emph{we claim} that there exist sequences $r_h\searrow 0$ and $k_h\nearrow\infty$ such that the sets 
$$
\Omega_h:= U_4 \cap \sigma_{r_h}(\Omega),\quad
\Sigma_h:= U_4 \cap \sigma_{r_h}(\Sigma),\quad 
$$
and 
$$
E_h:= U_4 \cap \sigma_{r_h}(A_{k_h}),\qquad 
J_{E_h} =  U_4 \cap \sigma_{r_h}(J_{A_{k_h}}),
$$ 
the functions $u_h(x) = w_{k_h}(r_hx)\in GSBD^2( U_4; \R^2),$  the numbers $g_s = g_+(0,s)\in[0,+\infty)$ and the vectors $u^\pm=w^\pm(0)$ satisfy assumptions (a)-(h) of Lemma \ref{lem:creation_of_delamination}.

Indeed, let $\tau$  be any homeomorphism between $\R^2$ and a bounded subset of $\R^2;$  for example, one can take $\tau(x_1,x_2) = (\tanh(x_1),\tanh(x_2)).$ By \eqref{condition_J_A}, $w_k(rx)\to w(rx)$ as $k\to\infty$ for a.e.\ $x\in U_4 $ and for any $r\in(0,r_0/4),$ so that by the Dominated Convergence Theorem, 
\begin{equation}\label{w_k_to_w}
\lim\limits_{k\to\infty} \int_{U_1} |\tau(w_k(rx)) - \tau(w(rx))|dx=0. 
\end{equation}
Moreover, by (g4), the definition \cite[Definition 2.4]{D:2013} of the (approximate) jump of the  function $w$ and \cite[Remark 2.2]{D:2013},  
\begin{equation}\label{w_r_to_w_0}
\lim\limits_{r\to0}\int_{U_1} |\tau(w(rx)) - (\tau(u(x)) |dx=0,
\end{equation} 
where $u(x):=w^+(0)\chi_{U_1^+}(x) + w^-(0)\chi_{U_1 \setminus U_1^+}(x).$  We use \eqref{w_k_to_w} and \eqref{w_r_to_w_0} to extract sequences $k_h\to\infty$ and $r_h\to0$ such that $w_{k_h}(r_hx) \to u(x)$ a.e.\ in $U_1.$  By assumption $A_k\overset{\tau_\fA}{\to}A$ and the relations  \eqref{A_r_to_half}, \eqref{eq_weak_nmeasdas}, \eqref{w_r_to_w_0} and \eqref{w_k_to_w}, there exist  $k_h^1>1$  and a decreasing sequence  $r_h\in(0,\frac1h)$ such that for any $h>1$ and $k>k_h^1,$
\begin{subequations}\label{pass_k_h_and_r_h}
\begin{align}
& \|\sdist(\cdot,\sigma_{r_h}( U_{4r_h} \cap \p A_k)) - 
\sdist(\cdot,\sigma_{r_h}( U_{4r_h} \cap \p A))
\|_{L^\infty( U_{3/2})}^{} < \frac 1h   \label{A_k_blow_conver} \\[1mm]
& \|\sdist(\cdot,\sigma_{r_h}( U_{4r_h} \cap \p A)) - \sdist(\cdot,\p U_4^+ )
\|_{L^\infty(U_{3/2})}^{} < \frac 1h   \label{A_blow_conver} \\[1mm]
& \mu_k(\overline{U_{r_h}}) <  \mu_0(U_{r_h}) +r_h^2  \label{measure_conver}
\\[1mm]
& \int_{U_1} |\tau(w(r_hx)) - \tau(u(x)) |dx <\frac1h  \label{w_r_h_to_u} \\[1mm]
&
\int_{U_1} |\tau(w_k(r_hx)) - \tau(w(r_hx)) |dx <\frac1h.\label{w_k_to_w_0_a} 
\end{align}
\end{subequations}
For every $h\ge1,$ we choose $k_h>k_h^1$ such that 
\begin{equation}\label{k_h_ni_top}
\frac{1}{k_hr_h} <\frac{1}{h}. 
\end{equation}
Now $u_h(x): = w_{k_h}(r_hx)\in GSBD^2(U_2;\R^2),$  and:  
\begin{itemize}[leftmargin = \dimexpr+4mm]
\item[--] by \eqref{omega_r_to_half}, 
$\sdist(\cdot, U_4\cap  \p \Omega_h)\to 
\sdist(\cdot,\p U_4^+)$  locally uniformly in $U_{3/2}$ as $h\to\infty;$

\item[--] by assumption (g3) and the Lipschitz property of $\Sigma,$ $U_{4r_h}\cap \Sigma$ is a graph of a $L$-Lipschitz function $l:[-4r_h,4r_h]\to\R$ so that $\Sigma_h= U_4 \cap \p\Omega_h$ is the graph of $l_h(t):=l(r_ht),$ where $t\in [-4,4],$ so that $l_h(0)=0$ and $|l_h'|\le \frac{L}{h}$ by choice \eqref{k_h_ni_top} of $r_h;$

\item[--] by \eqref{A_k_blow_conver} and \eqref{A_blow_conver},
$\sdist(\cdot,U_4\cap  \p E_h)\to \sdist(\cdot,\p U_4^+)$ as $h\to+\infty;$

\item[--] by assumption $A_k\in\fA_m,$ the number of connected components of $\p E_h$ lying strictly inside $U_4$ does not exceed $m;$

\item[--] by (g5), $|g_1-g_0|\le \phi({\bf e_2});$

\item[--] by \eqref{condition_J_A}, $J_{u_h}\subset 
(\Omega_h \cap \p E_h) \cup J_{E_h}\cup \hat L_h,$ where by \eqref{k_h_ni_top}, $\hat L_h:=\sigma_{r_h}({U_{4r_h}}\cap L_{k_h})$ satisfies $\cH^1(\hat L_h)<\frac1h;$
 
\item[--]  since
$$
\int_{U_4}|\str{u_h}|^2 dx \le \int_{U_{4r_h}}|\str{w_{k_h}}|^2dx \le 
\int_{B_{r_0}(0)} |\str{w_{k_h}}|^2dx, 
$$
by \eqref{condition_J_A}, we have
$
\sup\limits_{h\ge1} |\str{u_h}|^2dx<\infty;
$
\item[--] by \eqref{w_r_h_to_u}-\eqref{w_k_to_w_0_a},
$$
\lim\limits_{h\to\infty} \int_{U_1} |\tau(u_h(x)) - \tau(u(x)) |dx =0,
$$
thus, possibly passing to further not relabelled subsequence,
$u_h\to u=u^+\chi_{U_1^+}+u^-\chi_{U_1\setminus U_1^+}$
a.e.\ in $U_1.$
\end{itemize}
This implies the claim.

Now we prove \eqref{eq:at_delaminations}.  Given $\delta\in(0,1),$  by
the continuity of 
$\varphi,$ (g5) and (g6), 
there exists $h_\delta^1>1$ such that 
\begin{equation}\label{sadadadaad}
\mu_{k_h}(\overline{U_{r_h}}) \ge \hat\mu_{k_h}(\overline{U_{r_h}}) -
2\delta \cH^1(U_{r_h}\cap (\p A\cup\Sigma))  
\end{equation}
for all  $h>h_\delta^1,$
where $\hat \mu_k$ is defined exactly the same as $\mu_k$
with $\phi$ and $g_s$ in place of $\varphi$ and $g_+(x,s),$
here we used (g5) as 
$$
\int_{U_{r_h}\cap\Sigma} |g_+(x,s) -g_+(0,s)|d\cH^1 \le\frac{\delta}{4}\,\cH^1(U_{r_h}\cap\Sigma) 
$$
provided \red that \black $h$ is large enough.
By Lemma \ref{lem:creation_of_delamination}, there exists $h_\delta^2>h_\delta^1$ such that 
for any $h>h_\delta^2,$
$$
\frac{1}{r_h}\,\hat\mu_{k_h}(\overline{U_{r_h}}) \ge 2\phi({\bf e_2}) + 2g_0 -2 \delta.
$$
Moreover, by \eqref{finsler_norm} and non-negativity of $g_+,$ 
$$
\cH^1(U_{r_h}\cap (\p A\cup\Sigma) = \cH^1(U_{r_h}\cap\Omega\cap \p A) + \cH^1(U_{r_h}\cap \Sigma\cap \p A) \le \frac{\mu_{k_h}(\overline{U_{r_h}})}{c_1} + \cH^1(U_{r_h}\cap \Sigma).
$$
Thus, by \eqref{sadadadaad} for any $h>h_\delta^2,$ 
$$
\Big(1+ \frac{\delta}{c_1} \Big)\,\frac{\mu_{k_h}(\overline{U_{r_h}})}{2r_h} + 2\delta\,\frac{\cH^1(U_{r_h}\cap\Sigma)}{2r_h} \ge 
\phi({\bf e_2}) + g_0 -  \delta.
$$
From here and \eqref{measure_conver} we get 
$$
\phi({\bf e_2}) + g_0 \le   \delta  +  
\Big(1+ \frac{\delta}{c_1} \Big)\,\Big(\frac{\mu_0(U_{r_h})}{2r_h} +\frac{r_h}{2}\Big)+2 \delta\,\frac{\cH^1(U_{r_h}\cap\Sigma)}{2r_h},
$$
therefore, first letting $h\to\infty,$ then $\delta\to0$, 
and using (g6) we obtain \eqref{eq:at_delaminations}. 
\black 
\end{proof}

Now we address the lower semicontinuity  of $\cF.$
We start with the following auxiliary extension result.

\begin{lemma}\label{lem:extension_desparate}
Let $\openset\subset\R^n$ and $\anyset\supset \openset$ be  non-empty bounded connected Lipschitz open sets and let $\cE:H^1(\openset;\R^2)\to H^1(\anyset;\R^n)$ be the Sobolev extension map, i.e., a bounded linear operator such that for any $v\in H^1(\openset;\R^n),$ $Ev=v$ a.e.\ in $\openset$ and 
there exists $C_\openset>0$ such that $\|\cE v\|_{H^1(\anyset)}\le C_P\|v\|_{H^1(\openset)}.$  
Consider any $\{u_k\}\subset H^1(\openset;\R^n)$ such that 
\begin{equation}\label{unif_b_has}
\sup\limits_{k} \int_\openset |\str{u_k}|^2dx <\infty 
\end{equation}
and $u_k\to u$ a.e.\ in $\openset$ for some function $u:\openset\to\R^n.$ 
Then there exist a subsequence $\{u_{k_l}\}_l$ and $v\in H^1(Q;\R^n)$ such that $v=u$ a.e.\ in $\openset$ and  $\cE u_{k_l}\to v$ in $L^2(\anyset)$ 
and 
$$
\sup\limits_{l}  \|\cE u_{k_l}\|_{H^1(\anyset)}<\infty. 
$$
\end{lemma}

\begin{proof}
By Proposition \ref{prop:yo_cheksiz_yo_finite},  $u\in H_\loc^1(\openset;\R^n)\cap GSBD^2(\openset;\R^n).$  By  Poincar\'e-Korn inequality, there exist $c_\openset>0$ and  a sequence $\{a_k\}$ of rigid displacements  such that 
\begin{equation}\label{poincare_korn}
\|u_k+a_k\|_{H^1(\openset)} \le c_\openset\|\str{u_k}\|_{L^2(\openset)} 
\end{equation}
for any $k.$ Since $u_k\to u$ a.e.\ in $\openset$, reasoning as in the proof of Proposition \ref{prop:yo_cheksiz_yo_finite} (with $\openset$ in place of $B_\epsilon$), up to a not relabelled subsequence, $a_k\to a$ a.e.\ in $\R^n$ for some rigid displacement $a:\R^n\to\R^n.$ In particular, $H^1(\openset;\R^n)$-norm of $a_k$ is uniformly bounded independently of $k,$ hence, 
\begin{equation}\label{a_k_hhsjs}
\sup_k \|\cE a_k\|_{H^1(\anyset)}<\infty. 
\end{equation}
Since
\begin{equation}
\|\cE(u_k + a_k)\|_{H^1(\anyset)}\le C_\openset \|u_k+a_k\|_{H^1(\openset)} \le C_\openset c_\openset\|\str{u_k}\|_{L^2(\openset)},
\end{equation}
by \eqref{unif_b_has}, the linearity of $E$ and \eqref{a_k_hhsjs}, 
$$
\sup_k \|\cE u_k\|_{H^1(\anyset)}<\infty.
$$
Thus, by the Rellich-Kondrachov Theorem, there exists a not relabelled subsequence $\{u_k\}$ and $v\in H^1(\anyset;\R^n)$ such that $\cE u_k\to v$ in $L^2(\anyset)$ and a.e.\ in $\anyset.$%
\end{proof}
\black 

\begin{proof}[Proof of Theorems \ref{teo:lsemico_film}]
Without loss of generality,
we assume that 
\begin{equation}\label{bounded_energy_iofe}
\sup\limits_{k\ge1} \int_{A_k\cup  \substrate} |\str{u_k}|^2dx + 
\cH^1(\p A_k)<\infty, 
\end{equation}
The lower semicontinuity of the elastic-energy part can be shown by using convexity $W(x,\cdot)$. Indeed, let $D\strictlyincluded \Int{A}.$ Then by $\tau_\fA$-convergence of $A_k,$ $D\strictlyincluded \Int{A_k}$ for all large $k.$ Since $u_k\to u$ a.e. in $A\cup \substrate,$  by \eqref{bounded_energy_iofe} and the weak-compactness of $L^2(D\cup\substrate)$, $\str{u_k}\wk\str{u}$ in $L^2(D\cup\substrate).$ Therefore, from the convexity of $\cW(D,\cdot)$ it follows that 
$$
\cW(D,u) \le 
\liminf\limits_{k\to\infty} \cW(D,u_k) \le  \liminf\limits_{k\to\infty} \cW(A_k,u_k). 
$$
Now letting $D\nearrow A\cup\substrate$ we get 
$$
\cW(A,u) \le \liminf\limits_{k\to\infty} \cW(A_k,u_k). 
$$
\black 

Since 
$\cS(E,v) = \cS(E,J_v;\varphi,g)$ 
with $J_E=J_v$ and 
$g(x,s) = \beta(x)s,$
the lower semicontinuity of 
of the surface part, follows from Proposition \ref{prop:lsc_surface_energy}
provided  that for 
$\cH^1$-a.e.\ $x\in J_u$ there exists $r_x>0$, 
$w_k\in GSBD(B_{r_x}(x);\R^2)$ and relatively open sets 
$L_k$ of $\Sigma$ with $\cH^1(L_k)<1/k$ 
such that  
\eqref{condition_J_A} holds.
Let
$$
r_0^x:=\frac14\,\min\{
\dist(x,\p\Omega \setminus \Sigma),
\dist(x,\p  \substrate \setminus \Sigma)\}
$$
so that $B_{r_0^x}(x)\strictlyincluded \Omega \cup\Sigma\cup  \substrate,$ 
and choose  $r=r_x\in(0,r_0^x)$ such that 
$$
\begin{aligned}
\cH^1(\p B_r(x)\cap \p A_k) = &\cH^1(\p B_r(x)\cap J_{u_k}) \\
= &
\cH^1(\p B_r(x)\cap \p A) = \cH^1(\p B_r(x)\cap J_u) = 0 
\end{aligned}
$$
(see \cite[Proposition 2.6]{Ma:2012}) and $B_r(x)\cap \substrate$ is connected.
We construct $\{w_k\}$ by extending
$\{u_k\}$ in $B_r(x)\setminus (A_k\cup \substrate)$ without creating extra jumps at the interface on the exposed surface of the substrate.  More precisely, we apply Lemma \ref{lem:extension_desparate} with
$\anyset:= B_r(x),$  $\openset:=B_r(x)\cap \substrate,$ and $u_k\big|_\openset.$ Since $u_k\to u$ a.e.\ in $\openset,$  by Lemma \ref{lem:extension_desparate}, there exist $v\in H^1(\anyset;\R^2)$ and a not relabelled subsequence $\{u_k\}$ such that the  Sobolev extension $\cE u_k$ of $u_k\big|_\openset$ to $\anyset$ converges to $v$ a.e.\ in $\anyset.$ Define 
$$
w_k:=u_k\chi_{B_r(x)\cap (A_k\cup \substrate)} + \cE u_k\chi_{B_r(x)\setminus (A_k\cup \substrate)}.
$$
Perturbing $w_k$ slightly if necessary, we can assume $J_{w_k}=\Gamma:=B_r\cap (J_{u_k} \cup (\Omega\cup \p^*A_k) \cup (A_k^{(1)}\cap \p A_k))$ up to a $\cH^1$-negligible set. In fact, by \cite[Proposition 2.6]{Ma:2012} there exist $\xi\in\R^2$ with arbitrarily small $|\xi|>0$ for which $\cH^1(\{y\in \Gamma:\,\, [u_k](y)=\xi\})=0$ (with $[u_k](x)$ the size of the jump of $u_k$), and hence, we can perturb $u_k$ with a $W^{1,\infty}(B_r(x)\setminus \Gamma)$-function with arbitrarily small norm, which is equal to $\xi$ on an arbitrarily large subset of $\Gamma.$  
By construction, 
$$
w_k\to w:=u\chi_{B_r(x)\cap (A\cup\substrate)} + v\chi_{B_r(x)\cap (A\cup\substrate)},
$$ 
thus, by \cite[Theorem 1.1]{ChC:2018jems}, $w\in GSBD^2(B_r(x);\R^2).$ Notice also that $J_u\subset J_w$ since $w=u$ a.e.\ in $B_r(x)\cap (A\cup \substrate).$   
\black 
Thus $w_k$ and $w$ satisfy \eqref{condition_J_A}.
\end{proof}

We conclude this section 
by proving a lower semicontinuity property of $\cF'$
with respect to $\tau_\admissible'$.
Observe that if
$(A_k,u_k)\overset{\tau_\admissible'}{\to} (A,u),$ then
$A =\Int{A}$ 
so that the weak convergence of $u_k$ to $u$ in
$H_\loc^1(A\cup \substrate;\R^2)$ is well-defined.
However, notice that 
$\admissible_m'$ 
is not closed with respect to $\tau_{\admissible}'$-convergence.

\begin{proposition}[\textbf{Lower 
semicontinuity of $\cF'$}]\label{prop:lsemico_film_voids}
Assume {\rm (H1)-(H3)}.
If $(A_k,u_k)\in\admissible_m'$ and $(A,u)\in\admissible$ are such that 
$(A_k,u_k)\overset{\tau_{\admissible}'}{\to} (A,u),$ 
then 
\begin{equation}\label{laasjkadbfaj}
\liminf\limits_{k\to\infty} \cF'(A_k,u_k) \ge \cF'(A,u). 
\end{equation}
\end{proposition}

\begin{proof}
Consider the auxiliary functional $\tilde \cF:\admissible\to\R$ defined as 
$$
\tilde \cF(A,u) = \cF(A,u) - \int_{\Sigma \cap A^{(0)} \cap \p A} \big(\phi(x,\nu_A) 
+ \beta\big) d\cH^1.
$$
Since $\tilde \cF$ does not see wetting layer energy,
\begin{equation}\label{f_prime_dan_f_ga}
\cF'(G,u) = \tilde \cF(G,u) - \int_\Sigma\beta d\cH^1  =
\cF(G\cup\Sigma,u) - \int_\Sigma\beta d\cH^1 
\end{equation}
for any $G\in\fA_m':=\{A\in\fA:\,\,A\cup\Sigma\in\fA_m\}.$
Repeating the proof of
Theorem \ref{teo:lsemico_film} one can readily show that $\tilde \cF$
is also $\tau_\admissible$-lower semicontinuous.

Now we prove \eqref{laasjkadbfaj}.
Without loss of generality we suppose that 
{\it liminf} is a finite limit.
Let $E_k:=A_k\cup\Sigma.$ By the definition 
of $\fA_m'$ and $\tau_\admissible'$-convergence,
$\{E_k\}\subset\fA_m$ and 
$\sup\cH^1(\p E_k)<\infty,$ therefore by Proposition
\ref{prop:compactness_A_m}, there exist a (not relabelled)
subsequence and $E\in \fA_m$
such that $E_k\overset{\tau_\admissible}{\to} E.$
By  Remark \ref{rem:kuratiwski_and_sdistance},
$A = \Int{E},$ thus, by \eqref{f_prime_dan_f_ga}, 
\begin{align*}
\lim\limits_{k\to\infty} \cF'(A_k,u_k) &= 
\lim\limits_{k\to\infty} \tilde \cF(A_k\cup\Sigma,u_k) - \int_\Sigma\beta d\cH^1 \\
&\ge \tilde \cF(E,u)  - \int_\Sigma\beta d\cH^1 
\ge \cF'(A,u).  
\end{align*}
\end{proof} 

\section{Existence}\label{sec:min_config}

In this section we prove  Theorems \ref{teo:existence} and 
\ref{teo:existence_thin_void}.


\begin{proof}[Proof of Theorem \ref{teo:existence}]
We start by showing the existence of solutions of
problems \eqref{minimum_prob_1} and \eqref{minimum_prob_2}.

For the constrained minimum problem, 
let $\{(A_k,u_k)\}\subset \admissible_m$ be arbitrary minimizing sequence 
such that 
$$
\sup\limits_{k\ge1} \cF(A_k,u_k) <\infty.
$$
By Theorem \ref{teo:compactness_Ym}, there exist $(A,u)\in \admissible_m$, 
a not relabelled subsequence $\{(A_k,u_k)\}$ and a sequence $\{D_k,v_k\}\subset\admissible_m$ 
such that $(D_k,v_k)\overset{\tau_\admissible}{\to} (A,u)$ and $|B_k|=|A_k|=\fm$ 
and 
$$
\liminf \limits_{k\to \infty}  \cF(A_k,u_k) \ge \liminf \limits_{k\to \infty}  \cF(D_k,v_k). 
$$
By Lemma \ref{lem:f_convergence} (b), $D_n\to A$ in $L^1(\R^2)$
so that $|A| = \fm.$ Now by Theorem \ref{teo:lsemico_film}  
$$
\inf\limits_{(V,v)\in \admissible_m,\,\,|V| = \fm} \cF(V,v) = 
\liminf\limits_{k\to\infty} \cF(D_k,v_k)\ge \cF(A,u) 
$$ 
so that $(A,u)$ is a minimizer.
The case of the unconstrained problem is analogous.

Now we prove \eqref{zur_tenglik}. Observe that in general 
\begin{equation}\label{compar_CP_UP}
\inf\limits_{(A,u)\in \admissible} \cF^\lambda(A,u)\le 
\inf\limits_{(A,u)\in \admissible,\,\,|A| = \fm} \cF(A,u) 
\end{equation}
and the same inequality still holds if we replace $\admissible$
with $\admissible_m.$ Moreover, any solution $(A,u)\in \admissible_m$
of \eqref{minimum_prob_2} satisfying $|A| = \fm$  solves also \eqref{minimum_prob_1}. By Proposition \ref{prop:fusco_extension},
there exists a universal constant $\lambda_0>0$  with the following property: $(A,u)\in \admissible_m$ is a solution of \eqref{minimum_prob_1} if and only if it solves  \eqref{minimum_prob_2} for some (and hence for all) $\lambda\ge \lambda_0$. Thus,
\begin{equation}\label{compar_CP_UP2}
\inf\limits_{(A,u)\in \admissible_m} \cF^\lambda(A,u)=
\inf\limits_{(A,u)\in \admissible_m,\,\,|A| = \fm} \cF(A,u)  
\end{equation}
for any $m\ge1$ and $\lambda\ge\lambda_0.$ %
Since $\admissible_m\subset \admissible_{m+1}\subset \admissible,$ 
the map 
$$
m\in\N \mapsto \inf\limits_{(A,u)\in \admissible_m,\,\,|A|=\fm} \cF(A,u)
$$
is nonincreasing, 
and  
$$
\inf\limits_{(A,u)\in \admissible,\,\,|A|=\fm} \cF(A,u) 
\le \inf\limits_{(A,u)\in \admissible_m,\,\,|A|=\fm} \cF(A,u), 
$$
so that 
\begin{equation}\label{zur_tenglikning_chapi}
\inf\limits_{(A,u)\in \admissible,\,\,|A|=\fm} \cF(A,u) 
\le\lim\limits_{m\to\infty} \inf\limits_{(A,u)\in 
\admissible_m,\,\,|A|=\fm} \cF(A,u).  
\end{equation}
In view of \eqref{compar_CP_UP} and \eqref{zur_tenglikning_chapi} 
to conclude the proof of \eqref{zur_tenglik}
it suffices to show that for any $\epsilon\in(0,1)$
and $\lambda>\lambda_0,$
there exist $n\ge1$ and $(E,v)\in 
\admissible_n$ such that 
\begin{equation}\label{qioajdeofla}
\inf\limits_{(A,u)\in \admissible} \cF^\lambda(A,u) +\epsilon\ge 
\cF^\lambda(E,v)
\end{equation}
Indeed, 
by \eqref{qioajdeofla} and \eqref{compar_CP_UP2}, given $\epsilon\in(0,1)$  
\begin{align*}
\inf\limits_{(A,u)\in \admissible} \cF^\lambda(A,u) +\epsilon\ge  &
\cF^\lambda(E,v)\\ 
\ge&
\inf\limits_{(A,u)\in \admissible_n} 
\cF^\lambda(A,u)\\
= &\inf\limits_{(A,u)\in \admissible_n,\,|A|=\fm} 
\cF(A,u)\\
\ge & \lim\limits_{m\to\infty} \inf\limits_{(A,u)\in \admissible_m,\,\,|A|=\fm} 
\cF(A,u). 
\end{align*}   
Now letting $\epsilon\to0$ and using \eqref{compar_CP_UP2} and 
\eqref{zur_tenglikning_chapi} 
we get \eqref{zur_tenglik}.

We construct $(E,v)\in \admissible_n$ satisfying \eqref{qioajdeofla} as follows.
Fix $\epsilon\in(0,1)$ and $\lambda>\lambda_0,$ and choose 
$(A,u)\in\admissible$ such that 
\begin{equation}\label{asdsa}
\inf  \cF^\lambda +\frac{\epsilon}{4} > \cF^\lambda(A,u). 
\end{equation}
Notice that:
\begin{itemize}
\item[--] removing the exterior filaments decreases the energy, i.e., 
$\cF(A,u) \ge \cF(\Int{A}, u),$ thus, we assume that $A=\Int{A}$ so that $A$ is open;\\

\item[--] let $\{A_i\}_{i\in I}$ be all open connected components of $A.$ Since
$$
\cF(A,u) = \sum\limits_{i\in I} \cF(A_j, u),
$$
by the finiteness of $\cF(A,u)$ and $|A|,$ we can choose a finite set $I'\subset I$ such that 
$$
|A| \le \sum\limits_{j\in I'} \red|A_j| \black + \frac{\epsilon}{8\lambda}.
$$
Thus, setting $A':=\bigcup_{j\in I'} A_j$ and $u':=u\big|_{A'}$ we get that $(A',u')\in \admissible$ and
\begin{equation}\label{asdsa1}
\cF^\lambda(A,u) + \frac{\epsilon}{8} \ge \cF(A',u');
\end{equation}

\item[--] let $\{F_j\}_{j\in J}$  be all open connected components of $\Omega\setminus \cl{A'}.$ Since $\p F_j\subset \p A'\cup \p\Omega$ and $|\Omega| + \cH^1(\p A') + \cH^1(\p\Omega)<\infty,$ there exists a finite set $J'\subset J$ such that 
$$
\sum\limits_{j\in J\setminus J'} \cS(F_j;u_0) <\frac{\epsilon}{16} 
$$
and 
$$
\sum\limits_{j\in J\setminus J'} |F_j| <\frac{\epsilon}{16\lambda}. 
$$
Hence, setting $A'':= A\cup \bigcup_{j\in J\setminus J'} F_j$ and $u'':=u'\chi_{A'} +u_0\chi_{\bigcup_{j\in J\setminus J'} F_j}$ we get that $(A',u')\in \admissible$ and
\begin{equation}\label{asdsa2}
\cF^\lambda(A',u')+\frac{\epsilon}{4} \ge \cF^\lambda(A'',u'').
\end{equation}
Notice that for $j\in J\setminus J',$ the set $\p A'\cap \p F_j$ becomes the internal crack for $A'',$ and there is no elastic energy contribution in $F_j;$
\end{itemize}
see  Figure \ref{step1existence}. 
\begin{figure}[htp] 
\begin{center}
\includegraphics[scale=0.3]{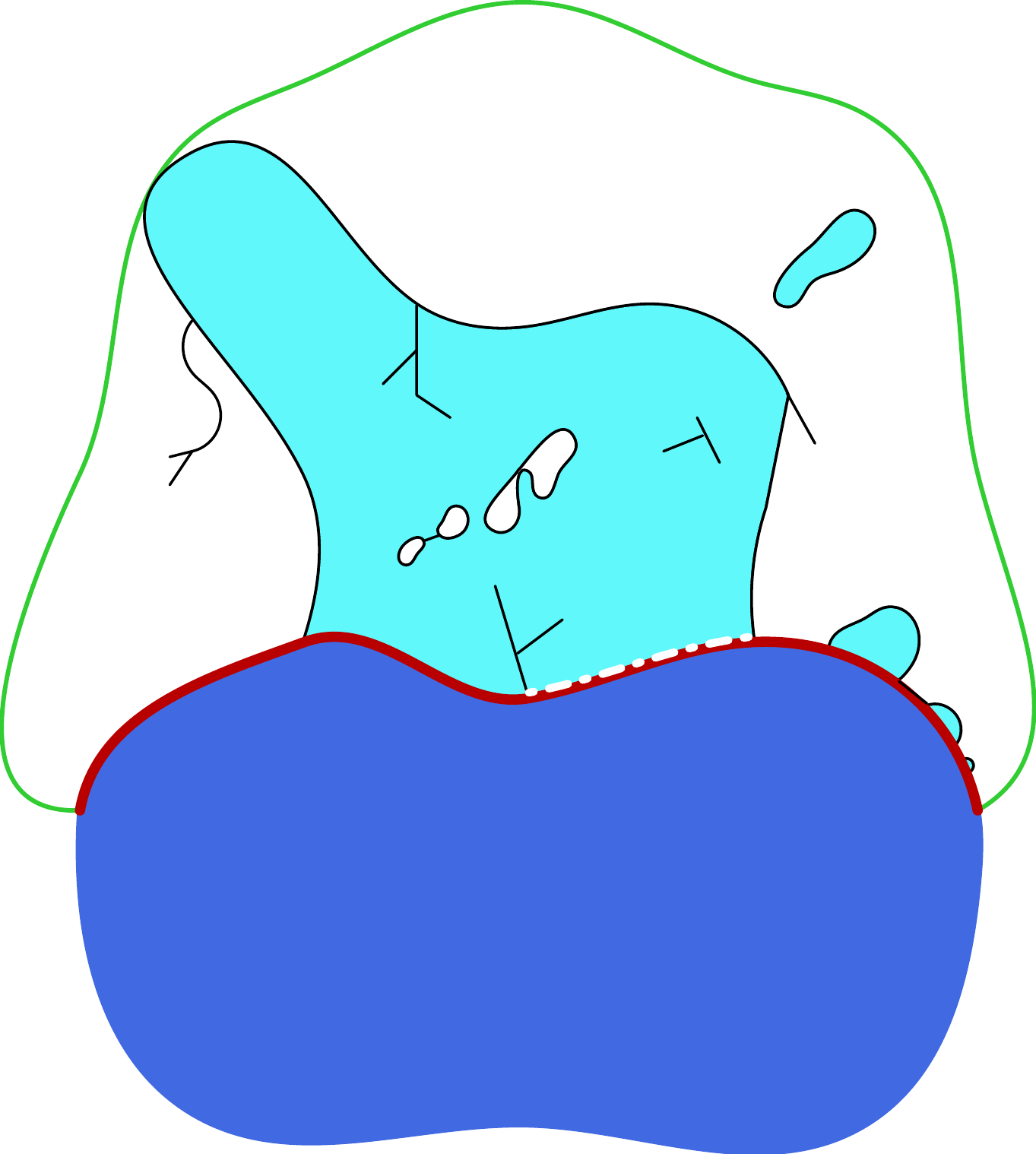} 
\includegraphics[scale=0.3]{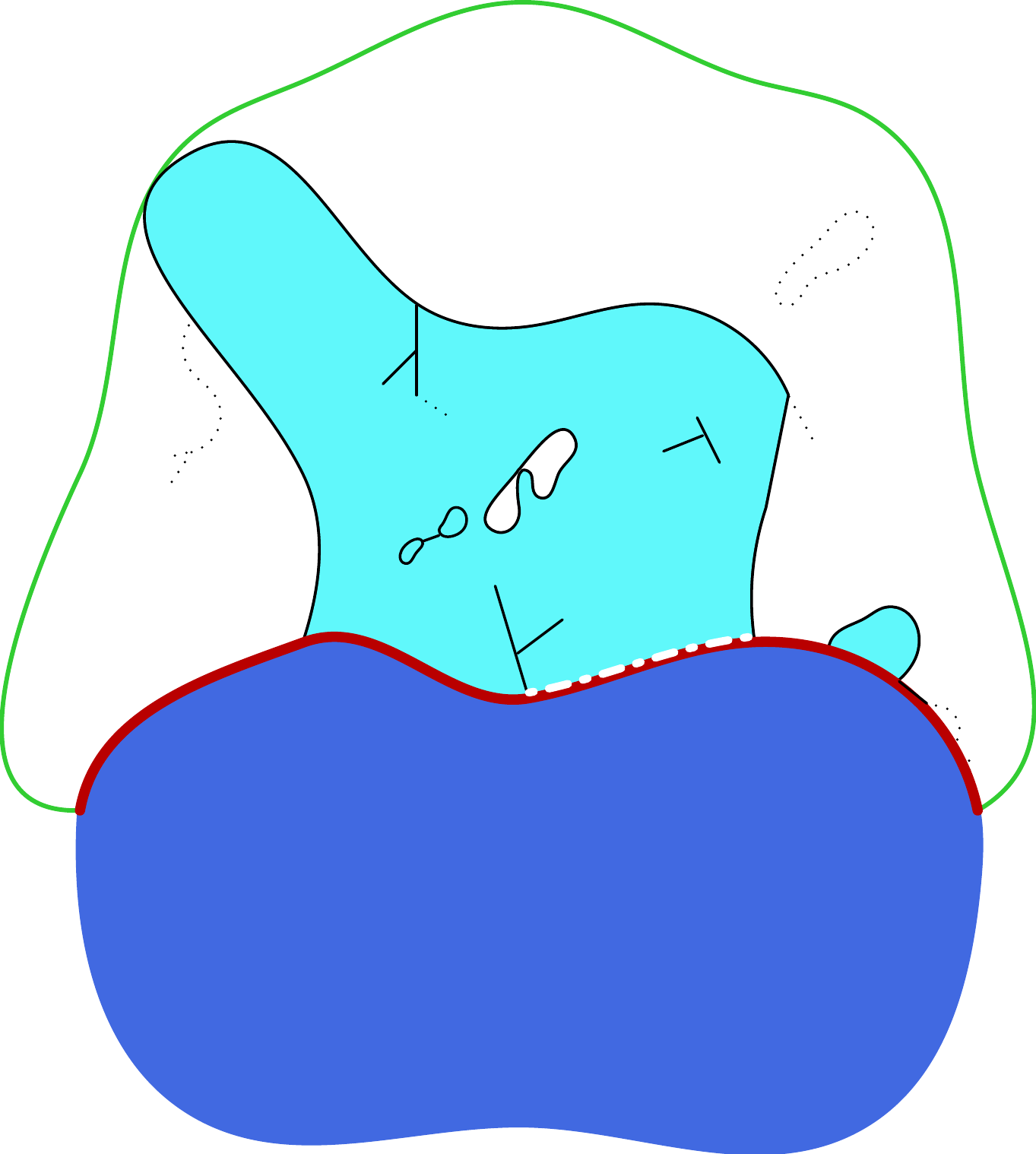} 
\includegraphics[scale=0.3]{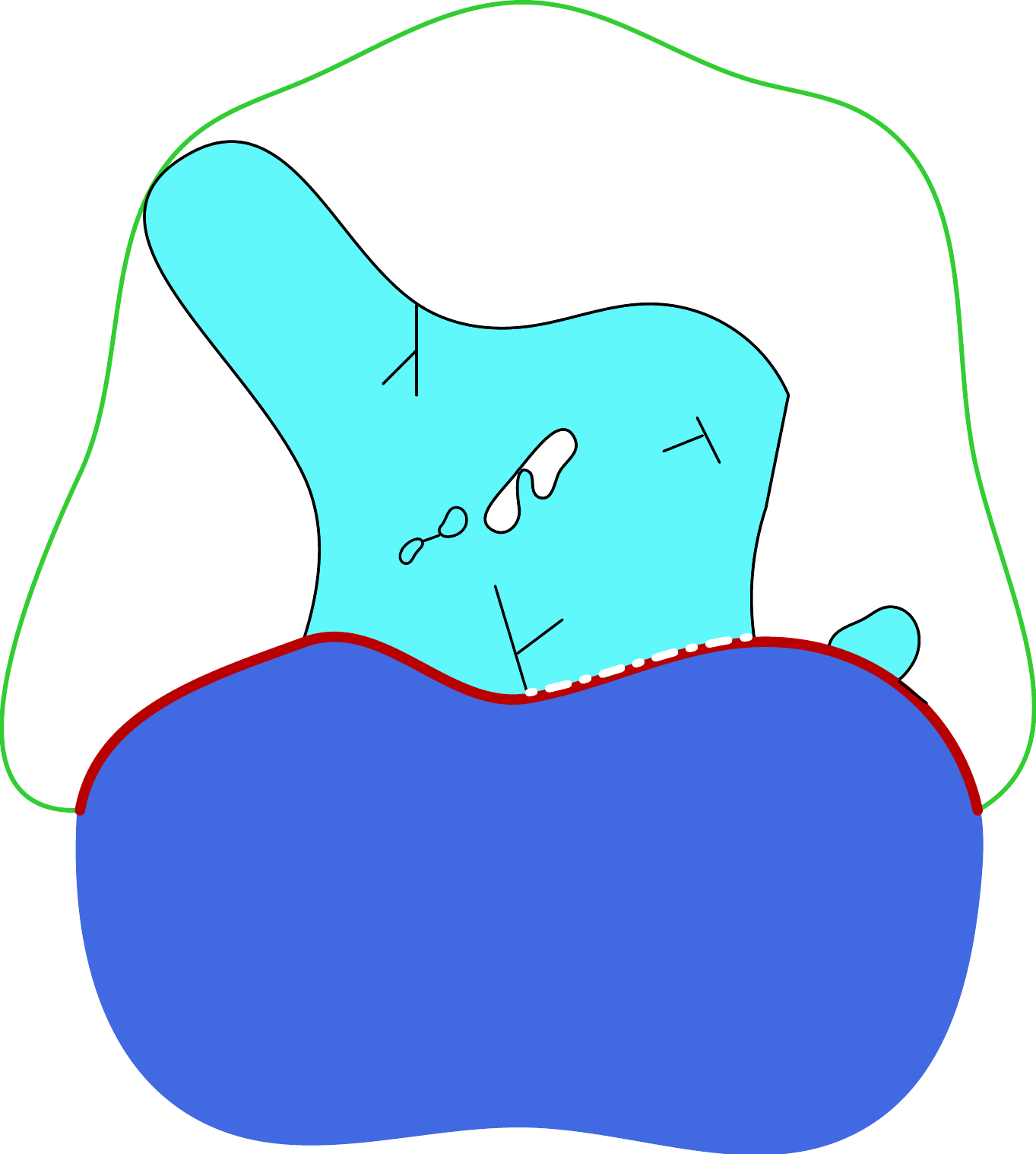} 
\caption{We pass from the set $A$ represented on the 
left  to the set $A''$ on the right by eliminating the external filaments, removing sufficiently small connected components of $A$ 
and filling in sufficiently small holes. }
\label{step1existence}
\end{center}
\end{figure}   

Hence, $A''$ is a union of finitely many connected open sets with finitely many ``holes'' inside so that $\p \cl{A''} = \overline{\p^*A''}$ consists of finitely many connected sets with finite length. 
Moreover, by \eqref{asdsa}, \eqref{asdsa1} and \eqref{asdsa2},
\begin{equation}\label{asdsa3}
\inf \cF^\lambda+\frac{\epsilon}{2} \ge \cF^\lambda(A'',u'').
\end{equation}

In view of \eqref{asdsa} and  \eqref{asdsa3} it remains to show that there exists $m\ge1$ and $(E,v)\in\admissible_m$ such that 
\begin{equation}\label{step2_sgsgs}
\cF^\lambda(A'',u'') +  \frac{\epsilon}{2} > \cF^\lambda(E,v). 
\end{equation}

Let  $G:=\Int{\overline{ A'' }}$ so that $G$ is open and  $\p G =\overline{\p^*G}.$ Since $\Sigma$ is a $1$-dimensional Lipschitz manifold, by the outer regularity of $\cH^1\res\Sigma$ there exists a finite union $I$ of subintervals of $\Sigma$ such that 
$$
J_{ u'' }\subseteq I
$$
and 
\begin{equation}\label{shjsjh}
\cH^1( I \setminus J_{ u'' }) < \frac{\epsilon}{64c_2}. 
\end{equation}
Since $\p \Omega$ is Lipschitz and $\varphi$, there exists a Lipschitz open set $V\subset\Omega$ such that  $\p V\cap \Sigma = I$ and 
\begin{equation}\label{jump_ni_gumdon}
\int_{\Omega\cap \p^* V} \varphi(x,\nu_U)d\cH^1 \le  \int_{I} \varphi(x,\nu_\Sigma)d\cH^1  + \frac{\epsilon}{64} 
\end{equation}
and 
\begin{equation}\label{volume_ni_gumdon}
|V| < \frac{\epsilon}{32\lambda}; 
\end{equation}
since $\varphi$ is uniformly continuous, basically, $V$ is obtained slightly translating $I$ inside $A.$

Let us consider $(B,w)\in\admissible$ with $B:= A'' \setminus V$ and $w= u'' \chi_{ A'' \setminus V}^{}.$ Since $ B \subset A'',$
$$
\cW( B , w ) \le \cW( A'' , u'' );
$$
since $J_{ w }=\emptyset,$  by \eqref{shjsjh} and \eqref{jump_ni_gumdon} 
\begin{equation}\label{tragefs}
\cS( B , w ) \le \cS( A'' , u'' ) +\frac{\epsilon}{16}, 
\end{equation}
and by \eqref{volume_ni_gumdon} 
$$
\lambda\big|| B | - \fm\big| \le \lambda\big|| A'' | - \fm\big| +\frac{\epsilon}{16}.
$$
Thus,
\begin{equation}\label{hauezrjk}
\cF^\lambda( B , w ) \le \cF^\lambda( A'' , u'' ) +\frac{\epsilon}{8}. 
\end{equation}
Let $\tilde w\in GSBD^2(\Int{\cl{ B \cup \substrate}};\R^2)$ be such that $\tilde w= w $ a.e. in $\cl{ B \cup \substrate}.$ Notice that $\Sigma\cap J_{ \tilde w }=\emptyset$ and  $J_{ \tilde w }\subseteq { B }^{(1)}\cap \p  B .$  Perturbing approximate continuity points of $ w $ along $ B ^{(1)}\cap \p  B $ (as has been done in the proof of Theorem \ref{teo:lsemico_film}), we may suppose that $ B ^{(1)}\cap \p  B $ is a jump set for $ \tilde w .$ Hence,  using the Vitali class of covering squares for $J_{ \tilde w }$ contained in $\Omega$ in the proof of \cite[Theorem 1.1]{ChC:2017arma}  we find $\tilde v\in SBV^2(\Ins{ B )};\R^2)\cap W^{1,\infty}(\Ins{ B };\R^2)$ such that $J_{\tilde v}$ is contained in a union of finitely many closed connected curves in $\cl{ B }$ (see \cite[pp. 1353 and 1359]{ChC:2017arma}) and
\begin{subequations}\label{chambolle_approximate}
\begin{align}
&\int_{ B  \cup \substrate}|\str{\tilde v} -\str{ \tilde w  }|^2dx <\frac{\epsilon^2}
{512(\cW( B , \tilde w ) + 1)(\|\C\|_\infty + 1)},
\label{elastic_yaqinlar}\\[1mm]
&\cH^1(J_{\tilde v}\Delta J_{ \tilde w }) <\frac{\epsilon}{32c_2}.
\label{surface_yaqinlar}
\end{align}
\end{subequations}
Notice that we do not need to control the boundary trace of $\tilde w$  that's why we can use the approximation result \cite[Theorem 1.1]{ChC:2017arma} only inside $B\cup\Sigma\cup  \substrate.$ Moreover, since $J_{\tilde w}\subset \Int{\cl{B}}$ and we use Vitali class of covering cubes only inside $\Omega$  by the formula \cite[page 1359]{ChC:2017arma} for the jump of the approximating sequence, it follows that $J_{\tilde v}\subset \cl{B}.$ In particular, ${\tilde v}\in H^1(\substrate;\R^2).$   
 
By the convexity of the elastic energy and  
the Cauchy-Schwarz inequality for nonnegative quadratic forms,  
\begin{align}\label{adhsfabjhb}
&\int_{ B \cup \substrate} W(x,\str{\tilde v} - E_0)dx \le   \int_{ B \cup \substrate} W(x,\str{  \tilde w }- E_0)dx\no \\
&+2\int_{ B \cup \substrate} \no \C(x)[\str{\tilde v}- E_0]:[\str{\tilde v} - \str{ \tilde w  }]dx\\
&\no \le  \int_{ B \cup \substrate} W(x,\str{ \tilde w }- E_0)dx\\
& +2\sqrt{\int_{ B \cup \substrate} W(x,\str{\tilde v}- E_0)dx}
\sqrt{\int_{ B \cup \substrate} W(x,\str{\tilde v} - \str{ \tilde w  })dx}.
\end{align}
Since  
$$
\int_{ B \cup \substrate} W(x,\str{\tilde v} - \str{ \tilde w  })dx\le 
\|\C\|_\infty \int_{ B \cup \substrate} |\str{\tilde v} - \str{ \tilde w  }|^2dx,
$$
and 
$$
\begin{aligned}
&\int_{ B \cup \substrate}   W(x,\str{\tilde v}- E_0)dx \\
& \le  
2\int_{ B \cup \substrate} W(x,\str{\tilde v}- \str{ \tilde w  })dx  + 
2\int_{ B  \cup \substrate} W(x,\str{ \tilde w  }-E_0)dx\\
&\le   2\|\C\|_\infty\int_{ B \cup \substrate} |\str{\tilde v}- \str{ \tilde w  }|^2dx  +2 
\cW( B , \tilde w  ) 
\le  2 \cW( B , \tilde w  )+ 2, 
\end{aligned}
$$
by \eqref{elastic_yaqinlar} and \eqref{adhsfabjhb},
\begin{equation}\label{nakhalss}
\int_{ B  \cup \substrate} W(x,\str{\tilde v}- E_0)dx \le  
\int_{ B   \cup \substrate} W(x,\str{ \tilde w  } - E_0)dx + \frac{\epsilon}{4}. 
\end{equation}
As $J_v$ is contained in at most finitely many closed $C^1$-curves,
we can find finitely many arcs of those curves 
whose union $\Gamma\subset\cl{ B }$ 
still contains $J_{\tilde v}$ and satisfies
\begin{equation}\label{gamma_ksjs}
\cH^1(\Gamma\setminus J_{\tilde v})<\frac{\epsilon}{32c_2}. 
\end{equation}

Set $E: = \Int{\cl{ B }}\setminus \Gamma$ and $v:=\tilde v\big|_{E}.$ We show that $(E,v)$ satisfies \eqref{step2_sgsgs}. Note that $J_v\cap (E\cup\substrate)=\emptyset,$  thus, $v\in H_\loc^1(E\cup\substrate;\R^2)\cap GSBD^2(\Ins{E};\R^2).$  Moreover, by construction,  $\cl{\p^*A''},$ $\Gamma$  and $\p V$ consist  of finitely many connected components, therefore, there exists $m\ge1$ such that $(E,v)\in \admissible_m.$ Notice that by the definition of $E,$
\begin{equation}\label{volume_constanat}
|E| = |\Int{\cl{ B }}| = | B |,  
\end{equation}
by the definition of $v,$ $\tilde w$ and \eqref{nakhalss},
\begin{equation}\label{elasddd}
\cW(E, v) \le  \cW( B , \tilde w  )  + \frac{\epsilon}{4}=\cW( B, w  )  + \frac{\epsilon}{4},  
\end{equation}
and by  $\p^*E=\p^*B,$  \eqref{gamma_ksjs} and \eqref{surface_yaqinlar} as well as \eqref{finsler_norm} and \eqref{tragefs},
\begin{align}\label{surfaceeess}
 \cS(E, v)  \le &    \cS( B , w ) + \frac{\epsilon}{8}.
\end{align}
From \eqref{volume_constanat}-\eqref{surfaceeess} we get 
\begin{equation*}
\cF^\lambda(E,v) \le \cF^\lambda(B,w) +\frac{3\epsilon}{8}. 
\end{equation*}
Combining this with \eqref{hauezrjk} we obtain \eqref{step2_sgsgs}.
\end{proof}

\begin{proof}[Proof of Theorem \ref{teo:existence_thin_void}]
In view of  Proposition \ref{prop:lsemico_film_voids} the assertion 
follows from the direct methods of the Calculus of Variations.
\end{proof}

\appendix

\section{ }\label{sec:kuratow}
\addtocontents{toc}{\protect\setcounter{tocdepth}{0}}

In this section we recall some results from the literature for the 
reader's convenience.

\subsection{Kuratowski convergence}
Let $\{E_k\}$ be a  sequence of subsets of $\R^2.$ 
A set $E\subset\R^2$ is the \emph{$\cK$-lower limit} of $\{E_k\}$ if for every 
$x\in E$ and $\rho>0$ there exists $n>0$ such that 
$B_\rho(x)\cap E_k \ne \emptyset$ for all $k>n.$ A set 
$E\subset\R^2$ is the \emph{$\cK$-upper limit} of $\{E_k\}$ if for every 
$x\in E$ and $\rho>0$ and $n\in\N$ there exists $k>n$ such that
$B_\rho(x)\cap E_k \ne \emptyset.$ 

The $\cK$-lower and $\cK$-upper limits of $\{E_k\}$ are always exist and 
respectively denoted as
$$
\text{$\cK$-}\liminf\limits_{k\to\infty} E_k 
\qquad
\text{and}
\qquad
\text{$\cK$-}\limsup\limits_{k\to\infty} E_k.
$$
It is clear that both sets are closed sets and 
\begin{equation*}
\text{$\cK$-}\liminf\limits_{k\to\infty} E_k 
\subseteq 
\text{$\cK$-}\limsup\limits_{k\to\infty} E_k;  
\end{equation*}
in case of equality, we say $E_k$ converges to 
$E = \text{$\cK$-}\liminf\limits_{k\to\infty} E_k 
=
\text{$\cK$-}\limsup\limits_{k\to\infty} E_k$ in 
the Kuratowski sense and write
$$
E= \text{$\cK$-}\lim\limits_{k\to\infty} E_k\qquad\text{or} \qquad 
E_k\overset{\cK}{\to}E.
$$
Observe that 
by the definition of $\cK$-convergence, 
$E_k$ and $\cl{E_k}$ have the same $\cK$-upper
and $\cK$-lower limits. Moreover,
Kuratowski limit is always unique. 

\begin{proposition}\label{prop:equivalent_kuratowski}
The following assertions are equivalent:
\begin{itemize}
\item[\rm (a)]  $E_k\overset{\cK}\to E;$

\item[\rm (b)] if $x_k\in E_k$ converges to some $x\in \R^2,$  
then $x\in E,$ and for every $x\in E$ there exists
a subsequence $x_{n_k}\in  E_{n_k}$ converging to $x;$

\item[\rm (c)] $\dist(\cdot,E_k) \to \dist(\cdot, E)$ locally uniformly in $\R^2;$ 

\item[\rm (d)] if, in addition, $\{E_k\}$ is uniformly bounded, 
$E_k\to E$ with respect to Hausdorff distance $\dist_{\cH},$ where
\begin{equation}\label{haus_dista}
\dist_{\cH}(A,B):=
\begin{cases}
0 & \text{if $A=B=\emptyset,$}\\
\max\Big\{\sup\limits_{x\in A} \dist(x, B), \,
\sup\limits_{x\in B} \dist(x, A)\Big\} &\text{otherwise}.
\end{cases}
\end{equation}

\end{itemize}
\end{proposition}

\subsection{Rectifiability in $\R^2$} \label{subsec:rectifiable}
Below we recall some important regularity properties of 
compact connected subsets of finite $\cH^1$-measure of $\R^2$
most of them are taken from \cite[Chapters 2 and 3]{Fa:1985}.

The image $\Gamma$
of a continuous injection $\gamma:[a,b]\to \R^2$ is called 
{\it curve} (or {\it  Jordan curve}), and $\gamma$ is the parametrization of $\Gamma.$ 
Clearly,  any curve is compact and connected set, 
hence it is $\cH^1$-measurable. 
The length of a curve $\Gamma$ is defined as 
$$
\sup s(\gamma, P),
$$
where $P=\{t_0,t_1,\ldots,t_N\}$ is a partition of $[a,b],$  i.e.  
$a=t_0<t_1<\ldots<t_N=b,$ 
$$
s(\gamma, P):= \sum\limits_{i=1}^N |\gamma(t_{i-1}) - \gamma(t_i)|,
$$
and $\sup$ is taken over all partitions $P$ of $[a,b].$
By \cite[Lemma 3.2]{Fa:1985}, the length of curve $\Gamma$
is equal to $\cH^1(\Gamma).$

Any curve $\Gamma$ with finite length
admits so-called {\it arclength} parametrization in $[0,\cH^1(\Gamma)],$ 
which is a Lipschitz parametrization $\gamma_o$ with Lipschitz constant $1.$
%
Hence, by the Rademacher Theorem
\cite[Theorem 2.14]{AFP:2000}  it is 
differentiable at a.e.\ $\ell\in (0,\cH^1(\Gamma))$
and $|\dot{\gamma_o}(\ell)|\le 1.$
Hence $\Gamma$ has an (approximate)
tangent line at $\cH^1$-a.e.\ $x\in\Gamma$ and we can define the 
approximate unit normal $\nu_\Gamma(x)$ of $\Gamma$ at $\cH^1$-a.e.\ 
$x\in\Gamma.$ 

We recall the following characteristics of 
compact connected $\cH^1$-rectifiable sets 
(see \cite[Theorem 3.14]{Fa:1985} and 
\cite[Section 2]{Gi:2002}).

\begin{proposition}[\textbf{Rectifiable connected sets}]\label{prop:rectifiable_sets}
Every connected compact set $K\subset \cl{U}$ with
$\cH^1(K)<\infty$ is arcwise connected and 
countably $\cH^1$-rectifiable, 
i.e.,  
$$
K:= N\cup \bigcup\limits_{j\ge 1} \Gamma_j,
$$ 
where $N$ is a $\cH^1$-negligible set and $\Gamma_j:=\gamma_j([0,1])$  is a curve with finite length for a parametrization $\gamma_j:[0,1]\to\R^2$  such that 
$$
\gamma_j((0,1))\cap \bigcup\limits_{i=1}^{j-1}\Gamma_i = \emptyset,\qquad j\ge2.
$$
\end{proposition}

\begin{remark}[\textbf{Rectifiable curve is locally Jordan}]
\label{rem:rect_sets_boundary}
Let $\Gamma$ be a rectifiable curve. Then for $\cH^1$-a.e.\ $x\in \Gamma$
there exists $r_x>0$ such that $B_{r_x}(x)\setminus \Gamma$ has exactly 
two connected components. Indeed, suppose that there exists $x\in\Gamma$
such that $\theta^*(\Gamma,x) = \theta_*(\Gamma,x)=1$ and 
$B_r(x)\setminus\Gamma$ has at least three connected components 
for every $r>0$ such that  endpoints of $\Gamma$ lie outside
$\cl{B_r(x)}.$ Then $(B_r(x)\setminus B_{r/2}(x))\cap\Gamma$
should have three connected components and as a result 
$\cH^1((B_r(x)\setminus B_{r/2}(x))\cap\Gamma)\ge3r/2$ and
$$
1 =\lim\limits_{r\to0} \frac{\cH^1(B_r(x)\cap \Gamma)}{2r} \ge  
\lim\limits_{r\to0} \frac{\cH^1(B_{r/2}(x)\cap \Gamma)}{2r} +\frac{3}{4} = \frac54,
$$
a contradiction.
\end{remark}

\begin{proposition}[\textbf{Properties of regular points  
\cite{AFP:2000,Fa:1985,Gi:2002}}]\label{prop:tangent_line_conv}
Suppose that $K\subset\R^2$ be a connected compact set with $\cH^1(K)<\infty.$
Thus, it admits a unit (measure-theoretic) normal $\nu_K(x)$ 
at $\cH^1$-a.e.\ $x\in K;$ the map $x\mapsto \nu_K(x)$
is Borel measurable and if $L$ is any connected subset of $K$ 
then $\nu_L(x) = \nu_K(x)$ for any $x\in L$ for which 
the unit normal $\nu_K(x)$ to $K$ exists.
Moreover, $\cH^1\res \sigma_{\rho,x}(K)\wk^* \cH^1\res T_x$
and $\cl{U_{1,\nu_K}(x)}\cap \sigma_{\rho,x}(K)
\overset{\cK}{\to} \cl{U_{1,\nu_K}(x)}\cap T_x$ 
as $\rho\to0^+,$ where $T_x$ is the generalized tangent line to $K$ at $x.$  
\end{proposition}

Note that $U_{1,\nu_K}(x)$ in Proposition \ref{prop:tangent_line_conv} can be replaced by arbitrary $U_{R,\nu_K}(x)$. 

\begin{proposition}\label{prop:set_shrinks}
Let $A\in \fA$ and $x\in \p A$ be such that the measure-theoretic unit normal $\nu_A(x)$ to $\p A$ exists 
and $\cl{U_{R,\nu_A(x)}(x)}\cap \sigma_{\rho,x}(\p A)
\overset{\cK}{\to} \cl{U_{R,\nu_A(x)}(x)}\cap T_x$ for any $R>0$
as $\rho\to0^+.$ Then:
\begin{itemize}
\item[(a)] if $x\in A^{(1)}\cap \p A,$ then $\sdist(\cdot,\sigma_{\rho,x}(\p A))\to-\dist(\cdot,T_x)$ uniformly in $\overline{U_{1,\nu_A(x)}};$

\item[(b)] if $x\in A^{(0)}\cap \p A,$ then $\sdist(\cdot,\sigma_{\rho,x}(\p A))\to \dist(\cdot,T_x)$ uniformly in $\overline{U_{1,\nu_A(x)}}.$
\end{itemize}
\end{proposition}

\begin{proof}
We prove only (a); the proof of (b) is analogous. Let $x\in A^{(1)}\cap \p A$ be such that 
\begin{itemize}
\item[(x1)] $\nu_A(x)$ exists;
\item[(x2)] $\cl{U_{R,\nu_A(x)}(x)}\cap \sigma_{\rho,x}(\p A)
\overset{\cK}{\to} \cl{U_{R,\nu_A(x)}(x)}\cap T_x$ for any $R>0$
as $\rho\to0^+.$  
\end{itemize}
Without loss of generality we assume $x=0,$ $\nu_A(x)={\bf e_2}$ and $T_x=T_0$ is the $x_1$-axis. 
For shortness we write $A_\rho$ and $(\p A)_\rho$ in places of 
$\sigma_{\rho,x}(A)$ and $\sigma_{\rho,x}(\p A),$ respectively. Let $\{\rho_k\}\subset(0,1)$ be arbitrary sequence converging to $0.$ Consider $f_k:=\sdist(\cdot,(\p A)_{\rho_k})\big|_{U_4}\in W^{1,\infty}(U_4).$ For any $k\ge1,$ $f_k$ is $1$-Lipschitz and $f_k(0)=0,$ therefore, by the Arzela-Ascoli Theorem, there exist $f\in W^{1,\infty}(U_4)$ and not relabelled subsequence $\{f_k\}$ such that $f_k\to f$ uniformly in $U_4.$
By (x2) applied with $R>4$, $|f_k| = \dist(\cdot,(\p A)_{\rho_k})\to \dist(\cdot,T_0)$ uniformly in $U_4,$ therefore, $|f(x)| = \dist(x, T_0)$ for any $x\in U_4.$ 
Thus, it suffices to prove that $f\le 0.$

Assume by contradiction that $f>0$ in $U_4^+:= U_4\cap \{x_2>0\}.$
In addition, by (x2) for any $\delta\in(0,1)$ there exists $k_\delta>0$ such that  $U_4\cap (\p A)_{\rho_k}\subset T_0\times (-\delta,\delta)$ for any $k>k_\delta.$ Therefore,  
$\sdist(\cdot,(\p A)_{\rho_k})>0$ in $U_{4,\delta}^+:=U_4\cap \{x_2>\delta\},$ and hence,
$A\cap \rho_k U_{4,\delta}^+ = \emptyset,$ where $r D = \{rx:\,x\in D\}.$ 
Since $0\in A^{(1)}\cap\p A,$ this implies
$$
1=\lim\limits_{k\to\infty} \frac{|A\cap U_{4\rho_k}|}{|U_{4\rho_k}|} =
\lim\limits_{k\to\infty} \frac{|(A\cap U_{4\rho_k})\setminus \rho_kU_{4,\delta}^+|}{|U_{4\rho_k}|} < 1,
$$
a contradiction. 
 
Analogous contradiction is obtained assuming $f>0$ in $U_4\cap\{x_2<0\}.$
\end{proof}

\subsection{Minimizers of volume-constraint and unconstraint problems}

The following proposition is an  extension of \cite[Theorem 1.1]{EF:2011}.

\begin{proposition}\label{prop:fusco_extension}
Assume {\red \rm(H1)-(H3)}. Then there exists $\lambda_0>0$ (possibly depending on $c_1,$ $c_2$ and $\Omega$) with the following property: given $m\ge1,$ $(A,u)\in \admissible_m$ is a solution of \eqref{minimum_prob_1} if and only if $(A,u)$ is also a solution to \eqref{minimum_prob_2} for all $\lambda\ge \lambda_0.$ 
\end{proposition}

\begin{proof}
Note that any minimizer $(A,u)\in\admissible_m$ of $\cF^\lambda$ with $|A|=\fm$ is also minimizer of $\cF.$ Hence, it suffices to show that  there exists $\lambda_0>0$ such that any minimizer $(A,u)$ of $\cF^\lambda$ for $\lambda>\lambda_0$ satisfies  $|A|=\fm.$ 

Assume by contradiction that there exist sequences $\{m_h\}\subset\N$ and $\{\lambda_h\}\subset\R$ with $\lambda_h\to\infty$ and a sequence $(A_h,u_h)\in \admissible_{m_h}$ minimizing $\cF^{\lambda_h}$ such that  $|A_h|\ne \fm.$ Since $\Omega$ has finitely many connected components there exists an open Lipschitz set $A_0\subset\Omega$ with $|A_0|=\fm$  such that  
$\cF^{\lambda_h}(A_h,u_h) \le \cF^{\lambda_h}(A_0,u_0)=\cF(A_0,u_0)$
for all large $h.$
Thus, by \eqref{finsler_norm} and \eqref{hyp:bound_anis},
\begin{equation}\label{gsrtaer}
\sup\limits_h \cH^1(\p A_h) \le \Lambda:=\frac{\cF(A_0,u_0) + c_2\cH^1(\Sigma) + \cH^1(\p\Omega)}{c_1}, 
\end{equation}
and 
$\lambda_h||A_h| - \fm|\le \cF(A_0,u_0) + c_2\cH^1(\Sigma)$ for any $h.$ This implies $|E_h| \to \fm$ as $h\to\infty.$ By compactness, there exists a finite perimeter set $A\subset\Omega$ and a not relabelled subsequence such that $\chi_{A_h} \to \chi_A$ a.e.\ in $\R^2.$ In particular, $|A|=\fm.$  

We suppose that $|A_h|<\fm$ for all $h;$ the case $|A_h|>\fm$ is treated analogously. As in the proof of \cite[Theorem 1.1]{EF:2011} given $\epsilon\in(0,\frac18)$, there exist small  $r>0$ and $x_r\in\Omega$ such that  $B_r(x)\strictlyincluded\Omega$ and 
$$
|A\cap B_{r/2}(x_r)|<\epsilon r^2,\qquad |A\cap B_r(x_r)|>\frac{\pi r^2}{16}.
$$
For shortness, we suppose that $x_r=0$ we write $B_r:=B_r(x_r).$  Since $A_h\to A$ in $L^1(\R^2),$ for all large $h,$
\begin{equation}\label{efredefe}
|A_h\cap B_{r/2}|<\epsilon r^2,\qquad |A_h\cap B_r|>\frac{\pi r^2}{16}. 
\end{equation}
Let $\Phi:\R^2\to \R^2$ be the bi-Lipschitz map which takes $B_r$ into $B_r$ defined as 
$$
\Phi(x):= 
\begin{cases}
(1 - 3\sigma)x, & |x|<\frac r2,\\
x+ \sigma\Big(1 - \frac{r^2}{|x|^2}\Big)x, & \frac{r}{2} \le x < r,\\
x, & |x|\ge r 
\end{cases}
$$
for some $\sigma\in(0,\frac{1}{16}).$ 
Recall from \cite[pp. 420-422]{EF:2011} that 
the Jacobian $J\Phi$ of $\Phi$ satisfies 
$$
J\Phi(y) \ge 1+ \sigma \qquad y\in B_r\setminus B_{r/2},
$$
and 
$$
J\Phi(y) \le 1+ 8\sigma \qquad y\in B_r,
$$
and the tangential Jacobian $J_1T_x$ of $\Phi$ on the tangent space $T_x$ of $\p A_h$ satisfies 
\begin{equation}\label{tang_jacob}
J_1T_x \le 1+5\sigma,\qquad x\in B_r \cap \p A_h.  
\end{equation}
Set 
$$
E_h:= \Phi(A_h)\setminus \p B_r,\qquad v_h: =u_h\chi_{A_h}\setminus B_r + u_0\chi_{E_h\cap B_r}.
$$
Note that $|E_h|<\fm$ and since the bi-Lipschitz maps do not increase the number of connected components, $(E_h,v_h)\in \admissible_m.$ Let us estimate
\begin{align*}
\cF^{\lambda_h}(A_h,u_h) -& \cF^{\lambda_h}(E_h,v_h)\\
= & 
\int_{\overline{B_r}\cap \p A_h} \theta_{A_h}(x)\,\phi(x,\nu_{A_h})d\cH^1 - \int_{\overline{B_r}\cap \p E_h} \theta_{E_h}(x) \phi(x,\nu_{E_h})d\cH^1\\
& + \int_{B_r\cap A_h} W(x,\str{u_h} - E_0)dx - \int_{B_r\cap E_h} W(x,\str{v_h} - E_0)dx \\
& + \lambda_h \Big(|E_h| - |A_h|\Big)
:= I_1 + I_2 + I_3,
\end{align*}
where $\theta_F(x)$ is $1$ for $\cH^1$-a.e.\ on $\p^*F,$ 
is $2$ for $\cH^1$-a.e.\ on $F^{(1)}\cup F^{(0)}\cap \p F$ and 
is $0$ otherwise.
By the choice of $v_h,$
$I_2\ge0.$ Moreover, by \eqref{tang_jacob} and the area formula as well as from \eqref{finsler_norm}, \eqref{gsrtaer} and equality $\theta_{E_h}(\Phi(y)) = \theta_{A_h}(y)$ for $\cH^1$-a.e.\ $y\in \p A_h$, 
\begin{align*}
 \int_{B_r\cap \p E_h} &\theta_{E_h}(x) \phi(x,\nu_{E_h})d\cH^1 =   
  \int_{B_r\cap \p A_h} \theta_{A_h}(\Phi(y))\, \phi(\Phi(y),\nu_{A_h})J_1 T_y\,d\cH^1\\
&\le  2c_2(1+5\sigma) \cH^1(B_r\cap \p A_h) 
  \le 2c_2(1+5\sigma)\Lambda.
\end{align*}
Moreover, by \eqref{finsler_norm},
$$
 \int_{\p B_r\cap \p E_h} \theta_{E_h}(x) \phi(x,\nu_{E_h})d\cH^1 \le 2c_2\cH^1(\p B_r) \le 4\pi c_2r,
$$
thus, 
$$
I_1\ge - 2c_2(1+5\sigma)\Lambda - 4\pi c_2r.
$$
Finally, repeating the same arguments of Step 4  in the proof of \cite[Theorem 1.1]{EF:2011}, we obtain
\begin{align*}
I_3 \ge \lambda_h \sigma r^2 (1 - 7\epsilon),
\end{align*}
thus, 
\begin{equation}\label{asfgsa}
\cF^{\lambda_h}(A_h,u_h) -  \cF^{\lambda_h}(E_h,v_h) \ge  \lambda_h \sigma r^2 (1 - 7\epsilon) - 2c_2(1+5\sigma)\Lambda-4\pi c_2r. 
\end{equation}
Since the dependence of the right-side of \eqref{asfgsa} on $h$ is only through $\lambda_h,$ for sufficiently large $h$ we have 
$\cF^{\lambda_h}(A_h,u_h) >\cF^{\lambda_h}(E_h,v_h),$ which contradicts to the minimality of $(A_h,u_h).$
\end{proof}

\section*{Acknowledgments}

Sh. Kholmatov acknowledges support from the Austrian Science Fund (FWF) 
project M~2571-N32. P. Piovano acknowledges support from the Vienna 
Science and Technology Fund (WWTF), the City of Vienna, and Berndorf Privatstiftung through Project MA16-005, from the Austrian Science Fund (FWF) through project P~29681\red, and from BMBWF through the OeAD-WTZ project HR 08/2020.   
\red  Moreover, the authors are grateful to the anonymous referees for their careful reading and valuable comments. \black 


\end{document}